\documentclass[11pt]{article}

\usepackage{amsmath} 
\PassOptionsToPackage{hyphens}{url}\usepackage[colorlinks=true,bookmarks=false,linkcolor=blue,urlcolor=blue,citecolor=blue,breaklinks=true]{hyperref}
\usepackage{xurl}
\usepackage[T1]{fontenc} 
\usepackage{breakcites} 

\usepackage[numbers,sort&compress]{natbib}
\usepackage{doi}
\usepackage{graphicx}
\usepackage{amssymb} 
\usepackage{amsfonts} 
\usepackage{amsthm} 
\usepackage[shortlabels]{enumitem} 
\usepackage{comment} 
\usepackage{bbm}
\usepackage{mathtools}
\usepackage{algorithm}
\usepackage{algpseudocode}
\usepackage[margin=1in]{geometry} 
\usepackage{subcaption}
\usepackage{multirow} 
\usepackage{pifont} 
\usepackage{booktabs} 
\usepackage{mathrsfs}
\usepackage{tikz-cd}
\usetikzlibrary{calc, positioning}

\usepackage{multicol}

\makeatletter
\renewcommand{\paragraph}{%
  \@startsection{paragraph}{4}%
  {\z@}{1ex \@plus 1ex \@minus .2ex}{-.5em}%
  {\normalfont\normalsize\bfseries}%
}
\makeatother

\newtheorem{theorem}{Theorem}[section]
\newtheorem{lemma}[theorem]{Lemma}
\newtheorem{proposition}[theorem]{Proposition}
\newtheorem{corollary}[theorem]{Corollary}
\newtheorem{example}[theorem]{Example}
\newtheorem{definition}[theorem]{Definition}
\newtheorem{remark}[theorem]{Remark}

\newcommand{\be}{\begin{equation}}
\newcommand{\ee}{\end{equation}}
\newcommand{\bthm}{\begin{theorem}}
\newcommand{\ethm}{\end{theorem}}
\newcommand{\blem}{\begin{lemma}}
\newcommand{\elem}{\end{lemma}}
\newcommand{\bpof}{\begin{proof}}
\newcommand{\epof}{\end{proof}}
\newcommand{\bcor}{\begin{corollary}}
\newcommand{\ecor}{\end{corollary}}
\newcommand{\bprop}{\begin{proposition}}
\newcommand{\eprop}{\end{proposition}}

\newcommand{\diag}{\mathrm{diag}}

\newcommand{\mc}[1]{\mathcal{#1}}
\newcommand{\mbb}[1]{\mathbb{#1}}

\newcommand{\mscr}[1]{\mathscr{#1}}
\newcommand{\msf}[1]{\mathsf{#1}}

\newcommand{\TODO}[1]{{\color{red}{[#1]}}}

\newcommand{\RR}{\mathbb{R}}

\newcommand{\NN}{\mathbb{N}}

\newcommand{\vct}[1]{\mathbb{#1}}  

\newcommand{\FS}{\mathrm{FinSet}}

\newcommand{\coFS}{\mathrm{FinSet}^{\mathrm{op}}}

\usepackage[margin=1in]{geometry}

\newcommand{\R}{\mathbb{R}}

\title{Any-Dimensional Polynomial Optimization \\ via de Finetti Theorems}

\author{Eitan Levin$^\dag$ and Venkat Chandrasekaran$^{\ddag}$ \thanks{Emails: \texttt{eitanl@uchicago.edu},  \texttt{venkatc@caltech.edu}} \vspace{0.1in} \\ $^\dag$ Department of Statistics\\ University of Chicago \\ Chicago, IL 60637 \vspace{0.1in}\\ $^\ddag$ Department of Computing and Mathematical Sciences\\ Department of Electrical Engineering \\ California Institute of Technology \\ Pasadena, CA 91125}

\date{October 16, 2026}
\begin{document}

\maketitle

\begin{abstract}

Polynomial optimization problems often arise in sequences indexed by dimension, and it is of interest to compute bounds on the optimal values of all problems in the sequence. 
Examples include certifying inequalities between symmetric functions or graph homomorphism densities that hold over vectors and graphs of all sizes, and computing the value of mean-field games viewed as limits of games with a growing number of players.  In this paper, we study such any-dimensional polynomial problems using the theory of representation stability, and we develop a systematic framework to produce hierarchies of improving bounds on their limiting optimal values.
Our bounds are obtained by solving finite-dimensional polynomial optimization problems (or their relaxations).
These bounds converge at explicit rates as a consequence of new de Finetti-type theorems pertaining to sequences of random arrays projecting onto each other in different ways.  The proofs of these theorems are based on applying results from probability to representations of certain categories.
We apply our framework to produce new bounds on problems arising in a number of application domains such as mean-field games, extremal graph theory, and symmetric function theory, and we illustrate our methods via numerical experiments.


\vspace{0.5cm}

\noindent \textbf{MSC classes:} 90C23, 05E05, 60B15

\vspace{0.5cm}


\end{abstract}

\tableofcontents

\section{Introduction}\label{sec:intro}



Optimization problems specified by multivariate polynomials arise commonly in control, data science, dynamical systems, and theoretical computer science.
Such polynomial optimization problems (POPs) are NP-hard to solve in general, and much research effort has focused on producing tractable bounds on their optimal values using sums of squares certificates of polynomial nonnegativity~\cite{parrilo2003semidefinite,lasserre2001global,blekherman2012semidefinite}.  In many application domains such as game theory and extremal combinatorics, POPs arise naturally as sequences indexed by dimension and it is of interest to simultaneously bound all of the optimal values in the sequence.  The next few examples illustrate the ubiquity of these problems.

\begin{example}[Inequalities in moments]\label{ex:power_means}
    Fix a compact set $\Theta\subseteq\RR^d$ and consider the sequence of POPs:
    \begin{equation}\label{eq:power_means_general}
        u_n = \inf_{(x_1,\ldots,x_n)\in\Theta^n}f\left(\frac{1}{n}\sum_{i=1}^nx_i^{\alpha_1},\ldots,\frac{1}{n}\sum_{i=1}^nx_i^{\alpha_m}\right),
    \end{equation}
    where $f$ is an $m$-variate real polynomial and $\alpha_1,\ldots,\alpha_m\in\mbb N^d$ are multi-indices. Here $x^{\alpha}=\prod_{i=1}^dx_i^{\alpha_i}$ if $x\in\RR^d$ and $\alpha\in\mbb N^d$.
    Observing that $\frac{1}{n}\sum_{i=1}^nx_i^{\alpha}=\mbb E_{X\sim\mu(x)}X^\alpha$ where $\mu(x)=\frac{1}{n}\sum_{i=1}^n\delta_{x_i}$ is a discrete measure, computing $u_{\infty}=\inf_nu_n$ amounts to proving polynomial inequalities satisfied by moments of any measure supported on $\Theta$. Such problems also arise in the context of potential mean-field games~\cite{cardaliaguet2010notes} as well as in operator theory and quantum information~\cite{klep2025sums}.
    %
\end{example}

\begin{example}[Inequalities in symmetric functions]\label{ex:power_sums}
    Symmetric functions are certain symmetric polynomials defined for any number of variables, such as power sums $s_k(x)=\sum_ix_i^k$ and elementary symmetric functions $\sum_{i_1<\ldots<i_k}x_{i_1}\cdots x_{i_k}$. These are fundamental objects of study in enumerative combinatorics, representation theory, and analysis~\cite{macdonald1998symmetric,Stanley_Fomin_1999}. In particular, inequalities between symmetric functions, which hold regardless of the number of variables, are of interest in all of these areas~\cite{Hunter_1977,PROCESI1978219,Chavez_Garcia_Hurley_2023,khare2021sign}.
    Since any symmetric function can be written as a polynomial expression $f(s_1,\ldots,s_k)$ in power sums, checking whether a given symmetric function is nonnegative in all dimensions amounts to considering the sequence of POPs
    \begin{equation*}
        u_n = \inf_{x\in\RR^n} f\left(\sum_{i=1}^nx_i,\ldots,\sum_{i=1}^nx_i^k\right),
    \end{equation*}
    and checking whether $u_{\infty}=\inf_nu_n\geq0$. 
    The problem of checking $u_{\infty}\geq0$ given a polynomial $f$ cannot always be done by expressing $f$ as a sum of squares~\cite{acevedo2024symmetric}.
\end{example}
\begin{example}[Inequalities in graph densities]\label{ex:graph_densities}
    For graphs $H,G$, the \emph{homomorphism density} $t(H;G)$ of $H$ in $G$ is the fraction of maps between their vertex sets that are graph homomorphisms, and this fraction can be written as a polynomial in the entries of the adjacency matrix of $G$~\cite{raymond2018symmetric}. Many problems in extremal combinatorics can be formulated as proving inequalities between homomorphism densities that hold for all graph sizes~\cite{lovasz2012large,brosch_thesis}.
    Formally, given $c_1,\ldots,c_m\in\RR$ and graphs $H_1,\ldots,H_m$, we consider the sequence of POPs:
    \begin{equation}\label{eq:general_graph_densities}
        u_n = \inf_{\substack{\text{graphs }G,\\ |V(G)|=n}} ~ \sum_{i=1}^m c_i t(H_i;G),
    \end{equation}
    over graphs on $n$ vertices, and ask whether $u_\infty=\inf_nu_n\geq0$.  This is equivalent to asking whether the inequality $\sum_i c_i t(H_i;G)\geq0$ holds for graphs $G$ of all sizes.  The problem of checking whether $u_{\infty}\geq0$ given $c_i,H_i$ is undecidable in general, and there are nonnegative such expressions which are not sums of squares~\cite{hatami2011undecidability}. 
\end{example}

We refer to such sequences of polynomial minimization problems indexed by dimension as \emph{any-dimensional} POPs, and in each of the preceding examples, the central question is to compute $u_{\infty}=\inf_nu_n$ or to prove that $u_{\infty}\geq0$. As Example~\ref{ex:graph_densities} illustrates, this is in general undecidable. To address this difficulty, we adopt an optimization perspective and aim to compute a lower bound $\gamma\in\RR$ on $u_{\infty}$, or equivalently, on all the optimal values $u_n$ in the sequence, using a finite-time procedure. 
Although there is a large literature on computing such bounds for fixed-dimensional POPs~\cite{blekherman2012semidefinite}, producing bounds on one of the optimal values $u_n$ in the sequence does not, in general, yield bounds on the limiting value $u_{\infty}$.  
Also, while undecidability of general any-dimensional POPs implies that there are no finite-time procedures to decide, for example, whether $u_\infty$ is nonnegative, this does not preclude the existence of such procedures for producing some finite lower bounds on $u_\infty$.

In this paper, we develop a systematic framework for producing lower bounds on the limiting optimal value $u_{\infty}$ of any-dimensional POPs $(u_n=\inf_{x\in\Omega_n}p_n(x))_{n\in\NN}$ in terms of \emph{finite-dimensional POPs} of the form $\ell_n=\inf_{x\in\Omega_n}q_n(x)$, so that 
\begin{equation}\label{eq:nestedness}
    \ell_n \leq u_{\infty}\leq u_n,\quad \textrm{for all } n\in\NN.
\end{equation}
These lower-bounding problems have the same constraint sets as in the original sequence of problems, but their cost functions $q_n$ are distinct from $p_n$.  Moreover, these lower bounds $\ell_n$ converge at explicit rates to $u_\infty$.
Either solving the finite-dimensional problems defining $\ell_n$ or applying standard convex relaxations to them yields lower bounds for $u_\infty$. 
Our lower bounds can be viewed as searching over tractable families of \emph{any-dimensional certificates of nonnegativity}, and hence providing a computational procedure for proving inequalities valid in all dimensions.  The following are a few examples of the bounds we obtain using our framework.
\begin{example}[Mean-field games]\label{ex:mfg_intro}
    Consider the sequence of problems
    \begin{equation*}
        u_n=\inf_{x\in[-1,1]^n} 5\left(\frac{1}{n}\sum\nolimits_ix_i\right)^2 - 4\left(\frac{1}{n}\sum\nolimits_ix_i^2\right)\left(\frac{1}{n}\sum\nolimits_ix_i\right) - \frac{1}{n}\sum\nolimits_ix_i, 
    \end{equation*}
    which computes the value of an $n$-player extension of the non-convex-concave two-player game studied in~\cite[Ex.~3.2]{parrilo2006games}.  
    The limiting optimal value $u_{\infty}$ computes the value of a mean-field limit of these games, and we seek lower bounds on this value.
    Our framework yields the lower bounds for $n\geq 3$
    \begin{equation*}
        u_{\infty}\geq \ell_n= \inf_{x\in[-1,1]^n} 5\frac{1}{\binom{n}{2}}\sum\nolimits_{i<j}x_ix_j - 4\frac{1}{n(n-1)}\sum\nolimits_{i\neq j}x_i^2x_j - \frac{1}{n}\sum\nolimits_ix_i.
    \end{equation*}
    Moreover, we get a nonasymptotic bound for the convergence of our bounds, which reads 
    \begin{equation}\label{eq:theor_bd_mfg}
        u_n-\ell_n \leq \frac{60}{n},
    \end{equation}
    that is, $\ell_n$ converges to $u_{\infty}$ at a rate of $O(1/n)$ by~\eqref{eq:nestedness}.
    We numerically illustrate the convergence of $\ell_n$ to this asymptotic optimal value, and validate our theoretical rate of convergence, see Figure~\ref{fig:mfg_numerics} and Example~\ref{ex:mfg} for more details.
    \begin{figure}[h]
\centering
\begin{subfigure}{.48\textwidth}
  \centering
  \includegraphics[width=\linewidth]{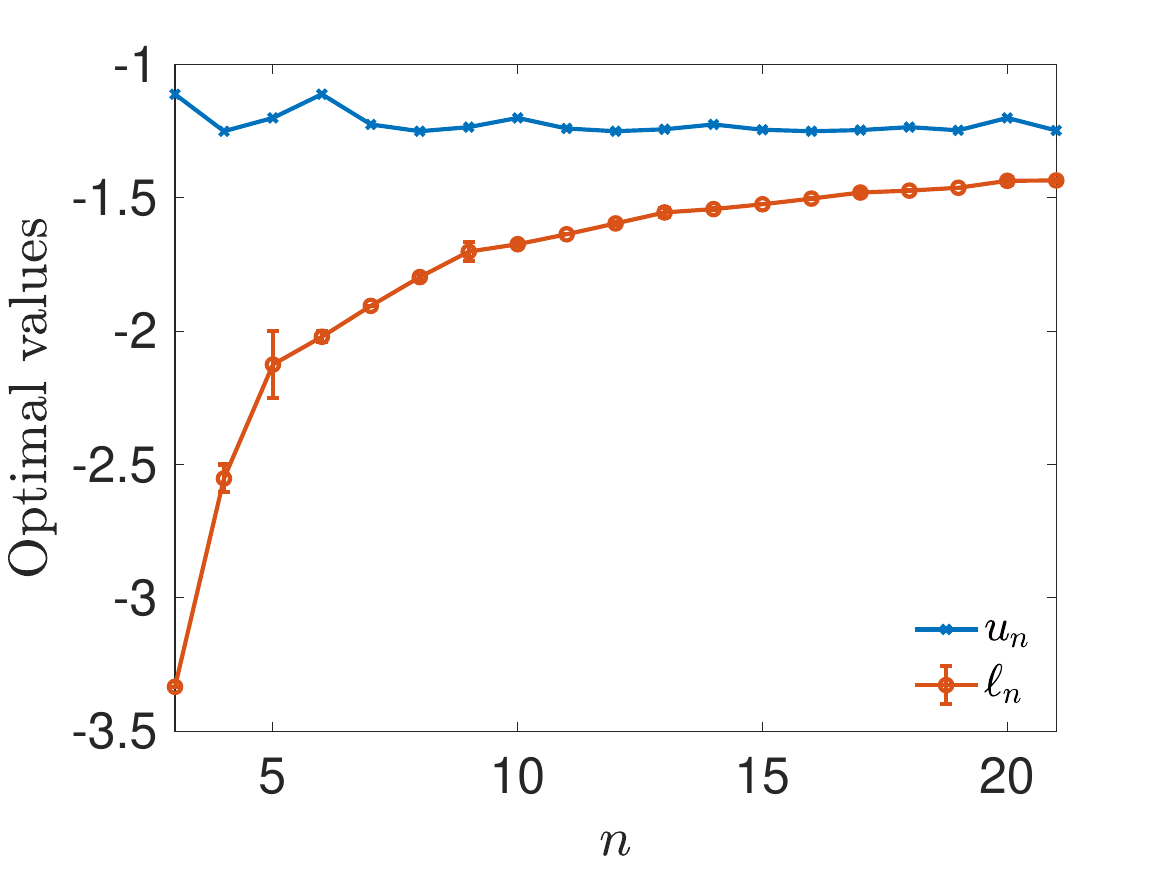}
  \caption{Upper bounds on $u_n$ and intervals containing $\ell_n$.}
  \label{fig:mfg_bounds}
\end{subfigure}%
\hfill
\begin{subfigure}{.48\textwidth}
  \centering
  \includegraphics[width=\linewidth]{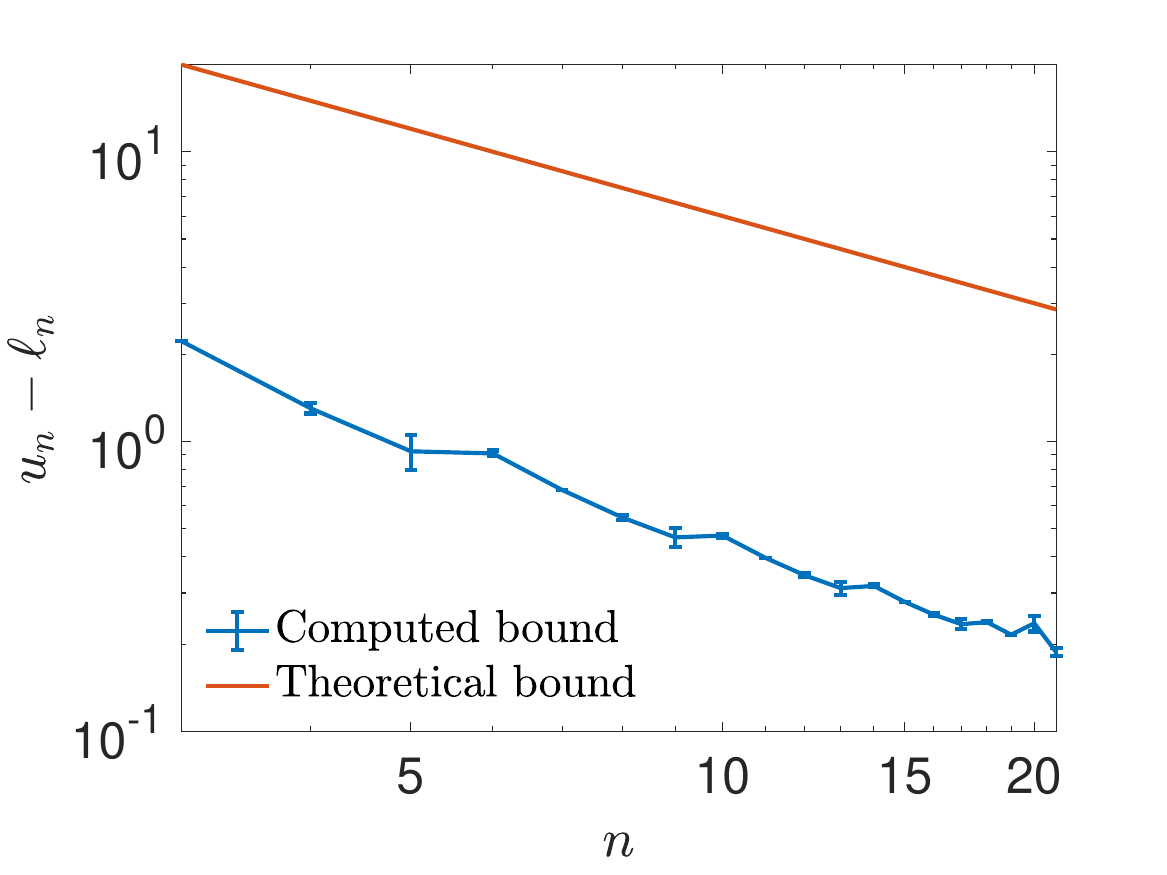}
  \caption{Comparing $u_n-\ell_n$ and bound from~\eqref{eq:theor_bd_mfg}.}
  \label{fig:mfg_bound_diffs}
\end{subfigure}
\caption{Numerical results for the mean-field game in Example~\ref{ex:mfg_intro}.}
\label{fig:mfg_numerics}
\end{figure}
\end{example}

\begin{example}[Nonnegative symmetric functions that are not SOS]\label{ex:SAGE_nonSOS}
Our framework allows us to prove new inequalities between symmetric functions that are inaccessible to previous, sums-of-squares-based approaches. 
Recall from Example~\ref{ex:power_sums} that any symmetric function can be written as a polynomial in power sums $s_k(x)=\sum_ix_i^k$. 
Consider the symmetric function
\begin{equation}\label{eq:sage_nonSOS_poly}
    \begin{aligned}
        f(s_1,s_2,s_4,s_6) = &\tfrac{1}{2}s_6 - \tfrac{15}{16}s_4s_2 + \tfrac{1}{16}s_4s_1^2 + \tfrac{15}{32}s_2^3 - \tfrac{3}{32}s_2^2s_1^2 - \tfrac{1}{32}s_2s_1^4 + \tfrac{1}{32}s_1^6\\ &+ \tfrac{3}{8}s_4 - \tfrac{9}{16}s_2^2 + \tfrac{3}{8}s_2s_1^2 - \tfrac{3}{16}s_1^4 + 1.
    \end{aligned}
\end{equation}
Checking whether or not this function is nonnegative in all dimensions reduces to setting $u_n=\inf_{x\in\RR^n}f(s_1(x),s_2(x),s_4(x),s_6(x))$ and checking whether $u_{\infty} = \inf_n u_n \geq0$. 
Notably, the polynomial~\eqref{eq:sage_nonSOS_poly} is not a sum of squares in any dimension $n\geq 2$, so its nonnegativity cannot be proved using the relaxations of~\cite{acevedo2024symmetric}, see Example~\ref{ex:SAGE_nonSOS_explicit} for details. 
We prove its nonnegativity in all dimensions by using our framework to produce the lower bound
\begin{equation}\label{eq:motzkin_bound}
    \ell_2 = \inf_{x\in\RR^2}x_1^4x_2^2+x_1^2x_2^4-3x_1^2x_2^2+1
\end{equation}
consisting of minimizing the well-known Motzkin polynomial. The polynomial in~\eqref{eq:motzkin_bound} is nonnegative by the AM-GM inequality, proving that $\ell_2\geq0$ and hence $u_{\infty}\geq0$.
\end{example}

\begin{example}[Graph densities]\label{ex:graph_densities_intro}
To certify inequalities in graph densities, we apply our framework to the sequence of problems~\eqref{eq:general_graph_densities} to produce the lower bounds
\begin{equation}\label{eq:general_inj_graph_densities}
    \ell_n = \inf_{\substack{\text{graphs }G,\\ |V(G)|=n}} ~ \sum_{i=1}^m c_i t_{\mathrm{inj}}(H_i;G),
\end{equation}
which are almost identical to~\eqref{eq:general_graph_densities} except that the homomorphism densities $t(H_i;G)$ are replaced with \emph{injective} homomorphism densities $t_{\mathrm{inj}}(H_i;G)$. The latter gives the fraction of injective maps between the vertex sets of $H_i$ and of $G$ that are graph homomorphisms. 
Moreover, we get the nonasymptotic convergence bound
\begin{equation*}
    u_n-\ell_n \leq \frac{V(V-1)\sum_{i=1}^m|c_i|}{n},\quad \textrm{where } V=\max_{i=1,\ldots,m}|V(H_i)|.
\end{equation*}
As an application, we prove that the Ramsey multiplicity of the graph consisting of a triangle with a pendant edge is lower-bounded by 0.0476, see Example~\ref{ex:ramsey_mult}. 
\end{example}

\begin{example}[Graph numbers]\label{ex:graph_nums_intro}
Unlike inequalities in graph densities, inequalities in their unnormalized counterparts have been less extensively studied. 
For example, let $\mathrm{inj}(H;G)$ be the number of injective graph homomorphisms from $H$ to $G$. This extends to a polynomial defined on all weighted graphs as explained in Section~\ref{sec:graph_densities}.
    Consider 
    \begin{equation}\label{eq:un_graph_nums_intro}
        u_n = \inf_{X\in\mbb S^n}\mathrm{inj}(P_3;X) - \mathrm{inj}(K_2\sqcup K_2; X),\quad \textrm{s.t. } X\geq0,\ \mathbbm{1}^\top X\mathbbm{1}\leq 1.
    \end{equation}
    The objective is a homogeneous quadratic, so certifying $u_{\infty}\geq\gamma$ implies $\mathrm{inj}(P_3;X) - \mathrm{inj}(K_2\sqcup K_2; X)\geq \gamma \mathrm{hom}(K_2;X)^2$ for all $X\geq0$ of all sizes. 
    Our lower bounds take the form
    \begin{equation}\label{eq:ln_graph_nums_intro}
        \ell_n = \inf_{X\in\mbb S^n} \frac{n^3}{n(n-1)(n-2)}\left[\mathrm{inj}(P_3;X)-\frac{n+1}{n-3}\mathrm{inj}(K_2\sqcup K_2;X)\right],\quad \textrm{s.t. } X\geq0,\ \mathbbm{1}^\top X\mathbbm{1}\leq 1.
    \end{equation}
    We numerically compute upper bounds on $u_n$ and intervals containing $\ell_n$ for $n\in[4,9]$, with the results shown in Table~\ref{tab:graph_inj_numerics}, see Example~\ref{ex:graph_nums} for more details.

    \begin{table}[h!]
    \centering
    \begin{tabular}{c|cccccc}
        $n$ & 4 & 5 & 6 & 7 & 8 & 9\\ \hline
        $u_n$ & -0.50 & -0.50 & -0.67 & -0.67 & -0.75 & -0.75\\
        $\ell_n$ & -6.67 & [-3.30, -3.12] & -2.80 & [-2.33, -2.18] & -2.06 & [-1.88, -1.81]
    \end{tabular}
    \caption{Bounds on $u_n$ and intervals containing $\ell_n$ for Example~\ref{ex:graph_nums_intro}.}
    \label{tab:graph_inj_numerics}
    \end{table}
\end{example}
We construct and analyze these bounds using new generalizations of de Finetti's theorem from probability~\cite{de1929funzione,de1937prevision,austin_survey}.
We proceed to illustrate this construction and the role de Finetti's theorem plays in it in the context of Example~\ref{ex:power_means}, highlighting the main steps that we shall generalize in the rest of the paper.

\subsection{A First Illustration}\label{sec:case_study}
We consider the sequence of POPs~\eqref{eq:power_means_general}, which can be written as
\begin{equation*}
    u_n = \inf_{x\in\Theta^n}p_n(x),\quad \textrm{where } p_n(x) = f(\mbb E_{\mu(x)}X^{\alpha_1},\ldots,\mbb E_{\mu(x)}X^{\alpha_m}),
\end{equation*}
and $\mu(x)=\frac{1}{n}\sum_{i=1}^n\delta_{x_i}$. In particular, note that for any $n | N$ the map $(x_1,\ldots,x_n) \mapsto (\underbrace{x_1,\ldots,x_1}_{N/n \mathrm{~times}}, \allowbreak \ldots, \allowbreak \underbrace{x_n,\ldots,x_n}_{N/n \mathrm{~times}})$ takes a feasible point of the $n$'th problem to a feasible point of the $N$'th problem with the same objective value.  This implies that the optimal values of \eqref{eq:power_means_general} satisfy $u_{nk} \leq u_{n}$ for each $n,k$.  The goal is to obtain lower bounds on the limiting optimal value $u_\infty = \inf_n u_n$, which may be expressed as:
\begin{equation}\label{eq:optim_over_measures}
    u_\infty = \inf_{\mu \in \mc P(\Theta)} p(\mu),\quad \textrm{where } p(\mu) = f\left(\mathbb{E}_\mu[X^{\alpha_1}], \dots, \mathbb{E}_\mu[X^{\alpha_m}]\right),
\end{equation}
where $X \in \R^d$ is a random variable and $\mc P(\Theta)$ is the set of probability measures supported on $\Theta$.  

We begin by rewriting the objective $p(\mu)$ as a linear function of a product measure. Suppose $f$ is equal to a single monomial with exponent $\beta \in \mathbb{N}^m$ of degree $k=\mathbbm{1}^\top\beta$.  In this case, we have that $f\left(\mathbb{E}_\mu[X^{\alpha_1}], \dots, \mathbb{E}_\mu[X^{\alpha_m}]\right) = \prod_{i=1}^m (\mathbb{E}_\mu[X^{\alpha_i}])^{\beta_i}$, which can be rewritten as:
\begin{equation*}
    \prod_{i=1}^m (\mathbb{E}_\mu[X^{\alpha_i}])^{\beta_i} = \prod_{i=1}^m \mathbb{E}_{\mu^{\otimes \beta_i}}[Y_1^{\alpha_i} \cdots Y_{\beta_i}^{\alpha_i}] = \mathbb{E}_{\mu^{\otimes 1' \beta}}[(Z_{1,1}^{\alpha_1} \cdots Z_{1,\beta_1}^{\alpha_1}) \cdots (Z_{m,1}^{\alpha_m} \cdots Z_{m,\beta_m}^{\alpha_m})]
\end{equation*}
where the $Y_i \in \R^d$ and the $Z_{i,j} \in \R^d$ are random variables.  In both equalities, we appeal to the observation that the expectation of a product of independent random variables is the product of their expectations.  Thus, when $f$ is a monomial of degree $k$, we have identified a monomial $q$ in $dk$ variables such that:
\begin{equation*}
    p(\mu) = \mathbb{E}_{\mu^{\otimes k}}[q]
\end{equation*}
By linearity, such a polynomial $q$ exists for any polynomial $f$. 
In particular, we can rewrite the sequence of objectives $(p_n)$ in the original sequence of POPs in terms of $q$, namely
\begin{equation}\label{eq:sampling_rep_example}
    p_n(x) = \mbb E_{\mu(x)^{\otimes k}}[q] = \mbb E[q(L_{k, n}x)],
\end{equation}
where $L_{k,n}\colon(\RR^d)^n\to(\RR^d)^k$ is a random linear map sampling $k$ coordinates uniformly at random with replacement.

We show next that we can obtain lower bounds on $u_\infty$ by changing the sampling map appearing in the identity~\eqref{eq:sampling_rep_example}.
Specifically, for $k\leq n$ let $\widetilde L_{k,n}\colon (\RR^d)^n\to(\RR^d)^k$ be a random linear map sampling $k$ coordinates \emph{without} replacement, and define the new sequence of polynomials
\begin{equation}\label{eq:dual_cost_pwr_means}
    q_n(x) = \mbb E[q(\widetilde L_{k,n}x)] = \frac{1}{|\msf S_n|}\sum_{\sigma\in\msf S_n}q(\sigma\cdot(x_1,\ldots,x_n)) = \mathrm{sym}_nq(x),
\end{equation}
where we have extended a polynomial $q$ on length-$k$ sequences to a polynomial on length-$n$ sequences that only depends on their first $k$ entries, and $\msf S_n$ is the group of permutations on $n$ letters acting by permuting the coordinates of $x$. 
We show that minimizing the new sequence of polynomials $(q_n)$ over the same sequence of constraint sets gives our desired lower bounds. Formally, setting
\begin{equation}\label{eq:lower_bound_pwr_means}
    \ell_n = \inf_{x\in\Theta^n}q_n(x),
\end{equation}
for each $n\geq k$, we show that $\ell_n\leq u_{\infty}$. To this end, we reformulate~\eqref{eq:lower_bound_pwr_means} in terms of probability distributions
\begin{equation}
    \ell_n = \inf_{\nu_n\in\mc P(\Theta^n)} \mbb E_{\nu_n}[q_n] = \inf_{\nu_n\in\mc P(\Theta^n)} \mbb E_{\nu_n}[\mathrm{sym}_nq].
\end{equation}
Next, because $\mathrm{sym}_nq$ is permutation-invariant, we can equivalently optimize only over exchangeable distributions $\nu_n$ on $\Theta^n$, meaning $(X_1,\ldots,X_n)\overset{d}{=}(X_{\pi(1)},\ldots,X_{\pi(n)})$ for any permutation $\pi\in\msf S_n$ if $(X_1,\ldots,X_n)$ is drawn from $\nu_n$. Moreover, we have $\mbb E_{\nu_n}[\mathrm{sym}_nq] = \mbb E_{\nu_n}[q]$ if $\nu_n$ is exchangeable, as can be seen directly from~\eqref{eq:dual_cost_pwr_means}. Therefore, we have
\begin{equation}
    \ell_n = \inf_{\nu_n\in\mc P(\Theta^n)} \mbb E_{\nu_n}[q]  ~~~ \mathrm{s.t.} ~~~ \nu_n \mathrm{~is~exchangeable}.
\end{equation}
Since $q$ only depends on a sequence of length $n$ via its first $k$ elements, we can further reformulate the above problem in terms of distributions on $\Theta^k$, namely
\begin{equation}\label{eq:lower_bounds_reform_example}
    \ell_n = \inf_{\nu_k\in\mc P(\Theta^k)} \mbb E_{\nu_k}[q] \quad \textrm{s.t.\ there is exchangeable } \nu_n\in\mc P(\Theta^n) \textrm{ projecting onto } \nu_k,
\end{equation}
where the projection condition in the constraint means that $\nu_k$ equals the pushforward of $\nu_n$ under the projection $(x_1,\ldots,x_n)\mapsto (x_1,\ldots,x_k)$ extracting the first $k$ entries.
In comparison, by combining~\eqref{eq:optim_over_measures} and~\eqref{eq:sampling_rep_example}, we can write the limiting optimal value $u_{\infty}$ as
\begin{equation}\label{eq:limiting_optimal_value_reform_example}
    u_{\infty} = \inf_{\nu_k\in\mc P(\Theta^k)}\mbb E_{\nu_k}[q] \quad \textrm{s.t.\ } \nu_k=\mu^{\otimes k} \textrm{ for some } \mu\in\mc P(\Theta).
\end{equation}
Comparing~\eqref{eq:lower_bounds_reform_example} and~\eqref{eq:limiting_optimal_value_reform_example}, it is now clear that $\ell_n\leq u_{\infty}$, since for any $\mu\in\mc P(\Theta)$ the product measure $\nu_k=\mu^{\otimes k}$ is feasible for~\eqref{eq:lower_bounds_reform_example} as it is the projection of the larger exchangeable measure $\mu^{\otimes n}$.
We have thus shown that the finite-dimensional problem $\ell_n$ from~\eqref{eq:lower_bound_pwr_means} gives a lower bound on $u_{\infty}$ for each $n$. If $\Theta$ is semialgebraic, then the problem~\eqref{eq:lower_bound_pwr_means} is a fixed-dimensional POP in dimension $d \cdot n$. 

Using the formulation~\eqref{eq:lower_bounds_reform_example}, we can further see that the lower bounds $(\ell_n)_{n\geq k}$ are monotonically increasing to $u_{\infty}$. Indeed, we have $\ell_n\leq \ell_N$ for $n\leq N$ because if $(X_1,\ldots,X_N)$ is an exchangeable random vector of length $N$ then extracting its first $n$ coordinates gives an exchangeable vector $(X_1,\ldots,X_n)$ of length $n$. 
Taking $n\to\infty$ gives
\begin{equation}\label{eq:limiting_lower_bd}
    \ell_{\infty} = \sup_n\ell_n = \inf_{(\nu_n)_{n\in\NN}} ~ \mathbb{E}_{\nu_k}[q] ~~~ \mathrm{s.t.} ~~~ \nu_n \in \mc P(\Theta^n) \mathrm{~is~exchangeable}, ~ \nu_N \mathrm{~projects~to~} \nu_n \mathrm{~for~} N \geq n.
\end{equation}
Equivalently, such sequences of measures $(\nu_n)$ correspond to distributions of infinite exchangeable arrays $(X_1,X_2,\ldots)$. Using de Finetti's theorem, we conclude that such sequences $(\nu_n)$ are precisely mixtures of sequences of product measures $(\mu^{\otimes n})$ for $\mu\in\mc P(\Theta)$, proving that $\ell_{\infty}=u_{\infty}$ and hence that our sequence of lower bounds converges to $u_{\infty}$. We can further quantify the rate of this convergence using a finite variant of de Finetti's theorem due to Diaconis and Freedman~\cite{diaconis_freedman}. Specifically, comparing sampling with replacement in the original sequence of objectives~\eqref{eq:sampling_rep_example} with sampling without replacement in the new sequence of objectives~\eqref{eq:dual_cost_pwr_means}, we obtain
\begin{equation} \label{eq:power-means-step5}
    u_n - \ell_n \lesssim \frac{1}{n}.
\end{equation}
Summarizing, we expressed the original sequence of objectives $(p_n)$ in terms of sampling with replacement applied to a fixed polynomial in~\eqref{eq:sampling_rep_example}. Changing the sampling map to sampling without replacement yields a new sequence of objectives~\eqref{eq:dual_cost_pwr_means} that we minimize over the same sequence of constraints to obtain lower bounds on the entire original sequence. These lower bounds converge at an explicit $O(1/n)$ rate due to de Finetti's theorem. 

In this paper, we generalize the above construction to other any-dimensional POPs, including Examples~\ref{ex:power_sums} and~\ref{ex:graph_densities}. Our generalization is based on expressing the sequence of objective functions in terms of more general sampling maps, which we then modify appropriately to obtain a new sequence of objectives. Minimizing the new objective functions over the same sequence of constraint sets, we obtain lower bounds converging at explicit $O(1/n)$ rates. This convergence is a consequence of new generalizations of de Finetti's theorem that we prove specifically for that purpose.


\subsection{Our Contributions}
Our framework consists of three components---formally defining any-dimensional POPs, proving new de Finetti-type theorems, and applying these theorems to obtain lower bounds on the limiting optimal values of any-dimensional POPs in terms of finite-dimensional POPs.  We outline each of these in some more detail:




\paragraph{Formalizing any-dimensional POPs (Section~\ref{sec:any_dim_poly_probs}).} For concreteness, denote a sequence of POPs by $\{u_n = \inf_{x \in \Omega_n} ~ p_n(x)\}$ with semialgebraic constraints $\Omega_n \subseteq \vct V_n$ and polynomial objectives $p_n : \vct V_n \rightarrow \R$, where $\vct V_n$ is the vector space underlying the $n$'th problem.  To obtain lower bounds on $u_\infty = \inf_n u_n$ via finite-time procedures, we require that the problems in the sequence are related to each other in some fashion, as without such relations we would be seeking lower bounds on arbitrary sequences of numbers.  In the illustration in Section~\ref{sec:case_study}, these relations are given by the map $\delta_{N,n}(x_1,\dots,x_n) =  (\underbrace{x_1,\ldots,x_1}_{N/n \mathrm{~times}}, \allowbreak \ldots, \allowbreak \underbrace{x_n,\ldots,x_n}_{N/n \mathrm{~times}})$ from $\R^{d \times n}$ to $\R^{d \times N}$ for $n | N$.  This map takes a feasible point of the $n$'th problem to a feasible point for the $N$'th problem, and the sequence of costs $(p_n)$ satisfies $p_n = p_N \circ \delta_{N,n}$.  The consequence is that the optimal values satisfy $u_{N} \leq u_n$ for $n | N$.  The sequence of problems in Example~\ref{ex:graph_densities} can also be understood as being related across dimension by an appropriate type of duplication map (see Example~\ref{ex:freely-described-POPs}), and the corresponding sequence of optimal values also satisfy $u_N \leq u_n$ for $n | N$.  In other cases such as Example~\ref{ex:power_sums}, one can check that the problems are similarly related across dimension by the zero-padding map $\zeta_{N,n}(x) = (x,0_{N-n})$ for $n \leq N$, so that the associated optimal values satisfy $u_{N} \leq u_n$ for $n\leq N$.  We refer to such sequences of POPs related via embeddings across problem dimensions as \emph{freely-described} (see Definition~\ref{def:freely_described_pops}).  In Section~\ref{sec:case_study}, the lower bounds are given by a sequence of POPs $\{\ell_n = \inf_{x \in \Omega_n} ~ q_n(x)\}$ in which the costs $q_n = \mathrm{sym}_nq$ are higher-dimensional symmetrizations of a fixed polynomial $q$; we refer to such sequences of POPs as \emph{freely-symmetrized} (see Definition~\ref{def:freely_sym_pops}).  Generalizing the case study in Section~\ref{sec:case_study}, our objective is to obtain lower bounds via freely-symmetrized POPs on the limiting optimal values of freely-described POPs.  The definitions of these two types of sequences of POPs are based on ideas from \emph{representation stability}, which concerns sequences of group representations related across dimension~\cite{CHURCH2013250}. 

\paragraph{Generalizing de Finetti (Section~\ref{sec:deFinetti}).} De Finetti's theorem characterizes an infinite sequence of measures that project onto each other as in~\eqref{eq:limiting_lower_bd}.  Diaconis and Freedman~\cite{diaconis_freedman} show that one can obtain a dense subset of such infinite sequences of measures by sampling with replacement the coordinates of finite random vectors.  We broaden this perspective in two ways.  First, we generalize the notion of infinite sequences of measures that project onto each other by again leveraging representation stability.  In particular, for sequences of representations of permutation groups related across dimension, we consider projections induced by these relations and associated sequences of exchangeable measures which project onto each other (see Definition~\ref{def:fdm}).  Second, we generalize the idea of sampling with replacement using actions of maps between finite sets on the underlying vector spaces (see Definitions~\ref{def:coFS_mod} and \ref{def:FS_mod}).  These notions enable us to generalize the result of Diaconis and Freedman in Theorems~\ref{thm:DeFin_for_coFS} and \ref{thm:dual_deFin_general}.  
As an illustration, we obtain the following new `dual' de Finetti theorem:
\begin{theorem}[informal, special case; see Theorem~\ref{thm:dual_deFin_general}]\label{thm:dual_deFin_informal}
    Consider any sequence of exchangeable random vectors $(X^{(n)}\in\RR^n)_n$ projecting onto each other in the following sense
\begin{equation*}
\begin{aligned}
X^{(nk)} = \Big(\underbrace{X^{(nk)}_1,\ldots, X^{(nk)}_k}, ~~~ &\cdots ~~~, \underbrace{X^{(nk)}_{(n-1)k+1},\ldots,X^{(nk)}_{nk}}\Big)\\
&~\, \Big\downarrow\\
X^{(n)} \overset{d}{=} \Big(X^{(nk)}_1+\ldots+X^{(nk)}_k,  ~~~ &\cdots ~~~, X_{(n-1)k+1}^{(nk)} + \ldots + X^{(nk)}_{nk}\Big),
\end{aligned}
\end{equation*}
with $\sup_n\mbb E\|X^{(n)}\|_1<\infty$.
Any such sequence can be approximated arbitrarily well by sequences constructed as follows: For a finite random vector $Y\in\RR^d$, and a uniformly random map $F_{n,d}\colon[d]\to[n]$ that is independent of $Y$, construct the sequence $(\tau(F_{n,d})Y\in\RR^n)_{n\in\NN}$ where the $i$th entry of $\tau(F_{n,d})Y$ is given by $\sum_{j\in F_{n,d}^{-1}(i)}Y_j$.
%
\end{theorem}
Applying $\tau(F_{n,d})$ to $Y$ amounts to randomly binning the coordinates of $Y$ into $n$ bins and summing all the coordinates in each bin. Its adjoint $\tau(F_{n,d})^\star$ with respect to the standard inner products corresponds to sampling $d$ entries with replacement from a length-$n$ vector, hence the duality with classical de Finetti.
The proof is based on a comparison between random partitions and random equipartitions (partitions into equal-sized parts), analogously to the proof of classic de Finetti based on a comparison between sampling with and without replacement.
Our results also apply to higher-dimensional arrays. In particular, we derive new dual analogs of results from Aldous--Hoover theory and the theory of graphons~\cite{kallenberg2005probabilistic,lovasz2012large}.

\paragraph{Lower Bounds on Any-Dimensional POPs (Section~\ref{sec:free_poly_opts}).} We apply our generalized de Finetti theorems to construct lower bounds for the limiting optimal values of freely-described POPs in Section~\ref{sec:free_poly_opts}.  As we remarked earlier in Section~\ref{sec:case_study}, the key step in our construction is expressing objective functions of such POPs in terms of appropriate sampling maps applied to a fixed-dimensional polynomial as in~\eqref{eq:sampling_rep_example}.  We formalize this point in Theorem~\ref{thm:free_sym_bds}. The following is an informal version of Theorem~\ref{thm:free_sym_bds} applied to the problem of minimizing symmetric functions as in Example~\ref{ex:power_sums}.
\begin{theorem}[informal, special case; see Theorem~\ref{thm:free_sym_bds}]\label{thm:free_sym_bds_informal}
    Consider $u_n=\inf_{x\in\Omega_n}f(s_1(x),\ldots,s_k(x))$ where $f(s_1,\ldots,s_k)$ is a symmetric function as in Example~\ref{ex:power_sums} and $\Omega_n$ is either $\RR^n$, the simplex $\Delta^{n-1}$, or the $\ell_1$ unit ball. 
    \begin{enumerate}
        \item There exists a fixed polynomial $q$ in $d=\deg f(s_1,\ldots,s_k)$ variables satisfying 
        \begin{equation}\label{eq:sym_func_representation}
            f(s_1(x),\ldots,s_k(x))=\allowbreak\mbb Eq\left(\sum_{i\in F_{d,n}^{-1}(1)}x_i,\ldots,\sum_{i\in F_{d,n}^{-1}(d)}x_i\right),
        \end{equation}
        for all $x\in\RR^n$ and all $n\in\NN$, where $F_{d,n}\colon[n]\to[d]$ is a uniformly random map. 

        \item The optimal values $\ell_n = \inf_{x\in\Omega_n}q_n(x)$ yield increasing lower bounds $\ell_n\leq u_{\infty}$, where for $d|n$ we have
        \begin{equation}\label{eq:sym_func_dual}
            q_n(x) = \mbb Eq\left(\sum_{i\in \Psi_{d,n}^{-1}(1)}x_i,\ldots,\sum_{i\in \Psi_{d,n}^{-1}(d)}x_i\right),
        \end{equation}
        and $\Psi_{d,n}\colon[n]\to[d]$ is a uniformly random equipartition.
        
        \item If $\Omega_n$ is compact, we get $u_n-\ell_n = O(1/n)$.
    \end{enumerate}
\end{theorem}
Note that $f(s_1,\ldots,s_k)$ and $q$ are related precisely by the sampling procedure arising in our generalization of de Finetti's theorem in Theorem~\ref{thm:dual_deFin_informal}, and that $q_n$ is obtained from $f(s_1,\ldots,s_k)$ by modifying this sampling procedure, in analogy to~\eqref{eq:sampling_rep_example} and~\eqref{eq:dual_cost_pwr_means} in Section~\ref{sec:case_study}.
The construction of $q$ given $f(s_1,\ldots,s_k)$ is fully automatic, much like in Section~\ref{sec:case_study}.
For the symmetric function in Example~\ref{ex:SAGE_nonSOS}, the fixed polynomial $q$ is the Motzkin polynomial.

Our results suggest a duality between the duplication and zero-padding embeddings in the following sense. For freely-described POPs related across dimension by duplication (resp. zero-padding) maps, the lower bounds generalizing~\eqref{eq:limiting_lower_bd} optimize over measures projecting onto each other by the adjoint of zero-padding (resp. duplication) maps.  
Our final bounds are given in terms of freely-symmetrized POPs consisting of the same constraint set as the original sequence of freely-described POPs but with a different sequence of costs.  
We also interpret and analyze our bounds from the perspective of any-dimensional certificates of nonnegativity in Section~\ref{sec:any_dim_nonneg_cones}.  

\paragraph{Applications (Section~\ref{sec:construct_bds}).} We illustrate our methods with several examples in Section~\ref{sec:construct_bds}, with the code to reproduce all our numerical results available at \url{https://github.com/eitangl/anyDimPolyOpt}.  In some cases, we clarify the relationship between any-dimensional POPs appearing in the literature, such as between the monomial means of Example~\ref{ex:power_means} and power means leveraged in~\cite{acevedo2024power} (Section~\ref{sec:power_means}), or between the homomorphism densities of Example~\ref{ex:graph_densities} and their injective analogs (Section~\ref{sec:graph_densities}). These are revealed to be consequences of our generalized de Finetti theorems.  
In other cases, we compute new bounds for symmetric functions (Section~\ref{sec:symmetric_funcs}) and homomorphism numbers (Section~\ref{sec:graph_numbers}), which are derived via bases for symmetric functions obtained using our de Finetti theorems; to the best of our knowledge, these bases and their applications have not been explored previously in the literature.

\paragraph{Paper Organization.} Section~\ref{sec:background} provides background on representation stability, which serves as a foundation for the later sections.  For readers primarily interested in the optimization applications, it is possible to get to the illustrations of our framework in Section~\ref{sec:construct_bds} after reading Sections~\ref{sec:background}, \ref{sec:any_dim_poly_probs}, \ref{sec:deFinetti_intro}, and \ref{sec:free_sym_bds}.  This shorter route through the paper presents the statements of our generalizations of de Finetti's theorem in Section~\ref{sec:deFinetti_intro} as well as statements and proofs of our results on freely-symmetrized bounds for freely-described POPs in Section~\ref{sec:free_sym_bds}, and these suffice for the developments in Section~\ref{sec:construct_bds}.  On the other hand, for readers primarily interested in a full treatment of our generalizations of de Finetti's theorem, it is possible to follow all of Section~\ref{sec:deFinetti} directly after Section~\ref{sec:background}. 

\subsection{Related Work}
Our work is related to several areas. We now survey past work in these areas and explain the connection with our own results in more detail.

\paragraph{Polynomial Optimization.} Leveraging the literature on sums-of-squares and moment (SOS) methods for fixed-dimensional POPs~\cite{parrilo2003semidefinite, lasserre2001global, blekherman2012semidefinite}, a recent body of work has developed convex relaxations for any-dimensional problems in special cases.  SDP relaxations for unconstrained problems of the form of Example~\ref{ex:power_means} for the particular case of $d=1$ were studied in~\cite{acevedo2024power}, while specific unconstrained cases of Example~\ref{ex:power_sums} were studied in~\cite{acevedo2023wonderful,acevedo2024symmetric} (see Section~\ref{sec:symmetric_funcs} for a `higher-dimensional' generalization of Example~\ref{ex:power_sums}). There is also a long line of work on certifying inequalities between graph homomorphism densities using SOS techniques, also known as flag algebras in the literature~\cite{razborov_2007,raymond2018symmetric,raymond2018symmetry,brosch_thesis}.  Setting up these SDP relaxations often involves serious and case-specific analytic work related to taking certain limits of symmetry-reduced SDPs. These, in turn, are based on understanding decompositions into irreducibles of spaces of polynomials with increasing degrees; see~\cite{acevedo2024power,acevedo2024symmetric,polak2022symmetry,brosch_thesis} for example.  In contrast, our methods are more straightforward to apply, as our finite-dimensional POPs for obtaining lower bounds can be constructed automatically for a given any-dimensional POP, including in settings where the relevant decompositions into irreducibles become quite involved~\cite{polak2022symmetry}.  
SDP relaxations of these finite-dimensional POPs can subsequently be computed using standard software packages. 
Another hierarchy of SDP relaxations was developed in~\cite{klep2025sums} for a class of optimization problems over measures encompassing Example~\ref{ex:power_means}. In that work, the authors apply techniques inspired by noncommutative SOS directly to the reformulation of the limiting optimal value as an optimization problem over measures~\eqref{eq:optim_over_measures}. It is not, however, clear how to extend their approach to other any-dimensional POPs, such as Examples~\ref{ex:power_sums}-\ref{ex:graph_densities}, where the limiting optimal value does not admit such a formulation.
There is also a literature studying the limitations of SOS-based relaxations, and more generally undecidability, for the above families of any-dimensional POPs, including~\cite{hatami2011undecidability, acevedo2023wonderful,Blekherman_Raymond_Wei_2024,acevedo2024symmetric, acevedo2024power, klep2025sums}. Some of these works involve explicit constructions or reductions, while others involve algebra-geometric methods such as tropicalization and Timofte's half degree principle.

A class of any-dimensional problems that we do not treat in this paper arises in quantum information, where it is of interest to minimize polynomials in quantum states over Hilbert spaces of any dimension (possibly even infinite dimension)~\cite{navascues2008convergent,npa,burgdorf2013tracial,ligthart2023inflation,ligthart2023convergent,renou_xu,klep2024state}. Such problems can be viewed as noncommutative generalizations of Example~\ref{ex:power_means}, where a probability distribution acting linearly on polynomials in commuting variables is replaced by a quantum state acting linearly on polynomials in noncommutative observables~\cite{huber2024positivity,klep2025sums}. Several hierarchies of relaxations have been developed for such problems in the above works. In particular, hierarchies based on quantum de Finetti theorems, derived similarly to Section~\ref{sec:case_study} above, have been proposed and analyzed in~\cite{ligthart2023inflation,ligthart2023convergent,renou_xu}. To our knowledge, no convergence rate has been proved for such relaxations, in contrast to our hierarchy.

\paragraph{Representation Stability.}
The study of sequences of representations related across dimensions was initiated in~\cite{CHURCH2013250}, where it was observed that the decompositions into irreducibles of such sequences often stabilize. Such stabilization was also observed in the context of polynomial optimization for some spaces of polynomials~\cite{riener2013exploiting, raymond2018symmetric, debus2020reflection, moustrou2021symmetry}. The seminal paper~\cite{FImods} studied sequences of representations of symmetric groups from the perspective of representations of categories, and proved that decompositions of such representations into irreducibles stabilize starting from a prescribed dimension. 
Subsequently, a large literature has emerged studying representations of various categories, in particular relating the combinatorics of the morphisms in the category with properties of its representations~\cite{WILSON2014269,FI_hom,sam_snowden_2015,sam2017grobner,GADISH2017450}.
We use the language and results of the representation stability literature to formalize sequences of symmetric POPs related across dimensions.
Further, we extend this body of work by showing that representations of the (opposite of the) category $\FS$ of finite sets with all maps between them underlie de Finetti's theorem and our generalizations of this theorem. 
Specifically, our extension of the notion of sampling with replacement appearing in de Finetti's theorem consists of applying a uniformly random morphism from the $\coFS$ and $\FS$ categories (see Section~\ref{sec:deFinetti}).  We then prove de Finetti-type theorems for representations of $\coFS$ and $\FS$ in Theorems~\ref{thm:DeFin_for_coFS} and~\ref{thm:dual_deFin_general}, respectively. 
Thus, our work represents a new point of contact between representation stability and probability via random actions of a category on its representations.


\paragraph{Exchangeable Arrays.}
Since de Finetti's original result~\cite{de1929funzione} on binary infinite exchangeable arrays, a number of extensions have been obtained pertaining both to the types of random variables involved and the sense in which they are exchangeable. The former includes extensions to random variables taking real values~\cite{de1937prevision,dynkin1953classes} and values in a compact Hausdorff space~\cite{hewitt_savage}.
The latter includes extensions to finite exchangeable arrays~\cite{diaconis_freedman}, higher-dimensional exchangeable arrays and partially-exchangeable arrays~\cite{hoover1979relations,aldous1981representations,kallenberg2005probabilistic}.
In a separate line of work~\cite{convergent_seqs1}, the convergence of sequences of dense graphs to so-called graphons (graph functions) is studied.  These were then shown to characterize consistent graph models~\cite{lovasz2012random}, which are sequences of measures on simple graphs projecting onto each other similarly to~\eqref{eq:limiting_lower_bd}, and their characterization in terms of graphons was shown to be a special case of Aldous--Hoover theory in~\cite{diaconis2007graph}.
In this paper, we generalize de Finetti's theorem in a different---and to our knowledge as yet unexplored---direction that is motivated by our optimization applications. 
Specifically, we view the preceding results as pertaining to sequences of exchangeable measures that project onto each other in specific ways, similarly to~\eqref{eq:limiting_lower_bd}.  We then study sequences of measures projecting onto each other in more general ways, and we generalize the results in~\cite{diaconis_freedman} to give dense subsets of such sequences as well as associated nonasymptotic error bounds.

\paragraph{Graphons and Flag Algebras.} 
The limiting optimal value $u_{\infty}$ for minimizing homomorphism densities in Example~\ref{ex:graph_densities} can be reformulated as an optimization problem over graphons, which are symmetric measurable functions $W\colon[0,1]^2\to[0,1]$, see~\cite[\S16]{lovasz2012large}. 
From this perspective, one can obtain $u_n-u_{\infty}=O(1/n)$ using~\cite[Lemma~2.4(b)]{LOVASZ2006933}.
Finally, taking SOS relaxations of $u_n$ and taking their limits as $n\to\infty$ yields precisely the lower bounds on $u_{\infty}$ that can be proved using flag algebras~\cite{raymond2018symmetric,raymond2018symmetry,brosch_thesis}.  Relative to this previous work, our framework gives several new results and suggests avenues for future work.  To the best of our knowledge, proving inequalities in graph densities by minimizing the corresponding expression in injective densities has not explicitly appeared in the literature before. 
These lower bounds are generally better than the ones obtained by computing $u_n$ and subtracting $O(1/n)$, as explained in Section~\ref{sec:free_poly_opts} and illustrated in Example~\ref{ex:goodman}. Further, they open the door to non-SOS-based proofs of graph density inequalities, since any relaxation applied to $\ell_n$ yields a lower bound on $u_{\infty}$. 
Finally, our bounds apply not just for simple graphs but also for weighted ones. 

\subsection*{Acknowledgments} 
We thank Emma Church, Jordan Ellenberg, and Graham White for suggesting the idea of using representations of FI and related categories to study de Finetti theorems. We also thank Jordan Ellenberg for helpful and productive discussions and for his feedback on our generalized de Finetti theorems. 
We thank Jorge Garza Vargas for the idea of using a Wasserstein distance in Lemma~\ref{lem:W1_dist}.
We thank the anonymous reviewers for their helpful comments and suggestions.
Some ideas in the proofs of~\eqref{eq:rate_FS} and Propositions~\ref{prop:free_sym_and_free_descr_isom}-\ref{prop:calc_for_stability} were suggested by the ChatGPT 5.6 Pro model.
The authors were partially supported by AFOSR grants FA9550-23-1-0070 and FA9550-23-1-0204. 
This work was conducted while EL was at the Department of Computing and Mathematical Sciences at Caltech.

\subsection*{Notation}
We take the set of natural numbers to be $\NN=\{0,1,2,\ldots\}$. We write $n|N$ for $n,N\in\NN$ to denote divisibility of $N$ by $n$, meaning that there is $k\in\NN$ satisfying $N=nk$. We denote by $\msf S_n$ the group of permutations on $n$ letters. If $\alpha\in\NN^d$ is a multi-index, we denote $|\alpha|=\sum_{i=1}^d\alpha_i$ and $\alpha!=\prod_{i=1}^d\alpha_i!$. We denote by $\preceq$ a partial order. 
Finite-dimensional real vector spaces are denoted by blackboard bold letters such as $\vct V,\vct U$, their product by $\vct V\times\vct U$, and their tensor product by $\vct V\otimes\vct U$. The $d$'th symmetric power of $\vct V$ is denoted by $\mathrm{Sym}^d\vct V$.
The space of symmetric $n\times n$ matrices is denoted by $\mbb S^n=\mathrm{Sym}^2\RR^n$, and the set of symmetric matrices with entries in a set $\Omega\subseteq\RR$ is denoted by $\mbb S^n(\Omega)$. 
The ring of polynomials over such a vector space $\vct V$ is denoted by $\RR[\vct V]$, with $\RR[\vct V]_{\leq d}$ and $\RR[\vct V]_d$ denoting the spaces of polynomials of degree at most $d$ and homogeneous polynomials of degree exactly $d$. 
All vector spaces in this paper are endowed with bases and inner products, and if $A\colon \vct V\to\vct U$ is a linear map between two vector spaces, we denote by $A^\star \colon\vct U\to\vct V$ the adjoint of $A$ with respect to the inner products on $\vct V,\vct U$.

The basis and inner product on $\RR^n$ are always the standard ones, namely $(e_i)_{i\in[n]}$ and $\langle x,y\rangle=x^\top y$. 
If $(e_{\alpha})_{\alpha\in\mc A}$ and $(e_{\beta})_{\beta\in \mc B}$ are orthogonal bases for inner product spaces $\vct V$ and $\vct U$, respectively, we endow $\vct V\times \vct U$ with the basis consisting of $\Big((e_{\alpha},0)\Big)_{\alpha\in\mc A}$ and of $\Big((0,e_{\beta})\Big)_{\beta\in\mc B}$ and inner product $\langle (v_1,u_1),(v_2,u_2)\rangle_{\vct V\times\vct U}=\langle v_1,v_2\rangle_{\vct V} + \langle u_1,u_2\rangle_{\vct U}$. We endow $\vct V\otimes\vct U$ with the basis $(e_{\alpha}\otimes e_{\beta})_{\alpha\in\mc A,\beta\in\mc B}$ and inner product $\langle v_1\otimes u_1,v_2\otimes u_2\rangle_{\vct V\otimes\vct U} = \langle v_1,v_2\rangle_{\vct V}\langle u_1,u_2\rangle_{\vct U}$, which is extended by linearity as usual. We identify $\mathrm{Sym}^d\vct V$ with the subspace of symmetric tensors in $\vct V^{\otimes d}$, endow it with the basis of orbit sums $(E_{[\alpha_1,\ldots,\alpha_d]}=\sum_{(\alpha_1',\ldots,\alpha_d')\in\msf S_d(\alpha_1,\ldots,\alpha_d)}e_{\alpha_1'}\otimes\cdots\otimes e_{\alpha_d'})_{[\alpha_1,\ldots,\alpha_d]\in \mc A^d/\msf S_d}$ indexed by multisets of size $d$, and with the inner product that makes $(E_{[\alpha_1,\ldots,\alpha_d]}/\sqrt{\prod_{i=1}^d\langle e_{\alpha_i},e_{\alpha_i}\rangle})$ orthonormal. Note that this is not the inner product induced from $\vct V^{\otimes d}$. 
For example, the basis we give $\mathrm{Sym}^2\RR^2=\mbb S^2$ is $E_{[1,1]}=e_1e_1^\top$, $E_{[1,2]} = e_1e_2^\top + e_2e_1^\top$, and $E_{[2,2]}=e_2e_2^\top$, and this basis is orthonormal. 
Finally, we endow $\RR[\vct V]_d$ with the basis of monomials $(m_{\mu}(x)=\prod_{\alpha\in\mc A}\langle x,e_{\alpha}\rangle^{\mu_{\alpha}})_{\mu\in\NN^{\mc A}, |\mu| = d}$ and with the Bombieri inner product that makes $(m_{\mu}/\sqrt{\frac{\mu!}{d!}\prod_{\alpha\in\mc A}\langle e_{\alpha},e_{\alpha}\rangle^{\mu_{\alpha}}})$ orthonormal. Sending a polynomial $p\in\RR[\vct V]_d$ to the unique symmetric tensor $P\in\mathrm{Sym}^d\vct V$ satisfying $p(x)=\langle P, x^{\otimes d}\rangle_{\vct V^{\otimes d}}$ defines an isomorphism between $\RR[\vct V]_d$ and $\mathrm{Sym}^d\vct V$, and the Bombieri inner product coincides with the restriction of the Frobenius inner product on $\vct V^{\otimes d}$ to the image of this isomorphism.
Random variables are denoted by capital letters. The distribution of a random variable $X$ is denoted $\mathrm{Law}(X)$. If $\mu$ is a probability measure, we write $X\sim\mu$ when $\mathrm{Law}(X)=\mu$. If $X$ and $Y$ are random variables, we write $X\overset{d}{=}Y$ to mean $\mathrm{Law}(X)=\mathrm{Law}(Y)$. Expectations are denoted by $\mbb E$.

\section{Background on Representation Stability}\label{sec:background}
In this section, we discuss the relevant concepts from the representation stability literature~\cite{CHURCH2013250,FImods,levin2023free}.  We begin by briefly reviewing linear representations of groups; more comprehensive introductions are available in~\cite{fulton2013representation,serre1977linear}.  A (linear) representation of a group $\msf G$ on a vector space $\vct V$ is a group homomorphism $\rho$ mapping elements of $\msf G$ to linear maps from $\vct V$ to itself.  A group element $g\in \msf G$ then acts on vectors $x\in\vct V$ via $g\cdot x\coloneqq \rho(g)x$. It is customary to omit the homomorphism $\rho$ since it is usually clear from context and to simply write $gx$ for the action of $g\cdot x$.  A vector $x\in\vct V$ is $\msf G$-invariant if $gx=x$ for all $g\in\msf G$. The collection of invariants in $\vct V$ forms a subspace denoted by $\vct V^{\msf G}$. 
In this paper, all groups are symmetric groups, that is, $\msf G=\msf S_n$ for some $n$. In this case, we can endow $\vct V$ with a $\msf G$-invariant inner product, so we shall always assume that our inner products are group-invariant (or equivalently, $\rho(\msf G)$ consists of orthogonal matrices). 
With respect to such an inner product, the orthogonal projection onto $\vct V^{\msf G}$ is given by the \emph{Reynolds operator} $\mathrm{sym}_{\msf G} = \frac{1}{|\msf G|}\sum_{g\in \msf G}g = \mbb EG$ where $G$ is a uniformly random element of $\msf G$. 
If $\vct V$ and $\vct U$ are representations of $\msf G$, then a linear map $A\colon\vct V\to\vct U$ is called \emph{$\msf G$-equivariant} if $A\circ g = g\circ A$ for all $g\in\msf G$, i.e., if $A$ commutes with the action of the group $\msf G$.

\subsection{Consistent Sequences}\label{sec:consist_seqs}


Representation stability studies nested sequences of representations called consistent sequences.  These sequences are indexed by a directed poset (one in which every two elements have a common upper bound) that specifies how these representations are nested.  For this paper, it suffices to consider the natural numbers $\NN$ with one of two partial orders $\preceq$, the standard total order denoted by $(\NN, \leq)$, and the divisibility partial order denoted by $(\NN, \cdot | \cdot)$.  In the first case $n \preceq N$ corresponds to $n \leq N$, while in the second case $n \preceq N$ corresponds to $N=nk$ for some $k\in\NN$.
\begin{definition}[Consistent sequence]
    A \emph{consistent sequence} is a collection $\mscr V = \{\mc N, \allowbreak (\msf S_n)_{n\in\mc N}, \allowbreak (\vct V_n)_{n\in\mc N}, \allowbreak (\varphi_{N,n})_{n\preceq N}\}$ composed of: 
    \begin{enumerate}[font=\emph, align=left]
        \item[(Poset)] a directed poset $\mathcal{N}$;
        \item[(Groups)] nested symmetric groups $\msf S_n\subseteq\msf S_N$ whenever $n \preceq N$;\footnote{Consistent sequences have been defined for more general sequences of groups~\cite{CHURCH2013250}, but only symmetric groups are used in this paper.}
        \item[(Vector spaces)] a sequence of vector spaces $(\vct V_n)_{n\in\mc N}$ such that $\vct V_n$ is an $\msf S_n$-representation; 
        \item[(Embeddings)] linear injective maps $(\varphi_{N,n}\colon \vct V_n\to \vct V_N)_{n \preceq N}$ such that $\varphi_{N,n}$ is $\msf S_n$-equivariant, such that $\varphi_{N,n}\circ\varphi_{n,k}=\varphi_{N,k}$ whenever $N\succeq n\succeq k$, and such that $\varphi_{n,n}=\mathrm{id}_{\vct V_n}$. 
    \end{enumerate}
\end{definition}

Throughout this paper, we focus on the following two consistent sequences arising from the sequence of permutation groups $\msf S_n$ acting on the sequence of vector spaces $\RR^n$ by permuting coordinates, namely, by $(\sigma\cdot x)_i=x_{\sigma^{-1}(i)}$.  They are both indexed by $\NN$ but with different partial orders, and they involve different embeddings between dimensions.
    \begin{enumerate}[font=\textbf, align=left]
        \item[(Zero-padding)] Let 
        \begin{equation} \label{eq:zero-padding}
            \mscr Z = \{(\NN,\leq), (\msf S_n)_{\mscr Z}, (\RR^n), (\zeta_{N,n})_{n\leq N}\},    
        \end{equation}
        be the consistent sequence in which $\msf S_n\subseteq \msf S_N$ by identifying $g\in\msf S_n$ with $\mathrm{blkdiag}(g, I_{N-n})\in\msf S_N$ and with embeddings by zero-padding $\zeta_{N,n}(x) = (x, 0_{N-n})$.
        \item[(Duplication)] Let  
        \begin{equation} \label{eq:duplicating}
            \mscr D = \{(\NN,\cdot\mid\cdot), (\msf S_n)_{\mscr D}, (\RR^n), (\delta_{N,n})_{n|N}\},
        \end{equation}
        be the consistent sequence in which $\msf S_n\subseteq \msf S_N$ by identifying $g\in\msf S_n$ with $g\otimes I_{N/n}\in\msf S_N$ and with embeddings by duplicating entries $\delta_{N,n}(x) = x\otimes\mathbbm{1}_{N/n}=(x_1,\ldots,x_1,\ldots,\allowbreak \underbrace{x_n,\ldots,x_n}_{N/n \textrm{ times}})$.
    \end{enumerate} 

We can construct new consistent sequences from existing ones such as $\mscr Z$ and $\mscr D$ using standard operations on vector spaces.\footnote{Formally, these `constructions' are (bi)functors from the category of vector spaces to itself.}  Moreover, it is possible to quantify the complexity of the new consistent sequence by keeping track of the constructions applied.  Specifically, if $\mscr W = \{\mc N, \allowbreak (\msf S_n)_{n\in\mc N}, \allowbreak (\vct W_n)_{n\in\mc N}, \allowbreak (\varphi_{N,n})_{n\preceq N}\}$, $\mscr U = \{\mc N, \allowbreak (\msf S_n)_{n\in\mc N}, \allowbreak (\vct U_n)_{n\in\mc N}, \allowbreak (\psi_{N,n})_{n\preceq N}\}$ are two consistent sequences indexed by the same poset and with the same nested sequence of symmetric groups acting on them, we can construct the following new consistent sequences:
\begin{enumerate}[align=left, font=\emph]
    \item[(Products)] The \emph{(Cartesian) product} of $\mscr W$ and $\mscr U$ is $\mscr W\times\mscr U = \{\mc N, (\msf S_n)_{n\in\mc N}, (\vct W_n\times \vct U_n)_{n\in\mc N}, (\varphi_{N,n}\times\psi_{N,n})_{n\preceq N}\}$. We denote by $\mscr W^d$ the product of $d$ copies of $\mscr W$.
    
    \item[(Tensors)] The \emph{tensor product} of $\mscr W$ and $\mscr U$ is $\mscr W\otimes\mscr U = \{\mc N, (\msf S_n)_{n\in\mc N}, (\vct W_n\otimes\vct U_n)_{n\in\mc N},\allowbreak (\varphi_{N,n}\otimes\psi_{N,n})_{n\preceq N}\}$.  The $k$'th tensor power $\mscr W^{\otimes k}$ of $\mscr W$ is the tensor product of $\mscr W$ with itself $k$ times, and the $k$th symmetric power is $\mathrm{Sym}^k\mscr W = \{\mc N,(\msf S_n)_{n\in\mc N},(\mathrm{Sym}^k\vct W_n)_{n\in\mc N}, (\varphi_{N,n}^{\otimes k})_{n\preceq N}\}$, where we view $\mathrm{Sym}^k\vct W_n\subseteq \vct W_n^{\otimes k}$ and restrict $\varphi_{N,n}^{\otimes k}$ to that subspace.
    
    \item[(Polynomials)] The set of \emph{degree-$k$ polynomials over $\mscr W$} is $\RR[\mscr W]_k = \{\mc N,(\msf S_n)_{n\in\mc N},\allowbreak (\RR[\vct W_n]_k)_{n\in\mc N},\allowbreak (\mc P_k(\varphi_{N,n}))_{n\preceq N}\}$, where $\mc P_k(\varphi_{N,n})p = p\circ \varphi_{N,n}^\star$ for $p\in \RR[\vct W_n]_k$. The set of polynomials of degree at most $k$ over $\mscr W$ is denoted $\RR[\mscr W]_{\leq k} = \bigoplus_{j=0}^k\RR[\mscr W]_j$ where $\RR[\mscr W]_0$ is the trivial consistent sequence with all vector spaces equal to $\RR$ and all actions being trivial.


\end{enumerate}
The actions of $\msf S_n$ on $\vct W_n$ and on $\vct U_n$ naturally extend to all of the above vector spaces. For instance, if $g\in\msf S_n$ we have $g\cdot(w,u)=(g\cdot w,g\cdot u)$ for $(w,u)\in\vct W_n\times\vct U_n$ and $g\cdot\sum_iw_i\otimes u_i=\sum_i(gw_i)\otimes(gu_i)$ for $\sum_iw_i\otimes u_i\in\vct W_n\otimes\vct U_n$.

If we start from a consistent sequence $\mscr V$ and apply some sequence of the above operations to it to obtain another sequence $\mscr U$, then we say that $\mscr U$ is a $\mscr V$-sequence. 
We denote a $\mscr V$-sequence by $\mscr U=\mc F(\mscr V)$ where $\mc F$ is the sequence of operations applied to $\mscr V$ to produce $\mscr U$.
We furthermore quantify the complexity of the operations $\mc F$ applied to $\mscr V$ using its \emph{degree}, defined inductively as follows.
We set the degree of the identity construction that doesn't change its input to $\deg \mathrm{id}=1$, and inductively set
\begin{equation*}\begin{aligned}
    &\deg \mc F_1\times \mc F_2 = \max\{\deg \mc F_1, \deg \mc F_2\}; && \deg \mc F_1 \otimes \mc F_2 = \deg \mc F_1 + \deg \mc F_2;\\ &\deg \mathrm{Sym}^k\mc F_1 = \deg \RR[\mc F_1]_k = k \deg \mc F_1.
\end{aligned}\end{equation*}
If $\mscr U$ is a $\mscr V$-sequence, we then say that $\mscr U$ has degree $d$ over $\mscr V$ if there is a construction $\mc F$ of degree $d$ satisfying $\mscr U=\mc F(\mscr V)$, in which case we write $\deg_{\mscr V}\mscr U=\deg\mc F$. We omit the subscript $\mscr V$ when it is clear from context.
Note that $\mscr Z$ and $\mscr D$ are indexed by different posets, so a given consistent sequence cannot be both a $\mscr Z$-sequence and a $\mscr D$-sequence.

\begin{example}\label{ex:Rdk}
    The consistent sequences $\mathrm{Sym}^k\mscr Z$ and $\mathrm{Sym}^k\mscr D$ both consist of the sequence of vector spaces $\vct V_n=\mathrm{Sym}^k\RR^n$, and they are degree-$k$ $\mscr Z$- and $\mscr D$-sequences, respectively.  Viewing an element $X\in\vct V_n$ as an order-$k$ symmetric tensor in $n$ dimensions, permutations act by $(\sigma\cdot X)_{i_1,\ldots,i_k}=X_{\sigma^{-1}(i_1),\ldots,\sigma^{-1}(i_k)}$, and this action is common to both consistent sequences.
    %
    The embeddings between dimensions are different, however, and given by
    $\zeta_{N,n}^{\otimes k}$ and $\delta_{N,n}^{\otimes k}$. Explicitly, these correspond to zero-padding and duplicating the entries of $X$ as follows:
    \begin{equation*}\begin{aligned}
        &(\zeta_{N,n}^{\otimes k}X)_{i_1,\ldots,i_k} = \begin{cases} X_{i_1,\ldots,i_k} &\textrm{if } i_j\in[n] \textrm{ for all } j\in[k]\\ 0 & \textrm{if } i_j\in[N]\setminus[n] \textrm{ for some } j\in[k]\end{cases}\\
        &(\delta_{N,n}^{\otimes k}X)_{(i_1-1)N/n+j_1,\ldots,(i_k-1)N/n + j_k} = X_{i_1,\ldots,i_k},\quad \textrm{for } i_{\ell}\in[n] \textrm{ and } j_{\ell}\in[N/n].
    \end{aligned}\end{equation*} 
    Here $n\leq N$ for the first line and $n|N$ for the second.  


\end{example}

\subsection{Free Invariants and Their Stabilization}

The sequences of polynomials that arise in the any-dimensional POPs in Examples~\ref{ex:power_means}-\ref{ex:graph_densities} and in the lower bounds \eqref{eq:lower_bound_pwr_means} in Section~\ref{sec:intro} are invariant polynomials defined over a consistent sequence, related to each other across dimensions in particular ways.  In this subsection, we show that sequences of such polynomials of bounded degree form a finite-dimensional vector space as a consequence of representation stability~\cite{CHURCH2013250}.  This fact allows us to describe any-dimensional POPs by finitely many parameters in Section~\ref{sec:any_dim_poly_probs}, and it plays an important role in our construction of lower bounds in Section~\ref{sec:free_poly_opts}.

Invariants in a consistent sequence $\mscr V=\{\mc N,(\msf S_n),(\vct V_n),(\varphi_{N,n})\}$ are related to each other across dimensions in two ways. First, projections of high-dimensional invariants yield low-dimensional ones, that is, for any $n \preceq N$, the adjoint of the embedding $\varphi_{N,n}^\star  \colon \vct V_N \rightarrow \vct V_n$ applied to a high-dimensional invariant $v_N \in \vct V_N^{\msf S_N}$ yields a low-dimensional invariant $\varphi_{N,n}^\star  v_N \in \vct V_n^{\msf S_n}$.  
Second, embedding low-dimensional invariants and symmetrizing them yields high-dimensional invariants, that is, if $v_n\in \vct V_n^{\msf S_n}$ then $\mathrm{sym}_{\msf S_N}\varphi_{N,n} v_n = \mbb E\Pi_N\varphi_{N,n}v_n\in \vct V_N^{\msf S_N}$ whenever $n \preceq N$ and $\Pi_N\in\msf S_N$ is a uniformly random permutation.  
In what follows, we abbreviate $\mathrm{sym}_{\msf S_N}\varphi_{N,n}$ for any $n\preceq N$ by $\mathrm{sym}_N^{\mscr V}$.
\begin{definition}\label{def:free_elements}
    Let $\mscr V$ be a consistent sequence.
    \begin{enumerate}[font=\emph, align=left]
        \item[(Freely-described)]  The space of \emph{freely-described elements} is the space of sequences
        \begin{equation}
            \varprojlim\vct V_n^{\msf S_n} = \left\{(v_n)\in\prod_{n\in\mc N}\vct V_n^{\msf S_n}: \varphi_{N,n}^\star v_N = v_n \textrm{ for all } n\preceq N\right\}.
        \end{equation}
        In words, a freely-described element is a sequence of invariants projecting onto each other. 

        \item[(Freely-symmetrized)] The space of \emph{freely-symmetrized elements} is the space of equivalence classes
        \begin{equation}
            \varinjlim\vct V_n^{\msf S_n} = \bigsqcup_{n\in\mc N}\vct V_n^{\msf S_n}/\sim\quad \textrm{where } v_n\sim v_N \textrm{ if } v_N=\mathrm{sym}_N^{\mscr V}v_n,
        \end{equation}
        where $\bigsqcup$ denotes the disjoint union. 
        A freely-symmetrized element is one of these equivalence classes, which we write as a sequence\footnote{Two such sequences are equivalent if they agree for all $n\succeq N$ for some fixed $N$.} $(\mathrm{sym}_n^{\mscr V}v_k)_{n\succeq k}$ when $v_k\in\vct V_k^{\msf S_k}$.
    \end{enumerate}
\end{definition}
The limits here correspond to inverse and direct limits from category theory. Informally, the directions of the arrows denote the `direction' in which a sequence of invariants must satisfy a relation---the left and right arrows correspond to mapping higher-dimensional invariants to lower-dimensional ones and vice-versa, respectively.

\begin{example}
The spaces of freely-described and freely-symmetrized elements in $\mscr Z$ are both one-dimensional and spanned by $(\mathbbm{1}_n)$ and $(\frac{1}{n}\mathbbm{1}_n)$, respectively.

For $\RR[\mscr Z]_{\leq 2}$, which consists of polynomials $\RR[x_1,\ldots,x_n]_{\leq 2}$ in $n$ variables of degree at most 2, we have:
\begin{equation*}
\begin{aligned}
    &\varprojlim\RR[x_1,\ldots,x_n]_{\leq 2}^{\msf S_n} = \mathrm{span}\left\{(1),\left(\sum_{i=1}^nx_i\right), \left(\sum_{i=1}^nx_i^2\right), \left(\sum_{1\leq i<j\leq n}x_ix_j\right)\right\},\\
    &\varinjlim\vct \RR[x_1,\ldots,x_n]_{\leq 2}^{\msf S_n} = \mathrm{span}\left\{(1), \left(\frac{1}{n}\sum_{i=1}^nx_i\right), \left(\frac{1}{n}\sum_{i=1}^nx_i^2\right),\left(\frac{1}{\binom{n}{2}}\sum_{1\leq i<j\leq n}x_ix_j\right)\right\}.
\end{aligned}
\end{equation*}
Here we have used the fact that if $p\in \RR[\vct V_N]_{\leq k}$ then $\mc P_{\leq k}(\varphi_{N,n})^\star p=p\circ\varphi_{N,n}$ with our choice of inner product on polynomials.
\end{example}
As we alluded to previously, in all of these examples the spaces of freely-described and freely-symmetrized elements have the same finite dimension, and moreover the difference between the basis elements for both is in how they are ``averaged'' across dimensions.  Remarkably, this is not a coincidence but a consequence of a general phenomenon whereby the projections $\varphi_{N,n}^\star $ and embeddings $\mathrm{sym}_N^{\mscr V}$ become isomorphisms between invariants.  To explain this phenomenon, we begin by observing that the space of freely-described elements can be identified with the dual space of freely-symmetrized elements.
\begin{proposition}\label{prop:free-duality}
    Let $\mscr V$ be a consistent sequence.  We have $(\varinjlim\vct V_n^{\msf S_n})^\star=\varprojlim\vct V_n^{\msf S_n}$, under the duality pairing $\langle (\mathrm{sym}_n^{\mscr V}v_k),(u_n)\rangle=\langle v_k,u_k\rangle$ for any $v_k\in\vct V_k^{\msf S_k}$ and $(u_n)\in\varprojlim\vct V_n^{\msf S_n}$.  Here $\vct V^\star$ denotes the (algebraic) dual space of $\vct V$ consisting of all linear functionals on $\vct V$.
\end{proposition}
\begin{proof}
    The given pairing is well-defined since if $v_N=\mathrm{sym}_N^{\mscr V}v_n$ then
    \begin{equation*}
        \langle v_N,u_N\rangle = \langle \mathrm{sym}_N\varphi_{N,n}v_n,u_N\rangle = \langle v_n,\varphi_{N,n}^\star u_N\rangle = \langle v_n,u_n\rangle. 
    \end{equation*}
    Furthermore, distinct freely-described elements define distinct linear functionals via this pairing because if $\langle (\mathrm{sym}_n^{\mscr V}v_k),(u_n)\rangle=0$ for all $v_k\in\vct V_k^{\msf S_k}$ and all $k$, then in particular $\langle(\mathrm{sym}_n^{\mscr V}u_m),(u_n)\rangle=\|u_m\|^2=0$ so $u_m=0$ for all $m$. Thus, it suffices to show that every linear functional on $\varinjlim\vct V_n^{\msf S_n}$ comes from a freely-described element. 
 Suppose $\ell\colon\varinjlim\vct V_n^{\msf S_n}\to\RR$ is a linear functional. Define the sequence of linear functionals $\ell_n\colon \vct V_n^{\msf S_n}\to\RR$ by $\ell_n=\ell|_{\vct V_n^{\msf S_n}}$, and let $u_n\in\vct V_n^{\msf S_n}$ satisfy $\ell_n(\cdot)=\langle \cdot,u_n\rangle$. Note that for any $n\preceq N$,
    \begin{equation*}
        \langle \varphi_{N,n}^\star u_N,v_n\rangle = \langle u_N, \mathrm{sym}_N\varphi_{N,n}v_n\rangle = \ell(\mathrm{sym}_N^{\mscr V}v_n)=\ell(v_n) = \langle u_n,v_n\rangle,
    \end{equation*}
    for all $v_n\in\vct V_n^{\msf S_n}$, so we conclude that $\varphi_{N,n}^\star u_N=u_n$ for all $n\preceq N$. Thus, $(u_n)$ is a freely-described element representing the linear functional $\ell$ under the above duality pairing.
\end{proof}
Based on this result, the space of freely-described elements is finite-dimensional if and only if the space of freely-symmetrized elements is, and in this case the two spaces have the same dimension.  The next proposition shows this finite dimensionality for $\mscr Z$- and $\mscr D$-sequences.
\begin{proposition}[{Finite-dimensional invariants}]\label{prop:finite_dim_free_invariants}
    If $\mscr V$ is a $\mscr Z$- or a $\mscr D$-consistent sequence of degree $d$, then $\dim\varprojlim\vct V_n^{\msf S_n}=\dim\varinjlim\vct V_n^{\msf S_n}\leq \dim\vct V_d^{\msf S_d}$. Equality holds when $\mscr V$ is a $\mscr Z$-sequence.
\end{proposition}
\begin{proof}
    If $\mscr V$ is a $\mscr Z$-sequence, all this follows from~\cite[Prop.~2.9, Thm.~2.11]{levin2023free}. If $\mscr V=\mc F(\mscr D)$ is a $\mscr D$-sequence, applying this result for the corresponding $\mscr Z$-sequence $\mc F(\mscr Z)$ we conclude that $\dim\vct V_n^{\msf S_n}$ is constant for all $n\geq d$. This implies that $\dim\varinjlim\vct V_n^{\msf S_n}\leq \dim\vct V_d^{\msf S_d}$ since if $(\mathrm{sym}_n^{\mscr V}v_{k_i})$ are linearly independent and $N\succeq k_i$ for all $i$ and $N\geq d$, then $\mathrm{sym}_N^{\mscr V}v_{k_i}$ must be linearly independent in $\vct V_N^{\msf S_N}$ whose dimension is $\dim\vct V_d^{\msf S_d}$. Proposition~\ref{prop:free-duality} now yields the result.
\end{proof}
In fact, we prove in Proposition~\ref{prop:calc_for_stability} below that equality is always attained in Proposition~\ref{prop:finite_dim_free_invariants} as we prove there that the projections and embeddings become isomorphisms of invariants for all sufficiently large dimensions.
This was proved for $\mscr Z$-sequences in~\cite{levin2023free} using the fact that the high-dimensional spaces in such sequences are spanned by the orbits of the low-dimensional ones. Below, we use our de Finetti theorems to extend this result to $\mscr D$-sequences.

\section{Any-Dimensional Polynomial Problems}\label{sec:any_dim_poly_probs}


In this section, we define the any-dimensional POPs that we study.  In particular, we define two classes of sequences of POPs defined on a sequence of growing-dimensional vector spaces. The first captures the relations across dimensions satisfied by the problems whose limiting values we seek to bound, e.g., in Examples~\ref{ex:power_means}-\ref{ex:graph_densities}, and the second captures the relations satisfied by our lower bounds for them, e.g., in~\eqref{eq:lower_bound_pwr_means}.

\subsection*{Freely-Described POPs}

The first class of problems, from which we start and for which we seek lower bounds, is one in which freely-described polynomials are minimized over a growing sequence of semialgebraic sets.
\begin{definition}[Freely-Described POPs]\label{def:freely_described_pops}
    Let $\mscr U=\{\mc N,(\msf S_n),(\vct V_n),(\varphi_{N,n})_{n\preceq N}\}$ be a consistent sequence.  A sequence of POPs $(u_n = \inf_{x\in\Omega_n}p_n(x))$ is called \emph{freely-described} if the following two conditions hold.
    \begin{enumerate}[font=\emph, align=left]
        \item[(Cost)] The sequence of costs $(p_n)$ is a freely-described element of $\RR[\mscr U]_{\leq d}$ for some degree $d$.  Equivalently, $(p_n)$ is a sequence of $\msf S_n$-invariant polynomials of degree at most $d$ that satisfy $p_N \circ \varphi_{N,n}= p_n$ for all $n \preceq N$.
        \item[(Constraint)] The sequence $(\Omega_n)$ consists of constraint sets that are semialgebraic, $\msf S_n$-invariant, and increasing in the sense that $\varphi_{N,n} \Omega_n \allowbreak \subseteq \allowbreak \Omega_N$ if $n\preceq N$.
    \end{enumerate}
\end{definition}
As we illustrate next, Examples~\ref{ex:power_means}-\ref{ex:graph_densities} in Section~\ref{sec:intro} are all instances of freely-described POPs.
\begin{example} \label{ex:freely-described-POPs}
    The any-dimensional POP~\eqref{eq:power_means_general} in Example~\ref{ex:power_means} is freely-described with respect to the consistent sequence $\mscr D^d$; indeed, as we noted in Section~\ref{sec:case_study}, the embeddings $\delta_{N,n}^d(x_1,\dots,x_n) = (\underbrace{x_1,\dots,x_1}_{N/n \mathrm{~times}}, \dots, \underbrace{x_n,\dots,x_n}_{N/n \mathrm{~times}})$ for $n | N$ map feasible points for the $n$th problem to feasible points for the $N$th one without changing the objective value.

    In Example~\ref{ex:power_sums} on symmetric functions, the sequence of POPs is freely-described with respect to the consistent sequence $\mscr Z$, since power sums are unchanged under zero-padding.

    Finally, in Example~\ref{ex:graph_densities} on graph density inequalities, we view the homomorphism densities as polynomials over adjacency matrices of simple graphs, which are points in $\Omega_n=\{X\in\mathbb{S}^n(\{0,1\}):\diag(X)=0\}$ consisting of symmetric matrices with binary entries and zero diagonals. It can then be shown that these homomorphism densities are unchanged under duplication of the vertices sending an adjacency matrix $X\in\Omega_n$ to $X\otimes\mathbbm{1}_k\mathbbm{1}_k^\top\in\Omega_{nk}$, see Section~\ref{sec:graph_densities}. Thus, these any-dimensional POPs are freely-described with respect to $\mathrm{Sym}^2\mscr D$.
    

\end{example}
Although any-dimensional POPs are given as an infinite sequence of POPs of growing dimension, they can often be described by finitely many parameters. In particular, if $\mscr V$ is a $\mscr Z$- or $\mscr D$-sequence, then Proposition~\ref{prop:finite_dim_free_invariants} shows that the freely-described sequence of costs $(p_n)$ is an element of a finite-dimensional vector space, so it can be specified by its coefficients in a basis.



A direct consequence of the definition of a freely-described POP is that the sequence of optimal values $(u_n)$ satisfies $u_N \leq u_n$ whenever $n \preceq N$:
\begin{equation} \label{eq:upper-bound-fd-pops}
    u_n = \inf_{x\in\Omega_n} p_n(x) = \inf_{x\in\Omega_n} p_N \circ \varphi_{N,n}(x) = \inf_{y \in \varphi_{N,n}\Omega_n} p_N(y) \geq \inf_{y \in \Omega_N} p_N(y) = u_N.
\end{equation}
Indeed, all of the examples above of freely-described POPs from Section~\ref{sec:intro} satisfy this relation between the optimal values.  
We seek lower bounds on the limiting optimal value $u_{\infty} = \inf_n u_n$ of a freely-described POP.  As each $u_n$ is, in general, larger than $u_\infty$ by~\eqref{eq:upper-bound-fd-pops}, merely obtaining a lower bound on one of the $u_n$'s in the sequence, e.g., using sums-of-squares methods, does not yield a bound on the limiting value.

\subsection*{Freely-Symmetrized POPs}

To derive lower bounds on $u_\infty$, we generalize the example from Section~\ref{sec:case_study} and construct a sequence of finite-dimensional POPs whose optimal values are monotonically \emph{increasing} to $u_{\infty}$.  The sequences of POPs we construct belong to the following class of freely-\emph{symmetrized} problems.
\begin{definition}[Freely-Symmetrized POPs]\label{def:freely_sym_pops}
    Let $\mscr L=\{\mc N,(\msf S_n),(\vct V_n),(\psi_{N,n})_{n\preceq N}\}$ be a consistent sequence.  A sequence of POPs $(\ell_n=\inf_{x\in\Omega_n}q_n(x))$ is called \emph{freely-symmetrized} if:
    \begin{enumerate}[font=\emph, align=left]
        \item[(Cost)] The sequence of costs $(q_n)$ is a freely-symmetrized element of $\RR[\mscr L]_{\leq d}$ for some degree $d$. That is, there is a $k$ and $q\in\RR[\vct V_k]_{\leq d}$ such that $q_n=\mathrm{sym}_n^{\mscr L} q$ for all $n\succeq k$.
        \item[(Constraint)] The sequence $(\Omega_n)$ consists of constraint sets that are semialgebraic, $\msf S_n$-invariant, and project onto each other in the sense that $\psi_{N,n}^\star \Omega_N\subseteq\Omega_n$ if $n\preceq N$.
    \end{enumerate}
\end{definition}

In a freely-symmetrized POP $(\ell_n = \inf_{x \in \Omega_n} ~ q_n(x))$, the sequence of optimal values $(\ell_n)$ is in general increasing with $n$.  Specifically, as $q_n(x) = \mathrm{sym}_n^{\mscr L} q_k(x)=\mbb Eq_k(\psi_{n,k}^\star \Pi_nx)$ for some $q_k\in\RR[\vct V_k]^{\msf S_k}$ with $n \succeq k$ and $\Pi_n\in\msf S_n$ uniformly random, we have
\begin{equation*} 
    \ell_n = \inf_{x\in\Omega_n}\mathrm{sym}_n^{\mscr L} q_k(x) = \inf_{\mu\in \mc P(\Omega_n)}\mbb E_{\Pi_n,X}[q_k(\psi_{n,k}^\star \Pi_n X)] = \inf_{\mu\in \mc P(\Omega_n)^{\msf S_n}}\mbb E_{\mu}[q_k(\psi_{n,k}^\star X)], 
\end{equation*}
where $X\sim \mu$ is independent of $\Pi_n$, in which case $\mathrm{Law}(\Pi_n X) = \mathrm{sym}_{\msf S_n}\mu\in\mc P(\Omega_n)^{\msf S_n}$.
Next, we can rewrite $\ell_n$ as an optimization problem over the pushforwards under $\psi_{n,k}^\star$ of measures in $\mc P(\Omega_n)^{\msf S_n}$, to obtain
\begin{equation}\label{eq:lower-bound-fs-pops1}
    \ell_n =  \inf_{\nu\in \psi_{n,k}^\star \mc P(\Omega_n)^{\msf S_n}}\mbb E_{\nu}[q_k(X)].
\end{equation}
Since $\psi_{N,n}^\star \Omega_N\subseteq\Omega_n$ for any $n\preceq N$, we have
\begin{equation} \label{eq:lower-bound-fs-pops2}
\psi_{N,k}^\star\mc P(\Omega_N)^{\msf S_N} = \psi_{n,k}^\star\psi_{N,n}^\star\mc P(\Omega_N)^{\msf S_N}\subseteq \psi_{n,k}^\star \mc P(\Omega_n)^{\msf S_n},
\end{equation}
for $k \preceq n \preceq N$.  
From~\eqref{eq:lower-bound-fs-pops1} we conclude that $\ell_N$ is the minimum value of the same objective over a smaller constraint set than $\ell_n$, hence the optimal values $(\ell_n)$ of freely-symmetrized POPs are increasing
\begin{equation} \label{eq:lower-bound-fs-pops}
    \ell_n \leq \ell_N,\quad \textrm{whenever } k \preceq n \preceq N.
\end{equation}
Note that when $\mscr L=\mscr Z$ and $k\leq n$, the random map $\psi_{n,k}^\star \Pi_n$ samples $k$ entries without replacement from a length-$n$ vector. In contrast, when $\mscr L=\mscr D$ and $k|n$ the random map $\psi_{n,k}^\star \Pi_n$ uniformly randomly partitions the entries of a length-$n$ vector into $k$ equally-sized bins and sums all the entries in each bin. These two are precisely the modified sampling maps giving lower bounds in Section~\ref{sec:case_study} and Theorem~\ref{thm:free_sym_bds_informal} from Section~\ref{sec:intro}, and will more generally give us lower bounds in Theorem~\ref{thm:free_sym_bds} below.
\begin{example}
For graphs $H$ and $G$, define $t_{\mathrm{inj}}(H;G)$ to be the fraction of injective maps $V(H)\to V(G)$ between their vertex sets that are graph homomorphisms. Note that the only difference with $t(H;G)$ from Example~\ref{ex:graph_densities} is that we restrict ourselves to injective maps.
Identifying a graph $X$ on $n$ vertices with its adjacency matrix $X\in\mbb S^n$, and defining the monomial $X^H = \prod_{1\leq i\leq j\leq |V(H)|}X_{i,j}^{H_{i,j}}$ over $\mbb S^{|V(H)|}$, we have $t_{\mathrm{inj}}(H;X)=\mathrm{sym}_n^{\mathrm{Sym}^2\mscr Z}X^H$ as we show in~\eqref{eq:tinj_is_sym} in Section~\ref{sec:graph_densities} below. 
Letting $\Omega_n=\{X\in\mbb S^n(\{0,1\}):\diag(X)=0\}$ be the adjacency matrices of simple graphs, which project onto each other under extraction of principal submatrices, for any $\alpha_1,\ldots,\alpha_m\in\RR$ and graphs $H_1,\ldots,H_m$, certifying inequalities in \emph{injective} graph densities over unweighted graphs can be written as 
\begin{equation*}
    \ell_n = \inf_{X\in\Omega_n}\sum_{j=1}^m\alpha_jt_{\mathrm{inj}}(H_j;X) = \inf_{X\in\Omega_n}\mathrm{sym}_n^{\mathrm{Sym}^2\mscr Z}\sum_{j=1}^m\alpha_jX^{H_j},
\end{equation*}
for all $n\geq \max_j|V(H_j)|$, and therefore constitutes a freely-symmetrized problem over $\mathrm{Sym}^2\mscr Z$. In particular, we have $\ell_n\leq\ell_{n+1}$ for all such $n$, which shows that $\ell_n\geq0$ for all such $n$ precisely when $\ell_k\geq0$ for $k=\max_j|V(H_j)|$. 
This observation implies that checking whether an inequality in injective homomorphism densities is nonnegative for graphs of all sizes is decidable (it involves checking nonnegativity over a finite set of graphs of size $\max_j|V(H_j)|$), in contrast to the undecidability of the analogous question for non-injective densities~\cite{hatami2011undecidability}.
\end{example}

Suppose now that we are given a freely-described POP with limiting optimal value $u_\infty$, and we wish to obtain lower bounds on $u_\infty$.  If we can identify a freely-symmetrized POP with optimal values $(\ell_n)$ such that their limiting optimal value $\ell_\infty = \sup_n \ell_n$ equals $u_\infty$, then each $\ell_n$ yields a lower bound on $u_\infty$ in terms of a finite-dimensional POP, and the sequence $(\ell_n)$ of these values gives monotonically increasing, convergent lower bounds for $u_\infty$.


\subsection*{Pairings of Freely-Described and Freely-Symmetrized Problems}
Given a freely-described POP with optimal values $(u_n)$, our goal is to construct a freely-symmetrized POP with optimal values $(\ell_n)$ such that $u_\infty = \inf_n u_n$ is equal to $\ell_\infty = \sup_n \ell_n$, generalizing the example in Section~\ref{sec:case_study}. 
Similarly to that example, we shall modify the sequence of objectives while keeping the same sequence of constraints. 
Importantly, the freely-described problems $(u_n)$ and their corresponding freely-symmetrized lower bounds $(\ell_n)$ are defined with respect to different consistent sequences.
In the example of Section~\ref{sec:case_study}, the freely-described problem is defined over the $\mscr D$-sequence $\mscr D^d$, while the freely-symmetrized problem is defined over the $\mscr Z$-sequence $\mscr Z^d$.  
We generalize this example in Section~\ref{sec:free_poly_opts} to any freely-described POP over a $\mscr Z$- or a $\mscr D$-sequence. If the freely-described POP is defined over a $\mscr D$-sequence, then our freely-symmetrized lower bounds are defined over the corresponding $\mscr Z$-sequence on the same underlying vector spaces, like in Section~\ref{sec:case_study}, and vice-versa.

Since the sequence of constraint sets $(\Omega_n)$ will be the same for the two sequences of POPs, we further require these sets to embed into each other in one consistent sequence and project into each other in the other consistent sequence.
This ensures Definitions~\ref{def:freely_described_pops} and~\ref{def:freely_sym_pops} are satisfied and the resulting optimal values $(u_n)$ and $(\ell_n)$ are monotonic as in~\eqref{eq:upper-bound-fd-pops} and~\eqref{eq:lower-bound-fs-pops}.
We now formally define the pairs of consistent sequences with respect to which our freely-described POPs and associated freely-symmetrized POPs are defined, and the conditions on the constraint sets we impose.
These correspond precisely to the class of problems to which our framework is applicable.
\begin{definition}[$\FS$-pair]\label{def:fs_pair}
    A \emph{$\FS$-pair of degree $d$} is a pair of consistent sequences 
    \begin{equation*}\begin{aligned}
        \mscr U=\{(\NN,\preceq_{\mscr U}),(\msf S_n),(\vct V_n),(\varphi_{N,n})_{n\preceq_{\mscr U} N}\} && \textrm{and} && \mscr L=\{(\NN,\preceq_{\mscr L}),(\msf S_n),(\vct V_n),(\psi_{N,n})_{n\preceq_{\mscr L} N}\},
    \end{aligned}\end{equation*}
    on the same sequence of vector spaces $(\vct V_n)$ such that there is a degree-$d$ sequence of constructions $\mc F$ satisfying $(\mscr U,\mscr L)=(\mc F(\mscr Z),\mc F(\mscr D))$ or $(\mc F(\mscr D),\mc F(\mscr Z))$.
    That is, $\mscr U$ and $\mscr L$ are obtained by applying the same constructions to $\mscr Z$ and $\mscr D$, or vice-versa.

    A sequence of subsets $(\Omega_n \subseteq \vct V_n)$ is said to be \emph{compatible with respect to the $\FS$-pair $(\mscr U, \mscr L)$}, or $\FS$-compatible for short, if $(i)$ each $\Omega_n$ is $\msf S_n$-invariant; $(ii)$ $\varphi_{N,n} \Omega_n \subseteq \Omega_N$ whenever $n \preceq_{\mscr U} N$; and $(iii)$ $\psi_{N,n}^\star \Omega_N \subseteq \Omega_n$ whenever $n\preceq_{\mscr L} N$.
\end{definition}
We show in Section~\ref{sec:deFinetti} below that such pairs of consistent sequences can be derived from certain actions of maps between finite sets, which is the reason for the terminology `FinSet'.  
Such actions play a central role in our generalization of de Finetti's theorem there.  

As an illustration, for the $\FS$-pair $(\mscr U,\mscr L) = (\mscr D,\mscr Z)$ the sequence of $\ell_{\infty}$-unit balls $\Omega_n=[-1,1]^n$ is $\FS$-compatible, but the sequence of $\ell_1$-unit balls $\Omega_n=\{x\in\RR^n:\|x\|_1\leq 1\}$ is not because it is not closed under duplication. On the other hand, for the $\FS$-pair $(\mscr U,\mscr L) = (\mscr Z,\mscr D)$ the sequence of $\ell_1$-balls is compatible, but the sequence of $\ell_{\infty}$-balls is not because it is not closed under the adjoint of duplication summing consecutive blocks of coordinates. The sequence of $\ell_2$-unit balls is not $\FS$-compatible for either pair. See Section~\ref{sec:construct_bds} for more examples and non-examples.

We leverage these generalized de Finetti theorems in Section~\ref{sec:free_poly_opts} to accomplish our goal: for any $\FS$-pair $(\mscr U,\mscr L)$ we start from a freely-described POP over $\mscr U$ with $\FS$-compatible constraints $(\Omega_n)$, and derive lower bounds in terms of freely-symmetrized POPs over $\mscr L$ with the same constraint sets.

\section{Free Measures and Generalized de Finetti Theorems}\label{sec:deFinetti}
As we have seen in the illustration of Section~\ref{sec:case_study}, de Finetti's theorem plays a prominent role in our derivation of lower bounds.  
In this section we present generalizations of this theorem by using ideas from representation stability.  
We begin in Section~\ref{sec:deFinetti_intro} by reformulating de Finetti's theorem as an assertion pertaining to sequences of measures, and state our generalizations in terms of such sequences of measures in Theorems~\ref{thm:DeFin_general_Zseq} and~\ref{thm:dual_deFin_Pseq}.  
The former pertains to exchangeable multi-dimensional arrays and recovers some aspects of Aldous--Hoover--Kallenberg theory, while the latter is a certain `dual' result that seems to not have been previously studied.
These are the two results that will be used in Section~\ref{sec:free_poly_opts} to construct lower bounds for more general freely-described sequences of problems.
Then in Sections~\ref{sec:deFinetti_coFS} and~\ref{sec:deFinetti_FS} we explain the key ingredients in proving our theorems, stated in the more abstract language of actions of maps between finite sets that yields cleaner statements and proofs. 
Finally, the proofs of all the key results are given in Section~\ref{sec:deFinetti_proofs}.  
This section is organized so that it is possible to proceed directly to Section~\ref{sec:free_poly_opts} after Section~\ref{sec:deFinetti_intro}.

\subsection{Reformulating and Generalizing de Finetti's Theorem}\label{sec:deFinetti_intro}

De Finetti's theorem states that an infinite exchangeable array is a mixture of iid arrays. 
More usefully for us, Diaconis and Freedman~\cite{diaconis_freedman} proved that infinite exchangeable arrays can be approximated arbitrarily well by particularly simple ones obtained by sampling the entries of finite random vectors with replacement.
Formally, given a finite random vector $(X_1,\ldots,X_n)$, one can form the infinite exchangeable array $(X_{i_1},X_{i_2},\ldots)$ where the indices $i_j$ are drawn iid from $[n]$.  In~\cite{diaconis_freedman} it is shown that infinite exchangeable arrays constructed in this fashion can approximate a target infinite exchangeable array in the sense that the marginal distributions of subsets of $m$ variables from the two arrays are close in total variation for $m \leq n$.  (We state this result more formally in the sequel.)
%

Observe that different random vectors give rise to the same exchangeable array under the preceding sampling scheme, as sampling with replacement from $(X_1,\ldots,X_n)$ yields the same distribution as sampling from $(X_{\pi(1)},\ldots,X_{\pi(n)})$ for any permutation $\pi\in\msf S_n$. Further, sampling from $(X_1,\ldots,X_n)$ yields the same distribution as sampling from $(X_1,\ldots,X_1,\ldots,X_n,\ldots,X_n)$ where each entry is repeated the same number of times.
Thus, yet another equivalent formulation of the above result from~\cite{diaconis_freedman} is that any infinite exchangeable array can be approximated by those obtained from \emph{equivalence classes} of random vectors via iid sampling of their entries.  We generalize this version of de Finetti's theorem by generalizing the three components of the result, namely, what an infinite exchangeable array is, how we can obtain such arrays by sampling the entries of finite random vectors, and which finite vectors are equivalent to each other in the sense that they yield the same infinite exchangeable arrays in this way.

\paragraph{Generalizing infinite exchangeable arrays:} An infinite exchangeable array $(X_1,X_2,\dots)$ corresponds to a sequence of measures $(\mathrm{Law}(X_1), \mathrm{Law}(X_1,X_2), \dots)$ that project onto each other, with each measure in the sequence being invariant under permutations.
This is analogous to the freely-described elements of Definition~\ref{def:free_elements} in which sequences of invariants project onto each other. 
We similarly define freely-described measures, which generalize infinite exchangeable arrays.
\begin{definition}[Freely-described measures] \label{def:fdm}
    Let $\mscr V$ be a consistent sequence, and let $(\Omega_m \subseteq \vct V_m)$ be a sequence of $\msf S_m$-invariant subsets satisfying $\varphi_{M,m}^\star\Omega_M \subseteq \Omega_m$ for all $m\preceq M$.  The collection of \emph{freely-described probability measures} supported on $(\Omega_m)$ is
    \begin{equation}\label{eq:fd_measures}
        \varprojlim_{\mscr V}\mc P(\Omega_m)^{\msf S_m} = \left\{(\mu_m)\in\prod_m\mc P(\Omega_m)^{\msf S_m}: \varphi_{M,m}^\star \mu_M=\mu_m \textrm{ for all } m\preceq M\right\},
    \end{equation}
    where $\varphi_{M,m}^\star \mu_M$ is the pushforward of $\mu_M$ under $\varphi_{M,m}^\star$. 
\end{definition}

\begin{example}[Freely-described measures over $\mscr Z$]\label{ex:inf_exch_arr_vs_fd_measure}
    The distributions of infinite exchangeable arrays correspond precisely to freely-described measures over $\mscr Z$.  Concretely, a freely-described measure over $\mscr Z$ defines a sequence of $\msf S_m$-invariant random vectors $X^{(m)}=(X_1,\ldots,X_m)$ such that $\zeta_{m+1,m}^\star X^{(m+1)}\overset{d}{=}X^{(m)}$ for all $m$, where $\zeta_{m+1,m}^\star $ extracts the first $m$ entries.  Conversely, if $(X_1,X_2, \allowbreak \ldots)$ is an infinite exchangeable array then the laws of the truncations $X^{(m)}=(X_1,\ldots,X_m)$ form a freely-described measure over $\mscr Z$.


\end{example}

\paragraph{Generalizing sampling with replacement:} Sampling the entries of a vector may be interpreted via maps between finite sets.  For a map $f\colon[m]\to[n]$ between finite sets, consider the associated linear map $\rho(f)\colon\RR^n\to\RR^m$ acting by $(\rho(f)x)_i=x_{f(i)}$.  
If $F_{n,m}\colon[m]\to[n]$ is a uniformly random map between finite sets, then each $i\in[m]$ is mapped independently to a uniformly random element $F_{n,m}(i)\in[n]$, so that $\rho(F_{n,m}) x\in\RR^m$ is a vector whose entries are sampled with replacement from those of $x \in \RR^n$.  
Sampling with replacement from the entries of a random vector $X^{(n)} = (X_1,\dots,X_n)$ then yields the freely-described measure $\left(\mathrm{Law}(\rho(F_{n,m})X^{(n)})\right)_{m\in\NN}$ where $X^{(n)}$ is independent of $F_{n,m}$.  
The key observation to generalize de Finetti's theorem is that the above action $\rho(f)$ of maps between finite sets on $(\RR^n)$ extends to an action on any sequence of vector spaces obtained from $(\RR^n)$ via the standard constructions from Section~\ref{sec:background} as follows. Suppose the sequences of vector spaces $(\vct W_n)$ and $(\vct U_n)$ are endowed with actions $\rho_W$ and $\rho_U$ of maps between finite sets.
\begin{enumerate}[font=\emph, align=left]
    \item[(Products)] For sequences of vector spaces $(\vct W_n), ~ (\vct U_n)$, setting $\rho(f)(w,u)=(\rho_{W}(f)w,\rho_{U}(f)u)$ defines an action on $(\vct W_n\times\vct U_n)$
    \item[(Tensors)] For sequences $(\vct W_n), ~ (\vct U_n)$, setting $\rho(f)(w\otimes u)=(\rho_{W}(f)w)\otimes(\rho_{U}(f)u)$ defines an action on $(\vct W_n\otimes \vct U_n)$, which restricts to symmetric tensors to give an action on $(\mathrm{Sym}^k\vct W_n)$.
    \item[(Polynomials)] For a sequence $(\vct W_n)$, setting $\rho(f)p = p\circ \rho_{W}(f)^\star$ defines an action on $\RR[\vct W_n]_k$.
\end{enumerate}
As an illustration, we get an action on the spaces $((\RR^n)^{\otimes d})$ via $(\rho(f)X)_{i_1,\ldots,i_d}=X_{f(i_1),\ldots,f(i_d)}$.  
In this manner, maps between finite sets yield linear maps between vector spaces constituting any $\mscr Z$- or $\mscr D$-sequence.

\paragraph{Generalizing equivalence under sampling:} 
Finally, fix a random vector $X^{(n)} = (X_1,\dots,X_n)$ and recall that sampling with replacement from any vector in the sequence $(\pi_N\cdot X^{(n)} \otimes \mathbbm{1}_{N/n})_{n|N}$ for any $\pi_N\in\msf S_N$ yields the same infinite exchangeable array.  
Therefore, from the perspective of our sampling map, all measures in the sequence $(\mathrm{Law}(\Pi_N \cdot X^{(n)} \otimes \mathbbm{1}_{N/n}))_{n|N}$ should be viewed as equivalent; here $\Pi_N$ is a uniformly random element of $\msf S_N$ independent of $X^{(n)}$. These sequences of exchangeable measures that are related by embeddings are analogous to freely-symmetrized elements from Definition~\ref{def:free_elements}.  
We similarly define freely-symmetrized measures, which capture the ambiguity in our sampling maps more generally.
\begin{definition}[Freely-symmetrized measures]
    Let $\mscr V$ be a consistent sequence, and let $(\Omega_n\subseteq \vct V_n)$ be a sequence of $\msf S_n$-invariant subsets satisfying $\varphi_{N,n}\Omega_n\subseteq\Omega_N$ for all $n\preceq N$. The collection of \emph{freely-symmetrized probability measures} over $(\Omega_n)$ is
    \begin{equation}\label{eq:fs_measures}
        \varinjlim_{\mscr V}\mc P(\Omega_n)^{\msf S_n} = \bigsqcup_n \mc P(\Omega_n)^{\msf S_n}/\sim \quad \textrm{where } \mu_n\sim\mu_N \textrm{ if } \mu_N=\mathrm{sym}_N^{\mscr V}\mu_n.
    \end{equation}
    A freely-symmetrized measure is one of these equivalence classes, denoted as a sequence $\allowbreak(\mathrm{sym}_N^{\mscr V}\mu_n)_{N\succeq n}$ for $\mu_n\in\mc P(\Omega_n)^{\msf S_n}$.
\end{definition}
\begin{example}[Freely-symmetrized measures over $\mscr D$]
    The exchangeable random vectors yielding the same infinite exchangeable array under sampling with replacement correspond to freely-symmetrized measures over $\mscr D$. Indeed, a freely-symmetrized measure over $\mscr D$ is the equivalence class of $\Big(\mathrm{sym}_N^{\mscr D}\mathrm{Law}(X^{(n)})=\mathrm{Law}(\Pi_NX^{(n)}\otimes\mathbbm{1}_{N/n})\Big)_{n|N}$ with $\Pi_N$ a uniformly random element of $\msf S_N$ independent of $X^{(n)}$. 
\end{example}

With these notions in hand, we can restate de Finetti's theorem in terms of free measures.  Specifically, consider the following map from freely-symmetrized measures $\varinjlim_{\mscr D}\mc P(\RR^n)^{\msf S_n}$ over $\mscr D$ to freely-described measures $\varprojlim_{\mscr Z}\mc P(\RR^m)^{\msf S_m}$ over $\mscr Z$:
\begin{equation}\label{eq:DeFin_Rcase}
    \left(\mathrm{sym}_N^{\mscr D}\mu_n\right)_{n|N} \mapsto \Big(\mathrm{Law}(\rho(F_{n,m})X)\Big)_m,
\end{equation}
where $X \sim \mu_n$ is independent of the uniformly random map $F_{n,m}\colon[m]\to[n]$.  
Diaconis and Freedman~\cite{diaconis_freedman} prove the error bound $\|\mathrm{Law}(\zeta_{n,m}^\star X)-\mathrm{Law}(\rho(F_{n,m})X)\|_{\mathrm{TV}}\leq \frac{m(m-1)}{n}$ whenever $m\leq n$, which holds for any exchangeable random vector $X$ in $\RR^n$. Here $\zeta_{n,m}^\star X$ is the vector of the first $m$ coordinates of $X$, while $\rho(F_{n,m})X$ is a vector of $m$ entries sampled with replacement from the coordinates of $X$, and their distributions are compared in total variation norm $\|\mu-\nu\|_{\mathrm{TV}}=\sum_{x\in\Omega}|\mu(x)-\nu(x)| = \sup_{f\colon\Omega\to[-1,1]}|\mbb E_{\mu}f - \mbb E_{\nu}f|$ for probability distributions $\mu,\nu$ on a finite set $\Omega$.
From this bound, one can conclude that the map~\eqref{eq:DeFin_Rcase} has a dense image in $\varprojlim_{\mscr Z}\mc P(\RR^m)^{\msf S_m}$ in total variation, meaning that if $(\mu_m)$ is a freely-described measure over $\mscr Z$, then there exists a sequence of freely-described measures $(\mu_m^{(n)})$ in the image of the map~\eqref{eq:DeFin_Rcase} such that $\lim_{n\to\infty}\|\mu_m-\mu_m^{(n)}\|_{\mathrm{TV}}=0$ for each fixed $m$.  
That is because for each $n\geq m$ and $X\sim\mu_n$ we have $\mu_m=\mathrm{Law}(\zeta_{n,m}^\star X)$ since $(\mu_m)$ is freely-described, while the image of $\left(\mathrm{sym}_N^{\mscr D}\mu_n\right)_{n|N}$ under the map~\eqref{eq:DeFin_Rcase} is $(\mu_m^{(n)}=\mathrm{Law}(\rho(F_{n,m})X))_m$, so $\|\mu_m - \mu_m^{(n)}\|_{\mathrm{TV}}\to0$ as $n\to\infty$ by the above error bound.
To obtain de Finetti's characterization of \emph{all} freely-described measures on $\mscr Z$, not just the dense subset above, one can use a simple compactness argument~\cite[Thm.~14]{diaconis_freedman}.


For more general consistent sequences, the characterization of all freely-described measures is more complicated~\cite[Chap.~7]{kallenberg2005probabilistic}.  
The map~\eqref{eq:DeFin_Rcase} does, however, readily generalize and yields simple dense subsets of freely-described measures, as we show in the following theorem.
Fortunately, these dense subsets, together with error bounds generalizing those of~\cite{diaconis_freedman}, are all we require to construct finite-dimensional POP lower bounds and obtain their convergence rates (see Section~\ref{sec:free_poly_opts}).

\begin{theorem}\label{thm:DeFin_general_Zseq}
    Let $\mc F(\mscr Z) = \{(\NN,\leq),(\msf S_n),(\vct V_n),(\varphi_{N,n})\}$ be a $\mscr Z$-sequence, and let $\mc F(\mscr D) = \{(\NN,\cdot|\cdot), \allowbreak (\msf S_n), \allowbreak (\vct V_n), \allowbreak (\psi_{N,n})\}$ be the corresponding $\mscr D$-sequence. 
    Let $(\Omega_n\subseteq\vct V_n)$ be a sequence of $\msf S_n$-invariant sets satisfying $\varphi_{N,n}^\star\Omega_N\subseteq\Omega_n$ whenever $n\leq N$ and $\psi_{N,n}\Omega_n\subseteq\Omega_N$ whenever $n|N$. Consider the map from freely-symmetrized measures $\varinjlim_{\mc F(\mscr D)}\mc P(\Omega_n)^{\msf S_n}$ over $\mc F(\mscr D)$ to freely-described measures $\varprojlim_{\mc F(\mscr Z)}\mc P(\Omega_m)^{\msf S_m}$ over $\mc F(\mscr Z)$ sending
    \begin{equation}\label{eq:DeFin_gen_map}
        \left(\mathrm{sym}_N^{\mc F(\mscr D)}\mu_n\right)_{n|N} \mapsto \Big(\mathrm{Law}(\rho(F_{n,m})X)\Big)_m,
    \end{equation}
    where $X \sim \mu_n$ is independent of the uniformly random map $F_{n,m}\colon[m]\to[n]$.  This map has a dense image in total variation, meaning that for any freely-described measure $(\mu_m) \in \varprojlim_{\mc F(\mscr Z)}\mc P(\Omega_m)^{\msf S_m}$ there exists $(\mu_m^{(n)})$ in the image of \eqref{eq:DeFin_gen_map} such that $\lim_{n\to\infty}\|\mu_m-\mu_m^{(n)}\|_{\mathrm{TV}}=0$ for each $m$.  
    More precisely, for any $\mu_n\in\mc P(\vct V_n)^{\msf S_n}$ and any $m\leq n$ we have $\|\mathrm{Law}(\varphi_{n,m}^\star X) - \mathrm{Law}(\rho(F_{n,m})X)\|_{\mathrm{TV}}\leq \frac{m(m-1)}{n}$ where $X \sim\mu_n$ is independent of $F_{n,m}$.
    

    
\end{theorem}

Note that the error bound above holds for any exchangeable measure on any fixed $\vct V_n$, not necessarily compactly supported.  The proof of this result is provided at the end of Section~\ref{sec:deFinetti_proofs} as a consequence of the more general Theorem~\ref{thm:DeFin_for_coFS}.  We instantiate Theorem~\ref{thm:DeFin_general_Zseq} in the setting of the original de Finetti's theorem as well as for two-dimensional arrays.
\begin{example}[Classic de Finetti]
With $\mc F=\mathrm{id}$ and $\Omega_n=\{0,1\}^n$ in Theorem~\ref{thm:DeFin_general_Zseq}, the set $\varprojlim_{\mscr Z}\mc P(\Omega_m)^{\msf S_m}$ corresponds to infinite binary exchangeable arrays, the original setting studied by de Finetti~\cite{de1929funzione}. 
Theorem~\ref{thm:DeFin_general_Zseq} states that a dense subset of such arrays is given by starting from a random exchangeable vector $X^{(n)}\in\{0,1\}^n$ and sampling its entries with replacement to get $(X_1,X_2,\ldots)$; equivalently, if $P^{(n)}\in\{0,1/n,2/n,\ldots,1\}$ is the (random) fraction of the entries of $X^{(n)}$ which equal 1, we can first sample $P^{(n)}$ and then sample $X_j$ iid from $\mathrm{Ber}(P^{(n)})$.  In comparison, de Finetti proved that all infinite binary exchangeable arrays are obtained by sampling $P\in[0,1]$ according to some distribution on $[0,1]$, and then sampling $X_j$ iid from $\mathrm{Ber}(P)$. 
Thus, Theorem~\ref{thm:DeFin_general_Zseq} yields the dense subset of these arrays for which the random $P$ is supported on $\{0,1/n,2/n,\ldots,1\}$ for some $n\in\NN$.
\end{example}


\begin{example}[Graphons and Aldous--Hoover]\label{ex:graphon_deFin} 
Let $\mc F=\mathrm{Sym}^2$ in Theorem~\ref{thm:DeFin_general_Zseq}, so $\mc F(\mscr Z)=\mathrm{Sym}^2\mscr Z$ and $\mc F(\mscr D)=\mathrm{Sym}^2\mscr D$ are the consistent sequences from Example~\ref{ex:Rdk} consisting of symmetric matrices on which permutations act by simultaneously permuting rows and columns, and with zero-padding and duplication embeddings.  Next, let $\Omega_n=\{X\in\mbb S^n(\{0,1\}):\diag(X)=0\}$ be the set of adjacency matrices of simple graphs.  An exchangeable measure on $\Omega_n$ is a random graph on $n$ vertices.  Freely-described measures on $(\Omega_n)$ are then a sequence of random graphs on increasingly many vertices closed under taking induced subgraphs, or equivalently, random graphs on countably many vertices.  Theorem~\ref{thm:DeFin_general_Zseq} states that a dense subset of such sequences of measures is obtained by starting from a fixed-dimensional random $X\in \Omega_n$ and sampling its row and column indices $i_1,i_2,\ldots$ iid from $[n]$ to obtain the infinite array $(X_{i_j,i_{j'}})_{j,j'\in\NN}$.  Equivalently, define the step-function $W_X\colon[0,1]^2\to\{0,1\}$ by first partitioning $[0,1]=\bigcup_{i=1}^nI_{i,n}$ with $I_{i,n}=[(i-1)/n,i/n)$ if $i<n$ and $I_{n,n}=[1-1/n,1]$, and setting $W_X(x,y)=X_{i,j}$ if $(x,y)\in I_{i,n}\times I_{j,n}$.  Then sampling $x_1,x_2,\ldots$ iid uniformly on $[0,1]$ and setting $X_{j,j'}=W_X(x_j,x_{j'})$ yields an infinite array with the same distribution as before.



To obtain \emph{all} freely-described measures on $(\Omega_n)$, it was shown in~\cite{lovasz2012random} that any countable random graph model $(X_{j,j'})$ is obtained from a random symmetric measurable function $W\colon[0,1]^2\to[0,1]$ called a \emph{graphon} by first sampling $x_1,x_2,\ldots$ iid uniformly from $[0,1]$ and then, conditionally on $W$ and on $(x_j)$, sampling $X_{j,j'}\sim\mathrm{Ber}(W(x_j,x_{j'}))$ independently for $j < j'$, while setting $X_{j,j'}=X_{j',j}$ for $j>j'$ and $X_{j,j}=0$. 
Thus, Theorem~\ref{thm:DeFin_general_Zseq} yields the dense subset of such arrays that can be obtained from \emph{step graphons} of the form $W_X$ for some $X\in\Omega_k$.

For more general sequences of sets $(\Omega_n\subseteq \mbb S^n)$, corresponding to weighted graphs, the characterization of all freely-described measures on $(\Omega_n)$ is even more involved, and requires random functions defined on $[0,1]^2\times[0,1]^{\binom{\NN}{2}}$, as a consequence of Aldous--Hoover--Kallenberg theory~\cite[Chap.~7]{kallenberg2005probabilistic}. Nevertheless, freely-described measures obtained from step graphons as above continue to be dense by Theorem~\ref{thm:DeFin_general_Zseq}, and the error bound there still applies.
\end{example}
Theorem~\ref{thm:DeFin_general_Zseq} gives a dense subset of freely-described measures over $\mscr Z$-sequences in terms of freely-symmetrized measures over $\mscr D$-sequences. This result allows us to construct lower bounds for freely-described problems over $\mscr D$-sequences in terms of freely-symmetrized problems over $\mscr Z$-sequences, generalizing the example in Section~\ref{sec:case_study}.  
To construct bounds for freely-described problems over $\mscr Z$-sequences, however, we need a dual de Finetti theorem.  
In more detail, we require a dense subset of freely-described measures over $\mscr D$-sequences in terms of freely-symmetrized measures over the corresponding $\mscr Z$-sequences.  
\begin{example}[Freely-described measures over $\mscr D$]\label{ex:P_fd_measures}
    A freely-described measure over $\mscr D$ corresponds to the laws of a sequence of exchangeable random vectors $(X^{(m)}\in\RR^m)_m$ satisfying
    \begin{equation}\label{eq:fd_measure_P}
        \left(\sum_{j=1}^{M/m}X^{(M)}_{(i-1)(M/m)+j}\right)_{i=1}^m \overset{d}{=} X^{(m)},\quad \textrm{for all } m|M.
    \end{equation}
    Examples of such sequences of random variables include the uniformly random coordinate vectors $X^{(m)}\sim\mathrm{Unif}\{e_1,\ldots,e_m\}$, the iid Gaussians $X^{(m)}\sim\mc N(0,\frac{1}{m}I_m)$, and the Dirichlet distributions $X^{(m)}\sim\mathrm{Dir}(\mathbbm{1}_m/m)$.
\end{example}
We obtain the desired `dual' version of Theorem~\ref{thm:DeFin_general_Zseq} by considering the adjoint of the action $\rho$ of maps between finite sets.
Specifically, for any map $f\colon[n]\to[m]$ we get the linear map $\rho(f)^\star\colon\RR^n\to\RR^m$ defined by $(\rho(f)^\star x)_i=\sum_{j\in f^{-1}(i)}x_j$. This extends to an action between the vector spaces constituting any $\mscr Z$- or $\mscr D$-sequence as before.
Considering the adjoint action of a uniformly random map, we obtain the following theorem. Below we endow each $\vct V_n$ with the 1-norm with respect to a particular basis constructed from the standard basis for $\RR^n$ (see Lemma~\ref{lem:ell1_ext} below).
\begin{theorem}\label{thm:dual_deFin_Pseq}
Let $\mc F(\mscr D) = \{(\NN,\cdot|\cdot), \allowbreak (\msf S_n), \allowbreak (\vct V_n), \allowbreak (\psi_{N,n})\}$ be a $\mscr D$-sequence of degree $d_{\mc F}$, and let $\mc F(\mscr Z) = \{(\NN,\leq), \allowbreak (\msf S_n), (\vct V_n),(\varphi_{N,n})\}$ be the corresponding $\mscr Z$-sequence.  Let $(\Omega_n\subseteq \vct V_n)$ be a sequence of compact $\msf S_n$-invariant sets satisfying $\psi_{N,n}^\star \Omega_N\subseteq \Omega_n$ whenever $n|N$, $\varphi_{N,n}\Omega_n\subseteq \Omega_N$ whenever $n\leq N$, and $\sup_n\sup_{x\in\Omega_n}\|x\|_1<\infty$.  
Consider the map from freely-symmetrized measures $\varinjlim_{\mc F(\mscr Z)}\mc P(\Omega_n)^{\msf S_n}$ over $\mc F(\mscr Z)$ to freely-described measures $\varprojlim_{\mc F(\mscr D)}\mc P(\Omega_m)^{\msf S_m}$ over $\mc F(\mscr D)$ sending
\begin{equation}\label{eq:dual_deFin_map}
    \left(\mathrm{sym}_N^{\mc F(\mscr Z)}\mu_n\right)_{N\geq n} \mapsto \Big(\mathrm{Law}(\rho(F_{m,n})^\star X)\Big)_m,
\end{equation}
where $X\sim\mu_n$ is independent of the uniformly random map $F_{m,n}\colon[n]\to[m]$. 
This map has a dense image in the Wasserstein distance $W_1$ with respect to an induced $\|\cdot\|_1$ norm on $(\vct V_n)$.  More precisely, for any $\mu_n\in\mc P(\vct V_n)^{\msf S_n}$ and $m|n$ we have
\begin{equation*}
        W_1\Big(\mathrm{Law}(\psi_{n,m}^\star X), \mathrm{Law}(\rho(F_{m,n})^\star X)\Big)\leq 2d_{\mc F}\sqrt{\frac{m(m-1)}{n/m}}\mbb E_{\mu_n}[\|X\|_{1}],
\end{equation*}
and
\begin{equation*}
    \|\mbb E\psi_{n,m}^\star X - \mbb E\rho(F_{m,n})^\star X\|_1\leq \frac{d_{\mc F}(d_{\mc F}-1)}{n}\mbb E_{\mu_n}\|X\|_1,
\end{equation*}
where $X \sim \mu_n$ is independent of $F_{m,n}$. 
\end{theorem}
Note that the error bounds in Theorem~\ref{thm:dual_deFin_Pseq} hold for any exchangeable measure on $\vct V_n$, but are stated in Wasserstein-1 distance.
Note also that we separately state an $O(1/\sqrt{n})$ bound between the full distributions of $\psi_{n,m}^\star X$ and $\rho(F_{m,n})^\star X$, and an $O(1/n)$ bound between their means. We will see in Section~\ref{sec:free_sym_bds} that the latter bound suffices to obtain tight convergence rates for our relaxations of polynomial POPs. Nevertheless, the former bound may be of independent interest, and we show in Remark~\ref{rmk:non_poly} that it gives tight rates of convergence for some non-polynomial optimization problems.
In contrast, the error bounds in Theorem~\ref{thm:DeFin_general_Zseq} are given in the larger total variation distance, and decay to zero at the rate of $1/n$, faster than the above $1/\sqrt{n}$.
We show in Sections~\ref{sec:deFinetti_coFS} and~\ref{sec:deFinetti_FS} that the use of Wasserstein-1 here is necessary, and that all the above rates are tight.  The proof of Theorem~\ref{thm:dual_deFin_Pseq} is provided at the end of Section~\ref{sec:deFinetti_proofs} as a consequence of the more general Theorem~\ref{thm:dual_deFin_general}.  We illustrate Theorem~\ref{thm:dual_deFin_Pseq} on two examples, one corresponding to the freely-described measures on $\mscr D$ from Example~\ref{ex:P_fd_measures} and the other corresponding to a dual version of Aldous--Hoover.
\begin{example}[Freely-described measures on $\mscr D$, Example~\ref{ex:P_fd_measures} continued]
    In the setting of Example~\ref{ex:P_fd_measures}, Theorem~\ref{thm:dual_deFin_Pseq} states that any sequence of random vectors satisfying~\eqref{eq:fd_measure_P} and supported on simplices $\Omega_n=\Delta^{n-1}$ can be approximated arbitrarily well by those of the form $X^{(m)}=\rho(F_{m,n})^\star(X)=\sum_{i=1}^nX_ie_{j_i}$ where $X\in\Delta^{n-1}$ is a fixed random vector supported on the simplex and $j_1,\ldots,j_n$ are sampled iid uniformly from $[m]$ (independently of $X$).
    That is, the right notion of sampling in this setting puts each coordinate of $x\in\RR^n$ independently into one of $m$ bins uniformly at random, and each coordinate of $\rho(F_{m,n})^\star x$ is obtained by summing the coordinates of $x$ in the corresponding bin.  
    While Theorem~\ref{thm:dual_deFin_Pseq} only yields a dense subset of freely-described measures over $\mscr D$ supported on $(\Delta^{n-1})$, a full characterization of all such measures can be obtained in terms of random exchangeable measures on $[0,1]$ using~\cite{kallenberg1990exchangeable, orbanz2011projective}, see~\cite{levin2025limits}. 
\end{example}
\begin{example}[Freely-described measures on $\mathrm{Sym}^2\mscr D$]
    Setting $\mc F=\mathrm{Sym}^2$ and $\Omega_n=\{X\in\mbb S^n:X_{i,j}\geq0,\ \sum_{i\leq j}X_{i,j}=1\}$ in Theorem~\ref{thm:dual_deFin_Pseq}, freely-described measures over $\mathrm{Sym}^2\mscr D$ can be approximated arbitrarily well by the following construction.  A finite random graph $X\in\Omega_n$ defines a freely-described measure over $\mathrm{Sym}^2\mscr D$ by defining the sequence of random graphs $(X^{(m)})$ via $X^{(m)}=\sum_{1\leq i\leq j\leq n}X_{i,j}E_{k_i,k_j}$ where $E_{i,j}=e_ie_j^\top + e_je_i^\top$ if $i\neq j$ while $E_{i,i}=e_ie_i^\top$, and $k_1,\ldots,k_n$ are drawn iid uniformly from $[m]$. Equivalently, $X^{(m)}$ is the random graph obtained by labeling the $n$ vertices of $X$ by uniformly random labels from $[m]$ and identifying vertices with the same labels (which might then result in self-loops and multiple edges).  
    As before, Theorem~\ref{thm:dual_deFin_Pseq} only yields a dense subset of freely-described measures over $\mathrm{Sym}^2\mscr D$ supported on $(\Omega_n)$, and a full characterization in terms of random exchangeable measures on $[0,1]^2$ can be obtained using~\cite{kallenberg1990exchangeable, orbanz2011projective}, see~\cite{levin2025limits}.
\end{example}


The key insight underlying the proofs of Theorems~\ref{thm:DeFin_general_Zseq} and~\ref{thm:dual_deFin_Pseq} is a link between actions of maps between finite sets and consistent sequences.  Specifically, the above actions of maps between finite sets recover all of the maps relating vector spaces in different dimensions that we have seen so far (e.g., zero-padding and duplication), and as a consequence, yield all $\mscr Z$- and $\mscr D$-sequences.  Leveraging this link, we reduce the proofs of the above theorems to proving bounds on random maps between finite sets.


\subsection{Freely-Described Measures Over \texorpdfstring{$\mscr Z$}{Z}-Sequences}\label{sec:deFinetti_coFS}

The aim of this section is to present a generalization of de Finetti's theorem that characterizes (a dense subset of) freely-described measures over $\mscr Z$-sequences. 
We begin by abstracting away the properties of the action $\rho$ of maps between finite sets on $(\RR^n)$ from Section~\ref{sec:deFinetti_intro}.  Note that if $f_1\colon[n]\to[m]$ and $f_2\colon[m]\to[k]$, then $\rho(f_2\circ f_1) =\rho(f_1) \circ \rho(f_2)$ and $\rho(\mathrm{id}_{[n]}) =\mathrm{id}_{\RR^n}$.  We call such an assignment of linear maps to maps between finite sets a $\coFS$-action, where `op' stands for `opposite' and pertains to the opposite order in which compositions act.\footnote{In fact, a $\coFS$-action precisely corresponds to a representation of the category $\coFS$, the opposite of the category $\FS$ whose objects are finite sets and whose morphisms are all maps between them.} 
\begin{definition}[$\coFS$-representation]\label{def:coFS_mod}
    A \emph{$\coFS$-action} on a sequence of vector spaces $(\vct V_n)_{n\in\NN}$ is an assignment of a linear map $\tau^{\mathrm{op}}(f)\colon\vct V_m\to\vct V_n$ for each map $f\colon[n]\to[m]$ between finite sets satisfying $\tau^{\mathrm{op}}(\mathrm{id}_{[n]})=\mathrm{id}_{\vct V_n}$ for all $n\in\NN$ and $\tau^{\mathrm{op}}(f_1\circ f_2)=\tau^{\mathrm{op}}(f_2)\circ \tau^{\mathrm{op}}(f_1)$ whenever the composition is well-defined.
    A sequence of vector spaces endowed with a $\coFS$-action $\mathtt{V}=\{(\vct V_n),(\tau^{\mathrm{op}})\}$ is called a \emph{$\coFS$-representation}.
\end{definition}
We denote the standard $\coFS$ representation defined earlier on $(\RR^n)$ by $\mathtt{R} = \{(\RR^n),(\rho)\}$.  Next, we illustrate how $\coFS$-representations yield consistent sequences.  Starting with the representation $\mathtt{R}$, we observe that specific choices of maps $f$ between finite sets yield many of the maps we have been using in our development.
\begin{enumerate}[font=\emph, align=left]
    \item[(Permutations)] If $\sigma\colon[n]\to[n]$ is a permutation, then $\rho(\sigma) =\sigma^{-1}$ gives the action of $\sigma^{-1}$ on $\RR^n$.
    \item[(Zero-padding)] If $\iota_{N,n}\colon[n]\hookrightarrow[N]$ is inclusion, then $\rho(\iota_{N,n}) =\zeta_{N,n}^\star $ is the projection in $\mscr Z$ extracting the first $n$ entries of a length-$N$ vector.
    \item[(Duplication)] If $d_{n,nk}\colon[nk]\to[n]$ is the standard equipartition $d_{n,nk}(k(i-1)+j) = i$ for $i\in[n], j\in[k]$, then $\rho(d_{n,nk}) =\delta_{nk,n}$ is the duplication embedding in $\mscr D$.
\end{enumerate}
Consequently, the consistent sequences $\mscr Z$ and $\mscr D$ can be derived from $\mathtt{R}$ as follows:
\begin{equation*}
\begin{aligned}
    \mscr Z = \{(\NN,\leq),(\msf S_n)_{\mscr Z},(\vct R_n),(\rho(\iota_{N,n})^\star)_{n\leq N}\} && \textrm{and}&& \mscr D = \{(\NN,\cdot|\cdot),(\msf S_n)_{\mscr D},(\vct R_n),(\rho(d_{n,N}))_{n|N}\},
\end{aligned}
\end{equation*}
where $(\msf S_n)_{\mscr Z}$ and $(\msf S_n)_{\mscr D}$ denote the chains of symmetric groups from~\eqref{eq:zero-padding} and~\eqref{eq:duplicating}, respectively.
We can generalize this observation to obtain $\mscr Z$- and $\mscr D$-sequences from any $\coFS$-representation.  Specifically, given a $\coFS$-representation $\mathtt V=\{(\vct V_n),(\tau^{\mathrm{op}})\}$, for each $n\in\NN$ we endow $\vct V_n$ with the $\msf S_n$ action $\sigma\cdot x = \tau^{\mathrm{op}}(\sigma^{-1})x$ for a permutation $\sigma\colon[n]\to[n]$ and $x\in\vct V_n$.  We then define the two consistent sequences:
\begin{equation}\label{eq:coFS_consistent_sequences}
\begin{aligned}
        \mscr Z(\mathtt V) = \{(\NN,\leq),(\msf S_n)_{\mscr Z},(\vct V_n),(\tau^{\mathrm{op}}(\iota_{N,n})^\star )_{n\leq N}\}\ \textrm{ and }\ \mscr D(\mathtt V) = \{(\NN,\cdot|\cdot),(\msf S_n)_{\mscr D},(\vct V_n),(\tau^{\mathrm{op}}(d_{n,N}))_{n|N}\}.
\end{aligned}
\end{equation}
Every $\mscr Z$- and $\mscr D$-sequence can be obtained in this way for some $\coFS$-representation.  To see this, recall that the standard constructions of Section~\ref{sec:consist_seqs}, which can be applied to $\mscr Z$ and $\mscr D$ to obtain more complicated consistent sequences, can also be applied to $\mathtt{R}$ as discussed in Section~\ref{sec:deFinetti_intro} to obtain more complicated $\coFS$-representations.  In particular, for any standard construction $\mathcal{F}$, we observe that $\mathcal{F}(\mathtt{R})$ is also a $\coFS$-representation.\footnote{We denote by the same $\mc F$ standard constructions applied to consistent sequences and to $\coFS$-representations. Formally, consistent sequences and $\coFS$-representations are functors from some category to the category of vector spaces, and we are post-composing them with a functor $\mc F$ from the category of vector spaces to itself. Informally, we can apply standard constructions $\mc F$ to any collection of vector spaces and linear maps between them.} 
Moreover, one can check that $\mscr Z(\mc F(\mathtt{R}))=\mc F(\mscr Z)$ and $\mscr D(\mc F(\mathtt{R}))=\mc F(\mscr D)$. That is, we get the same consistent sequences from $\mc F$ whether we apply it directly to $\mscr Z$ and $\mscr D$, or apply it to the $\coFS$-representation $\mathtt{R}$ and invoke \eqref{eq:coFS_consistent_sequences}.

We proceed to generalize de Finetti's theorem to free measures over consistent sequences of the form \eqref{eq:coFS_consistent_sequences} derived from any $\coFS$-representation.  The proofs of our results in this section are presented in Section~\ref{sec:deFinetti_proofs} to make the architecture of our argument more transparent.  We begin with a lemma characterizing the invariances of uniformly random maps. It will be used to argue that our sampling maps send equivalent freely-symmetrized measures to the same image, and that these images are indeed freely-described measures.
\begin{lemma}\label{lem:unif_map_invars}
    Let $F_{m,n}\colon[n]\to[m]$ be uniformly random maps for each $m,n\in\NN$. For any injective map $\phi_{n,k}\colon[k]\to[n]$ we have $F_{m,n}\circ\phi_{n,k}\overset{d}{=}F_{m,k}$, and for any equipartition $\Psi_{k,nk}\colon[nk]\to[k]$ we have $\Psi_{k,nk}\circ F_{nk,m}\overset{d}{=}F_{k,m}$ for all $k,m,n\in\NN$.
\end{lemma}
Next, we state the key technical lemma used in the proof of Theorem~\ref{thm:DeFin_for_coFS}, which is the quantitative version of the asymptotic equivalence of sampling with and without replacement from~\cite{sampling_tv_dist}, stated in our formalism.  
\begin{lemma}\label{lem:sampling}
    For $m\leq n$, let $F_{n,m}\colon[m]\to[n]$ be a uniformly random map, and let $\Phi_{n,m}\colon[m]\hookrightarrow [n]$ be a uniformly random \emph{injective} map. Then
    \begin{equation}\label{eq:sampling_bd}
        \|\mathrm{Law}(F_{n,m}) - \mathrm{Law}(\Phi_{n,m})\|_{\mathrm{TV}}\leq \frac{m(m-1)}{n}.
    \end{equation}
\end{lemma}
The above total variation bound applies to any pushforward of the random maps $F_{n,m}$ and $\Phi_{n,m}$, and in particular applies to their linear representatives acting on a $\coFS$-representation, so that $\|\mathrm{Law}(\tau^{\mathrm{op}}(F_{n,m}))-\mathrm{Law}(\tau^{\mathrm{op}}(\Phi_{n,m}))\|_{\mathrm{TV}}\leq\frac{m(m-1)}{n}$.  This allows us to prove the following theorem, generalizing de Finetti to consistent sequences derived from any $\coFS$-representation.
\begin{theorem}[De Finetti for $\coFS$-representations]\label{thm:DeFin_for_coFS}
    Let $\mathtt V=\{(\vct V_n),(\tau^{\mathrm{op}})\}$ be a $\coFS$ representation, and let $(\Omega_n)$ satisfy $\tau^{\mathrm{op}}(f)\Omega_n\subseteq \Omega_m$ for all $f\colon[m]\to[n]$. Consider the map from $\varinjlim_{\mscr D(\mathtt V)}\mc P(\Omega_n)^{\msf S_n}$ to $\varprojlim_{\mscr Z(\mathtt V)}\mc P(\Omega_m)^{\msf S_m}$ sending
    \begin{equation}
        \left(\mathrm{sym}_N^{\mscr D(\mathtt V)}\mu_n\right)_{n|N}\mapsto \Big(\mathrm{Law}(\tau^{\mathrm{op}}(F_{n,m})X)\Big)_m,
    \end{equation}
    where $X\sim\mu_n$ is independent of the uniformly random map $F_{n,m}\colon[m]\to[n]$. This map has a dense image in total variation. 
    More precisely, for any $\mu_n\in\mc P(\vct V_n)^{\msf S_n}$ and $m\leq n$ we have 
    \begin{equation*}
        \|\mathrm{Law}(\tau^{\mathrm{op}}(\iota_{n,m})X) - \mathrm{Law}(\tau^{\mathrm{op}}(F_{n,m})X)\|_{\mathrm{TV}}\leq \frac{m(m-1)}{n},
    \end{equation*}
    where $X\sim\mu_n$ is independent of $F_{n,m}$.
\end{theorem}
The condition that $\tau^{\mathrm{op}}(f)\Omega_n\subseteq \Omega_m$ for all $f\colon[m]\to[n]$ in Theorem~\ref{thm:DeFin_for_coFS} is equivalent to the sets $(\Omega_n)$ satisfying $(i)$ each $\Omega_n$ is $\msf S_n$-invariant; $(ii)$ $\tau^{\mathrm{op}}(\iota_{N,n})\Omega_N\subseteq\Omega_n$ for $n\leq N$; and $(iii)$ $\tau^{\mathrm{op}}(d_{n,N})\Omega_n\subseteq\Omega_N$ for $n|N$.  In other words, such sequences of sets $(\Omega_n)$ are precisely the ones which are group-invariant, embed into each other in $\mscr D(\mathtt V)$, and project into each other in $\mscr Z(\mathtt V)$.  This is a consequence of the following lemma, showing that any map between finite sets factorizes into an injection and an equipartition.
\begin{lemma}[Factorization of maps between finite sets]\label{lem:factorization_of_maps}
    For any map $f\colon[m]\to[n]$ and any $k\geq \max_{i\in[n]}|f^{-1}(i)|$, there exists $g\in\msf S_{nk}$ satisfying $f = d_{n,nk}g\iota_{nk,m}$.
\end{lemma}

Finally, we conclude this section with an example illustrating that the error bounds obtained in Theorem~\ref{thm:DeFin_for_coFS} are tight.

\begin{example}[Tightness of rates in Theorem~\ref{thm:DeFin_for_coFS}] To see that the $1/n$ rate in Theorem~\ref{thm:DeFin_for_coFS} is tight, let $X \sim\mathrm{Ber}(1/2)^{\otimes n}$, and note that $\zeta_{n,2}^\star X \sim\mathrm{Ber}(1/2)^{\otimes 2}$ while $\rho(F_{n,2})X \sim(1-1/n)\mathrm{Ber}(1/2)^{\otimes 2} + (1/n)\mathrm{Unif}\{(0,0)^\top,(1,1)^\top\}$, since with probability $1-1/n$ the map $\rho(F_{n,2})$ picks out two distinct coordinates from $X$, while with probability $1/n$ the two coordinates are equal and equally likely to be 0 or 1. The total variation distance between the two laws is therefore $1/n$.
\end{example}



\subsection{Freely-Described Measures Over \texorpdfstring{$\mscr D$}{D}-Sequences}\label{sec:deFinetti_FS}
In this section we present our dual de Finetti theorem giving a dense subset of freely-described measures over $\mscr D$-sequences in terms of freely-symmetrized measures over $\mscr Z$-sequences.  To this end, we consider the adjoint action $\rho(f)^\star$ of maps between finite sets.  This adjoint action preserves the order of compositions, so $\rho(f_1\circ f_2)^\star=\rho(f_1)^\star\circ \rho(f_2)^\star$ and once again we have $\rho(\mathrm{id}_{[n]})^\star=\mathrm{id}_{\vct V_n}$. We call such an assignment of linear maps to maps between finite sets a $\FS$-representation.
\begin{definition}[$\FS$-representation] \label{def:FS_mod}
    A \emph{$\FS$-action} on a sequence of vector spaces $(\vct V_n)_{n\in\NN}$ is an assignment of a linear map $\tau(f)\colon\vct V_n\to\vct V_m$ for each map $f\colon[n]\to[m]$ between finite sets satisfying $\tau(\mathrm{id}_{[n]})=\mathrm{id}_{\vct V_n}$ and $\tau(f_1\circ f_2)=\tau(f_1)\circ \tau(f_2)$ whenever the composition is well-defined. A sequence of vector spaces endowed with a $\FS$-action $\mathtt V=\{(\vct V_n),(\tau)\}$ is called a \emph{$\FS$-representation}.
\end{definition}
As the adjoint $\rho^\star$ of the standard $\coFS$-action $\rho$ is a $\FS$-action, we denote the standard $\FS$-representation on $(\RR^n)$ from Section~\ref{sec:deFinetti_intro} by $\mathtt R^\star = \{(\RR^n),(\rho^\star)\}$.  More broadly, note that if $\mathtt V = \{(\vct V_n), (\tau)\}$ is a $\FS$-representation, then $\mathtt V^\star = \{(\vct V_n), (\tau^\star)\}$ is a $\coFS$-representation (and vice versa), where $\tau^\star$ is the adjoint of $\tau$. 

We proceed to state our dual de Finetti theorem. The architecture of the proof is similar to Section~\ref{sec:deFinetti_coFS}, which we again highlight by deferring the proofs and some technical lemmas to Section~\ref{sec:deFinetti_proofs} below.
The following is the key result we will need, which is the analog of Lemma~\ref{lem:sampling}.
\begin{lemma}\label{lem:W1_dist}
    Let $F_{m,n}\colon[n]\to[m]$ be a uniformly random map and let $\Psi_{m,n}\colon[n]\to[m]$ be a uniformly random equipartition, for any $m|n$. Then
    \begin{equation}
        W_1(\mathrm{Law}(F_{m,n}), \mathrm{Law}(\Psi_{m,n})) = \inf_{\substack{(F,\Psi) \textrm{ random}\\F\overset{d}{=} F_{m,n},\ \Psi\overset{d}{=}\Psi_{m,n}}}\mbb E\left[\frac{|F\neq\Psi|}{n}\right] \leq \sqrt{\frac{m(m-1)}{n/m}},
    \end{equation}
    where $|F\neq\Psi|=|\{i\in[n]:F(i)\neq\Psi(i)\}|$. Furthermore, for any $S\subseteq[n]$ we have 
    \begin{equation}\label{eq:TV_dist_restriction}
        \|\mathrm{Law}(F_{m,n}|_S) - \mathrm{Law}(\Psi_{m,n}|_S)\|_{\mathrm{TV}}\leq \frac{|S|(|S|-1)}{n}.
    \end{equation}
\end{lemma}
Lemma~\ref{lem:W1_dist} gives a weaker bound than Lemma~\ref{lem:sampling}. Indeed, we only get an $n^{-1}$ rate when we restrict the maps involved to constant-sized subsets. In contrast, when comparing the entire random maps involved not only do we get an $n^{-1/2}$ rate here as opposed to $n^{-1}$ in Lemma~\ref{lem:sampling}, but also the present bound is in Wasserstein distance whereas the previous one is in total variation. 
In fact, we have $\|\mathrm{Law}(F_{m,n})-\mathrm{Law}(\Psi_{m,n})\|_{\mathrm{TV}}\not\to0$ as $n/m\to\infty$, since the distribution of the fiber sizes of the two maps satisfy 
\begin{equation}\label{eq:multinom_vs_delta}
\left\|\mathrm{Multinomial}(n,m,\mathbbm{1}_m/m)-\delta_{(n/m)\mathbbm{1}_m}\right\|_{\mathrm{TV}}=2\left(1-m^{-n}\binom{n}{n/m,\ldots,n/m}\right)\to2,
\end{equation}
when $m\geq 2$.
Thus, the two random maps in Lemma~\ref{lem:W1_dist} are only close when compared in a specific metric between such maps, namely, $\mathrm{dist}(f,\psi)=|f\neq\psi|/n$, or when restricting them to constant-sized subsets. 
As a consequence, our dual de Finetti theorem only applies to $\FS$-representations on which the $\FS$ action is appropriately local and continuous. The following theorem is the $\FS$ analogue of Theorem~\ref{thm:DeFin_for_coFS}.
\begin{theorem}[Dual de Finetti for $\FS$-representations]\label{thm:dual_deFin_general}
    Let $\mathtt V=\{(\vct V_n),(\tau)\}$ be a $\FS$-representation. Suppose that there is a basis $\{e_{\alpha}\}_{\alpha\in\mc A_n}$ for each $\vct V_n$ and subsets $\{S_{\alpha}\subseteq[n]\}_{\alpha\in\mc A_n}$ of sizes $|S_{\alpha}|\leq d$ such that $\tau(f)e_{\alpha}=\tau(h)e_{\alpha}$ if $f|_{S_{\alpha}} = h|_{S_{\alpha}}$ for all $f,h\colon[n]\to[m]$. Endow each $\vct V_n$ with the $\ell_1$ norm with respect to this basis, and suppose $\|\tau(f)e_{\alpha}\|_1\leq 1$ for all $f\colon[n]\to[m]$ and all $\alpha\in\mc A_n$.
    Let $(\Omega_n)$ be a sequence of compact sets satisfying $\tau(f)\Omega_m\subseteq \Omega_n$ for all $f\colon[m]\to[n]$ and $\sup_n\sup_{x\in\Omega_n}\|x\|_1<\infty$.
    Consider the map from $\varinjlim_{\mscr Z(\mathtt V^\star)}\mc P(\Omega_n)^{\msf S_n}$ to $\varprojlim_{\mscr D(\mathtt V^\star)}\mc P(\Omega_m)^{\msf S_m}$ sending
    \begin{equation}\label{eq:FS_defin_map_dense}
         (\mathrm{sym}_N^{\mscr Z(\mathtt V^\star)}\mu_n)_{N\geq n} \mapsto \Big(\mathrm{Law}(\tau(F_{m,n})X)\Big)_m
    \end{equation}
    where $X\sim\mu_n$ is independent of the uniformly random map $F_{m,n}\colon[n]\to[m]$. This map has a dense image in $W_1$-distance with respect to $\|\cdot\|_1$. 
    More precisely, for any $\mu_{n}\in\mc P(\vct V_{n})^{\msf S_{n}}$ and any $m|n$ we have 
    \begin{equation}\label{eq:W1_bound_FS}
        W_1(\mathrm{Law}(\tau(d_{m,n})X), \mathrm{Law}(\tau(F_{m,n})X))\leq 2d\sqrt{\frac{m(m-1)}{n/m}}\mbb E_{\mu_n}[\|X\|_1],
    \end{equation}
    and 
    \begin{equation}\label{eq:mean_bound_FS}
        \|\mbb E\tau(d_{m,n})X - \mbb E\tau(F_{m,n})X\|_1 \leq \frac{d(d-1)}{n}\mbb E_{\mu_n}[\|X\|_1],
    \end{equation}
    where $X\sim\mu_n$ is independent of $F_{m,n}$.
\end{theorem}
Once again, Lemma~\ref{lem:factorization_of_maps} shows that the condition that $\tau(f)\Omega_m\subseteq \Omega_n$ for all $f\colon[m]\to[n]$ in Theorem~\ref{thm:dual_deFin_general} is equivalent to the sets $(\Omega_n)$ being $\msf S_n$-invariant, embedding into each other in $\mscr Z(\mathtt V^\star)$, and projecting into each other in $\mscr D(\mathtt V^\star)$.

To apply Theorem~\ref{thm:dual_deFin_general} to relate freely-symmetrized measures on $\mscr Z$-sequences to freely-described measures on the corresponding $\mscr D$-sequences, we require a choice of basis satisfying the above hypotheses. 
The following lemma shows that if $\mc F$ has degree $d_{\mc F}$, then these extended bases for $\mc F(\mathtt{R})^\star$ constructed at the end of Section~\ref{sec:intro} satisfy the hypotheses of Theorem~\ref{thm:dual_deFin_general} with $d=d_{\mc F}$.
\begin{lemma}\label{lem:ell1_ext}
    Let $\mc F(\mathtt R)^\star = \{(\vct V_n), (\rho^\star)\}$ for a standard construction $\mc F$ of degree $d_{\mc F}$, and let $\rho^\star$ be the induced $\FS$-action.  Let $\{e_{\alpha}\}_{\alpha\in\mc A_n}$ be a basis for $(\vct V_n)$ that is extended from the canonical one for $\RR^n$.  Then this basis satisfies the hypotheses of Theorem~\ref{thm:dual_deFin_general} with $d=d_{\mc F}$. 
\end{lemma}

Combining Lemma~\ref{lem:ell1_ext} with Theorem~\ref{thm:dual_deFin_general}, we obtain Theorem~\ref{thm:dual_deFin_Pseq} giving a dense subset of freely-described measures over any $\mscr D$-sequence in terms of freely-symmetrized measures over the corresponding $\mscr Z$-sequences.  
We end with examples illustrating that both rates are tight.

\begin{example}[Tightness of rates in Theorem~\ref{thm:dual_deFin_general}]
To see that the use of the Wasserstein distance is necessary and that the rate of $1/\sqrt{n}$ in~\eqref{eq:W1_bound_FS} is sharp, let $X_{2n}=\frac{1}{2n}\mathbbm{1}_{2n}$ be a deterministic invariant vector. Then $\delta_{2n,2}^\star X_{2n} = \frac{1}{2}\mathbbm{1}_2$ while $\rho(F_{2,2n})^\star X_{2n} = \left(\frac{B}{2n},1-\frac{B}{2n}\right)^\top$ where $B\sim\mathrm{Bin}(2n,1/2)$. Then their laws do not converge in total variation by~\eqref{eq:multinom_vs_delta}, while the $W_1$ distance between their laws is $\frac{1}{n}\mbb E|B-n|\sim 1/\sqrt{\pi n}$ by~\cite{Blyth01081980}. In this case $d_{\mc F}=1$ and $\mbb E\rho(F_{2,2n})^\star X_{2n} = \mbb E\rho(\Psi_{2,2n})^\star X_{2n} = \frac{1}{2}\mathbbm{1}_2$. 
To see that the rate of $1/n$ in~\eqref{eq:mean_bound_FS} is sharp, let $X_{2n}=\frac{1}{(2n)^2}\mathbbm{1}_{2n}\mathbbm{1}_{2n}^\top\in\RR^{2n\times 2n}$ and note that $\|\mbb E\rho(F_{2,2n})^\star X_{2n} - \mbb E\rho(\Psi_{2,2n})^\star X_{2n}\|_1 = \frac{1}{2n}$ using the full entrywise $\ell_1$ norm.
\end{example}

\subsection{Proofs for Section~\ref{sec:deFinetti}}\label{sec:deFinetti_proofs}
We now prove the results stated earlier in the section, as well as a few additional technical lemmas needed for these proofs.

\paragraph{Proofs for Section~\ref{sec:deFinetti_coFS}:}
We begin by proving the invariances of a uniformly random map.
\begin{proof}[Proof (Lemma~\ref{lem:unif_map_invars}).]
    Note that $F_{m,n}(i)$ is a uniformly random element of $[m]$, independent for different $i$. If $\phi_{n,k}$ is injective, then each $i\in[k]$ is mapped under $F_{m,n}\circ\phi_{n,k}$ to an independent and uniformly random element of $[m]$, proving the first claim. 
    If $\psi_{k,nk}$ is an equipartition, then a uniformly random element of $[nk]$ is equally likely to fall in each of the equally-sized fibers of $\psi_{k,nk}$. Therefore, each $i\in[m]$ is mapped by $\psi_{k,nk}F_{nk,m}$ to an independent and uniformly random element of $[k]$, proving the second claim.
\end{proof}
Next, we prove the total variation bound between random maps and random injections.
\begin{proof}[Proof (Lemma~\ref{lem:sampling}).]
    Identify a map $[m]\to[n]$ as an element of $[n]^m$. Then $F_{n,m}$ corresponds to the element of $[n]^m$ obtained by sampling $m$ elements from $[n]$ uniformly with replacement, while $\Phi_{n,m}$ corresponds to sampling without replacement. The total variation between them is bounded by~\eqref{eq:sampling_bd} in~\cite[Eq.~(2.3)]{sampling_tv_dist}.
\end{proof}
We use the above lemma to prove our generalized de Finetti theorem for $\coFS$-representations.
\begin{proof}[Proof (Theorem~\ref{thm:DeFin_for_coFS}).]
    First, two measures defining the same freely-symmetrized element have the same images under the given map since if $\Pi_{nk}$ is a uniformly random element of $\msf S_{nk}$ independent of $F_{nk,m}$ then $\tau^{\mathrm{op}}(F_{nk,m})\Pi_{nk}\tau^{\mathrm{op}}(d_{n,nk}) \overset{d}{=} \tau^{\mathrm{op}}(d_{n,nk}\Pi_{nk}F_{nk,m}) \overset{d}{=} \allowbreak \tau^{\mathrm{op}}(F_{n,m})$ by Lemma~\ref{lem:unif_map_invars}. 
    Similarly, the images of the given map are indeed freely-described since $\tau^{\mathrm{op}}(\iota_{M,m})\tau^{\mathrm{op}}(F_{n,M})=\tau^{\mathrm{op}}(F_{n,M}\iota_{M,m})\overset{d}{=}\tau^{\mathrm{op}}(F_{n,m})$ by the same lemma.
    
    Second, we prove the given total variation estimate for a fixed-dimensional exchangeable measure. 
    To this end, if $\Phi_{n,m}\colon[m]\to[n]$ is a uniformly random injection, note that $\Phi_{n,m} \overset{d}{=} G\iota_{n,m}$ where $G\sim\mathrm{Haar}(\msf S_n)$. Therefore, if $\mu_n\in\mc P(\Omega_n)^{\msf S_n}$ is $\msf S_n$-invariant and $X_n\sim\mu_n$ is independent of $\Phi_{n,m}$, then $\tau^{\mathrm{op}}(\Phi_{n,m})X_n \overset{d}{=} \tau^{\mathrm{op}}(\iota_{n,m})X_n$. 
    Since $\mathrm{Law}(\tau^{\mathrm{op}}(\Phi_{n,m})X_n)$ and $\mathrm{Law}(\tau^{\mathrm{op}}(F_{n,m})X_n)$ are pushforwards of $\mathrm{Law}(\Phi_{n,m})\otimes\mu_n$ and $\mathrm{Law}(F_{n,m})\otimes\mu_n$, respectively, and since the total variation distance between pushforwards can only decrease, Lemma~\ref{lem:sampling} gives
    \begin{equation}
        \|\mathrm{Law}(\tau^{\mathrm{op}}(\Phi_{n,m})X_n) - \mathrm{Law}(\tau^{\mathrm{op}}(F_{n,m})X_n)\|_{\mathrm{TV}}\leq \|\mathrm{Law}(\Phi_{n,m})-\mathrm{Law}(F_{n,m})\|_{\mathrm{TV}}\leq\frac{m(m-1)}{n},
    \end{equation}
    as claimed.
    
    Finally, we show that the image of the given map is dense in total variation. Let $(\mu_m)\in\varprojlim_{\mscr Z(\mathtt V)}\mc P(\Omega_m)^{\msf S_m}$. For each $n\in\NN$ consider the image of $(\mathrm{sym}_N^{\mscr D(\mathtt V)}\mu_n)_{n|N}\in\varinjlim_{\mscr D(\mathtt V)}\mc P(\Omega_n)^{\msf S_n}$ under the given map, which yields the sequence of freely-described measures $(\mu_m^{(n)}=\mathrm{Law}(\tau^{\mathrm{op}}(F_{n,m})X_n))_m$ where $X_n\sim\mu_n$ is independent of $F_{n,m}$.
    Since $(\mu_m)$ is freely-described over $\mscr Z(\mathtt V)$, for each $n\geq m$ we have $\mu_m=\tau^{\mathrm{op}}(\iota_{n,m})\mu_n$. Therefore, for each fixed $m\in\NN$ we have
    \begin{equation}
        \|\mu_m^{(n)}-\mu_m\|_{\mathrm{TV}} = \|\mathrm{Law}(\tau^{\mathrm{op}}(F_{n,m})X_n) - \mathrm{Law}(\tau^{\mathrm{op}}(\iota_{n,m})X_n)\|_{\mathrm{TV}}\leq \frac{m(m-1)}{n}\xrightarrow{n\to\infty}0.
    \end{equation}
    This yields the claimed density.
\end{proof}

To obtain Theorem~\ref{thm:DeFin_general_Zseq} from Theorem~\ref{thm:DeFin_for_coFS}, we need to reformulate the condition that $\tau^{\mathrm{op}}(f)\Omega_n\subseteq\Omega_m$ for all $f\colon[m]\to[n]$. This was done above using Lemma~\ref{lem:factorization_of_maps} which we now prove.
\begin{proof}[Proof (Lemma~\ref{lem:factorization_of_maps}).]
    Denote the fibers of $f$ by $f^{-1}(i) = \{n_{i,1}<\ldots<n_{i,k_i}\}$ for $i\in[n]$. Since $k\geq\max_ik_i$ we have $m=\sum_{i=1}^nk_i\leq kn$. Define the injection $\phi\colon[m]\hookrightarrow [nk]$ by $\phi(n_{i,\ell}) = (i-1)k + \ell$, which is well-defined since $1\leq \ell\leq k_i\leq k$. Since $\phi$ is injective, there exists $g\in\msf S_{nk}$ satisfying $\phi=g\iota_{nk,m}$, and by construction $f=d_{n,nk}\circ \phi=d_{n,nk}g\iota_{nk,m}$.
\end{proof}

\paragraph{Proofs for Section~\ref{sec:deFinetti_FS}:} 
We begin by proving the Wasserstein bound between random maps and random equipartition.
\begin{proof}[Proof (Lemma~\ref{lem:W1_dist}).]
    Consider the following coupling of $F_{m,n}$ and $\Psi_{m,n}$. Let $(N_1,\ldots,N_m)\sim\mathrm{Multinomial}(n,m,\mathbbm{1}_m/m)$ be a sample from a multinomial distribution. Let $N_{(1)}\geq\ldots\geq N_{(m)}$ be these values in sorted order. Next, split $[n]$ into consecutive intervals of length $N_{(1)},\ldots,N_{(m)}$ and define $F_0\colon[n]\to[m]$ to take value $i$ on the $i$th interval. Formally,
    \begin{equation}
        F_0(j) = i \quad \textrm{if } j\in\left\{\sum_{\ell=1}^{i-1}N_{(\ell)} + 1,\ldots, \sum_{\ell=1}^iN_{(\ell)}\right\}.
    \end{equation}
    Finally, let $\Pi_n\in\msf S_{n}$ and $\Pi_m\in\msf S_m$ be uniformly random and independent of each other and of $(N_1,\ldots,N_m)$ and set $F=\Pi_m\circ F_0\circ \Pi_n$ and $\Psi=\Pi_m\circ d_{m,n}\circ \Pi_n$. 
    Note that $F\overset{d}{=}F_{m,n}$ since $(|F^{-1}(1)|,\ldots,|F^{-1}(m)|)\overset{d}{=}(N_1,\ldots,N_m)$ and both $F$ and $F_{m,n}$ give a uniformly random ordered partition of $[n]$ with these part sizes.
    We also have $\Psi\overset{d}{=}\Psi_{m,n}$ since equipartitions are precisely maps of the form $gd_{m,n}g'$ for some $g\in\msf S_m$ and $g'\in\msf S_{n}$.
    
    For the above coupling, we have
    \begin{equation*}\begin{aligned}
        |F\neq \Psi| &= |F_0\neq d_{m,n}|=\sum_{i=1}^m|F_0^{-1}(i)\setminus d_{m,n}^{-1}(i)|\leq \sum_{i=1}^m\sum_{\ell=1}^i(N_{(\ell)}-\tfrac{n}{m})\leq m\sum_{i=1}^m|N_i-\tfrac{n}{m}|.
    \end{aligned}\end{equation*}
    Indeed, observe that $\sum_{j=1}^iN_{(j)}\geq in/m$ for all $i$, otherwise we must have $N_{(\ell)}<n/m$ for all $\ell\geq i$ and $\sum_{j=1}^mN_{(j)} < in/m + \sum_{\ell=i+1}^mN_{(\ell)}<n$, a contradiction. 
    Therefore, we have
    \begin{equation*}
        F_0^{-1}(i)\setminus d_{m,n}^{-1}(i) = \left\{\max\left(in/m,\ \sum_{\ell=1}^{i-1}N_{(\ell)}\right)+1,\ldots,\sum_{\ell=1}^iN_{(\ell)}\right\},
    \end{equation*}
    so 
    \begin{equation*}
        |F_0^{-1}(i)\setminus d_{m,n}^{-1}(i)| = \sum_{\ell=1}^iN_{(\ell)} - \max\left(in/m,\ \sum_{\ell=1}^{i-1}N_{(\ell)}\right)\leq \sum_{\ell=1}^i(N_{(\ell)}-\tfrac{n}{m}).
    \end{equation*}
    Finally, using the variance of the coordinates of the multinomial distribution, we get
    \begin{equation}
        \mbb E\left[\frac{|F\neq\Psi|}{n}\right] \leq \frac{1}{n/m}\sum_{i=1}^m\mbb E[|N_i-\tfrac{n}{m}|^2]^{1/2} = \frac{m}{n/m}\sqrt{n(1/m)(1-1/m)} = \sqrt{\frac{m(m-1)}{n/m}},
    \end{equation}
    yielding the first claimed bound.

    For the second bound, for $S\subseteq[n]$ let $\Phi_{n,S}\colon S\to [n]$ be a uniformly random injection and let $F_{n,S}\colon S\to[n]$ be a uniformly random map. Then $F_{m,n}|_S \overset{d}{=} \Psi_{m,n}\circ F_{n,S}$ while $\Psi_{m,n}|_S \overset{d}{=} \Psi_{m,n}\circ \Phi_{n,S}$. Thus, we have $$\|\mathrm{Law}(F_{m,n}|_S)-\mathrm{Law}(\Psi_{m,n}|_S)\|_{\mathrm{TV}}\leq \|\mathrm{Law}(F_{n,S}) - \mathrm{Law}(\Phi_{n,S})\|_{\mathrm{TV}}\leq \frac{|S|(|S|-1)}{n},$$
    as claimed.
\end{proof}
Using the above estimate, we can then prove our dual de Finetti theorem.
\begin{proof}[Proof (Theorem~\ref{thm:dual_deFin_general}).]
    We begin by proving that the $\FS$-action is Lipschitz with respect to the distance $\mathrm{dist}(f,h)=|f\neq h|/n$ between maps $f,h\colon[n]\to[m]$. Specifically, we claim that for any $x\in\vct V_n$, we have
    \begin{equation}\label{eq:Lipschitz_action}
        \mbb E\|(\tau(f)-\tau(h))\Pi_n x\|_1 \leq 2d\|x\|_1\mbb E\left[\frac{|f\neq h|}{n}\right],
    \end{equation}
    where $\Pi_n\in\msf S_n$ is uniformly random. To this end, write $x=\sum_{\alpha\in\mc A_n}x_{\alpha}e_{\alpha}$ and let $(S_{\alpha}\subseteq[n])_{\alpha\in\mc A_n}$ be the subsets from the theorem statement. Then
    \begin{equation*}
        \mbb E\|(\tau(f)-\tau(h))\Pi_n x\|_1 \leq \sum_{\alpha\in\mc A_n}|x_{\alpha}|\cdot \mbb E\|\tau(f\circ\Pi_n) e_{\alpha} - \tau(h\circ\Pi_n)e_{\alpha}\|_1\leq \sum_{\alpha\in\mc A_n}|x_{\alpha}|\cdot 2\mbb P[f\circ\Pi_n|_{S_{\alpha}}\neq h\circ\Pi_n|_{S_{\alpha}}],
    \end{equation*}
    where the last inequality follows from the fact that $\tau(f\circ\Pi_n) e_{\alpha}$ depends only on $f\circ\Pi_n|_{S_{\alpha}}$ and similarly for $h\circ\Pi_n$, and the fact that $\|\tau(f\circ\Pi_n) e_{\alpha} - \tau(h\circ\Pi_n)e_{\alpha}\|_1\leq \|\tau(f\circ\Pi_n) e_{\alpha}\|_1 + \|\tau(h\circ\Pi_n)e_{\alpha}\|_1 \leq 2$ by assumption. Since $|S_{\alpha}|\leq d$ and since $\msf S_n$ acts transitively on subsets of $[n]$ of each fixed size, we conclude that
    \begin{equation*}
        \mbb P[f\circ\Pi_n|_{S_{\alpha}}\neq h\circ\Pi_n|_{S_{\alpha}}] = \frac{1}{\binom{n}{|S_{\alpha}|}}\sum_{\substack{S\subseteq[n] \\ |S|=|S_{\alpha}|}}\mathbbm{1}[f|_S \neq h|_S] \leq \sum_{i\in[n]}\mathbbm{1}[f(i)\neq h(i)]\frac{\binom{n-1}{|S_{\alpha}|-1}}{\binom{n}{|S_{\alpha}|}} = |S_{\alpha}|\frac{|f\neq h|}{n}.
    \end{equation*}
    Since $|S_{\alpha}|\leq d$ for all $\alpha\in\mc A_n$, we obtain the claimed bound~\eqref{eq:Lipschitz_action}.

    Using~\eqref{eq:Lipschitz_action}, we can now prove the first two claims. Two measures defining the same freely-symmetrized element map to the same image since if $\Pi_N$ is a uniformly random element of $\msf S_N$ independent of $F_{m,N}$ then $\tau(F_{m,N})\Pi_N\tau(\iota_{N,n})\overset{d}{=}\tau(F_{m,n})$ by Lemma~\ref{lem:unif_map_invars}. Similarly, the image of the given map is freely-described since $\tau(d_{m,mk})\tau(F_{mk,n})\overset{d}{=}\tau(F_{m,n})$ by the same lemma.
    Next, let $(\Psi_{m,n}',F_{m,n}')$ be an optimal coupling of a uniformly random equipartition $\Psi_{m,n}$ and a uniformly random map $F_{m,n}$. For any $\mu_n\in\mc P(\Omega_n)^{\msf S_n}$, sample $X_n\sim\mu_n$, and independently sample $(\Psi_{m,n}',F_{m,n}')$ and a uniformly random $\Pi_n\in\msf S_n$.
    Then $(\tau(\Psi_{m,n}')\Pi_nX_n,\tau(F_{m,n}')\Pi_nX_n)$ is a coupling of $\tau(d_{m,n})X_n$ and $\tau(F_{m,n})X_n$ since $gX_n\overset{d}{=}X_n$ for any $g\in\msf S_n$ and since $\Psi_{m,n}'\overset{d}{=}d_{m,n}\Pi_n$ for a uniformly random $\Pi_n\in\msf S_n$, so
    \begin{equation*}\begin{aligned}
        W_1(\mathrm{Law}(\tau(d_{m,n})X_n),\mathrm{Law}(\tau(F_{m,n})X_n))&\leq \mbb E[\|(\tau(\Psi_{m,n}')-\tau(F_{m,n}'))\Pi_nX\|]\\ &\leq 2d\mbb E\left[\|X\|\frac{|\Psi_{m,n}'\neq F_{m,n}'|}n\right]\leq 2d\sqrt{\frac{m(m-1)}{n/m}} \mbb E_{X\sim\mu_n}\|X\|.
    \end{aligned}\end{equation*}
    This proves the $W_1$-bound~\eqref{eq:W1_bound_FS}. This bound implies that the map~\eqref{eq:FS_defin_map_dense} has a dense image, since for any $(\mu_m)\in\varprojlim_{\mscr D(\mathtt V^\star)}\mc P(\Omega_m)^{\msf S_m}$ consider the sequence of freely-symmetrized elements $(\mathrm{sym}_N^{\mscr Z(\mathtt V^\star)}\mu_{n!})$ and its image $(\mathrm{Law}(\tau(F_{m,n!})X_{n!}))_m$ under~\eqref{eq:FS_defin_map_dense}. If $R=\sup_n\sup_{x\in\Omega_n}\|x\|_1$, then for any fixed $m$ and all $n\geq m$ we have
    \begin{equation*}
        W_1(\mu_m, \mathrm{Law}(\tau(F_{m,n!})X_{n!})) = W_1(\mathrm{Law}(\tau(d_{m,n!})X_{n!}),\mathrm{Law}(\tau(F_{m,n!})X_{n!}))\leq 2d\sqrt{\frac{m(m-1)}{n!/m}}R,
    \end{equation*}
    which vanishes as $n\to\infty$, showing that the image of~\eqref{eq:FS_defin_map_dense} is indeed dense.

    Finally, to prove~\eqref{eq:mean_bound_FS} we show that for any $x\in\vct V_n$ and any $m|n$ we have
    \begin{equation}\label{eq:mean_bound_determ}
        \|\mbb E\tau(\Psi_{m,n})x - \mbb E\tau(F_{m,n})x\|_1 \leq \frac{d(d-1)}{n}\|x\|_1,
    \end{equation}
    which yields~\eqref{eq:mean_bound_FS} after conditioning on $X$. To prove~\eqref{eq:mean_bound_determ}, expand $x=\sum_{\alpha\in\mc A_n}x_{\alpha} e_{\alpha}$ and observe that
    \begin{equation*}\begin{aligned}
        \|\mbb E\tau(\Psi_{m,n})x - \mbb E\tau(F_{m,n})x\|_1 &\leq \sum_{\alpha\in\mc A_n}|x_{\alpha}| \|\mbb E\tau(\Psi_{m,n})e_{\alpha} - \mbb E\tau(F_{m,n})e_{\alpha}\|_1\\ &\leq 2\sum_{\alpha\in\mc A_n}|x_{\alpha}| \inf_{(\Psi',F')}\mbb P[\Psi'|_{S_{\alpha}}\neq F'|_{S_{\alpha}}]\\ &= \sum_{\alpha\in\mc A_n}|x_{\alpha}|\cdot \|\mathrm{Law}(\Psi_{m,n}|_{S_{\alpha}}) - \mathrm{Law}(F_{m,n}|_{S_{\alpha}})\|_{\mathrm{TV}} \leq \frac{d(d-1)}{n}\|x\|_1,
    \end{aligned}\end{equation*}
    where in the second line the infimum is over couplings $(\Psi',F')$ of $\Psi_{m,n}$ and $F_{m,n}$, which satisfy $\|\mbb E\tau(\Psi_{m,n})e_{\alpha}-\mbb E\tau(F_{m,n})e_{\alpha}\|_1\leq \mbb E\|\tau(\Psi')e_{\alpha}-\tau(F')e_{\alpha}\|_1\leq 2\mbb P[\Psi'|_{S_{\alpha}}\neq F'|_{S_{\alpha}}]$ for each $\alpha\in\mc A_n$. The last inequality follows from Lemma~\ref{lem:W1_dist}.
\end{proof}
To instantiate this theorem on $(\RR^n)$ and sequences derived from it, we prove Lemma~\ref{lem:ell1_ext}.
\begin{proof}[Proof (Lemma~\ref{lem:ell1_ext}).]
    We induct on the constructions $\mc F$. For $\mc F=\mathrm{id}$ the standard basis $(e_i)_{i\in[n]}$ satisfies the hypotheses of Theorem~\ref{thm:dual_deFin_general} with singleton subsets $(S_i=\{i\})$ and with $d=1$, as $\tau(f)e_i=e_{f(i)}$. Suppose that the claim holds for $(\mc F_1(\RR^n))$ and $(\mc F_2(\RR^n))$, with bases $(e_{\alpha})_{\alpha\in\mc A_n}$ and $(e_{\beta})_{\beta\in\mc B_n}$ and subsets $(S_{\alpha}\subseteq[n])$ and $(S_{\beta}\subseteq[n])$ of sizes $|S_{\alpha}|\leq \deg \mc F_1$ and $|S_{\beta}|\leq \deg \mc F_2$. 
    It is easy to check that the basis $\{e_{\alpha}\}\sqcup \{e_{\beta}\}$ satisfies the desired conclusions with subsets $\{S_{\alpha}\}\sqcup\{S_{\beta}\}$, whose size is at most $\max\{\deg\mc F_1,\deg \mc F_2\} = \deg \mc F_1\times\mc F_2$, and that the basis $\{e_{\alpha}\otimes e_{\beta}\}$ satisfies these conclusions with subsets $\{S_{\alpha} \cup S_{\beta}\}$, whose size is at most $\deg\mc F_1 + \deg \mc F_2 = \deg \mc F_1\otimes \mc F_2$. An explicit computation also shows that the bases $\{e_{\alpha_1}\cdots e_{\alpha_d}\}$ and $\{x_{\alpha_1}\cdots x_{\alpha_d}\}$ for $\mathrm{Sym}^d\mc F_1(\RR^n)$ and for $\RR[\mc F_1(\RR^n)]_d$, respectively, satisfy the needed hypotheses with subsets $\{S_{\alpha_1}\cup\ldots\cup S_{\alpha_d}\}$, whose size is at most $d\deg \mc F_1 = \deg \mathrm{Sym}^d\mc F_1 = \deg \RR[\mc F_1]_d$.
\end{proof}

\paragraph{Proofs for Section~\ref{sec:deFinetti_intro}:} 
We derive Theorems~\ref{thm:DeFin_general_Zseq} and \ref{thm:dual_deFin_Pseq} from Theorems~\ref{thm:DeFin_for_coFS} and \ref{thm:dual_deFin_general}.


\begin{proof}[Proof (Theorem~\ref{thm:DeFin_general_Zseq})]
Let $\mc F(\mathtt{R})$ be the $\coFS$-representation derived from the construction $\mc F$, where $\mathtt{R}$ is the standard $\coFS$-representation.  Next we recall that the consistent sequences $\mscr Z(\mc F(\mathtt{R}))$ and $\mscr D(\mc F(\mathtt{R}))$ from \eqref{eq:coFS_consistent_sequences} are equal to $\mc F(\mscr Z)$ and $\mc F(\mscr D)$, respectively.  Instantiating Theorem~\ref{thm:DeFin_for_coFS} on the $\coFS$-representation $\mc F(\mathtt{R})$ yields the desired conclusion.
\end{proof}

\begin{proof}[Proof (Theorem~\ref{thm:dual_deFin_Pseq})]
The proof follows the same strategy as that of Theorem~\ref{thm:DeFin_general_Zseq}, applying Theorem~\ref{thm:dual_deFin_general} to the $\FS$-representation $\mc F(\mathtt{R})^\star$ and using Lemma~\ref{lem:ell1_ext}.
\end{proof}

\section{Bounds for Any-Dimensional Polynomial Problems}\label{sec:free_poly_opts}

In this section, we use Theorems~\ref{thm:DeFin_general_Zseq} and \ref{thm:dual_deFin_Pseq} from Section~\ref{sec:deFinetti} to construct lower bounds on the limiting optimal value of a freely-described POP via a freely-symmetrized POP, and we prove associated rates of convergence.  While the constants involved in these rates, as well as their derivations, differ for freely-described problems over $\mscr Z$-sequences and over $\mscr D$-sequences, the construction of the bounds is the same and we seek to provide as unified a treatment as possible.  We begin by abstracting away the common features of these two cases into a single definition.
\begin{definition}[De Finetti maps]\label{def:definetti_maps}
    Let $(\mscr U,\mscr L)$ be a $\FS$-pair on $(\vct V_n)$ as in Definition~\ref{def:fs_pair} and let $F_{m,n}\colon[n]\to[m]$ be a uniformly random map. Define the \emph{de Finetti maps} on $(\mscr U, \mscr L)$ to be the random linear maps $L_{m,n}\colon\vct V_n\to\vct V_m$ given by
    \begin{enumerate}
        \item $L_{m,n} = \rho(F_{n,m})$ if $(\mscr U,\mscr L)=(\mc F(\mscr D),\mc F(\mscr Z))$,
        \item $L_{m,n} = \rho(F_{m,n})^\star$ if $(\mscr U,\mscr L) = (\mc F(\mscr Z),\mc F(\mscr D))$,
    \end{enumerate}
    for some sequence of standard constructions $\mc F$.  Here we have used the action $\rho$ of maps between finite sets on $(\vct V_n)$ derived in Section~\ref{sec:deFinetti_intro} from the standard action on $(\RR^n)$.
\end{definition}


We construct bounds for freely-described problems over $\mscr U$ in terms of freely-symmetrized problems over $\mscr L$ using the random maps $L_{m,n}$ and their associated de Finetti theorems.  In particular, Theorems~\ref{thm:DeFin_general_Zseq} and \ref{thm:dual_deFin_Pseq} from Section~\ref{sec:deFinetti} can be summarized as follows. For any $\FS$-pair $(\mscr U,\mscr L)$ and any $\FS$-compatible sequence of sets $(\Omega_n)$ satisfying the hypotheses of the corresponding theorem, the map from  $\varinjlim_{\mscr U}\mc P(\Omega_n)^{\msf S_n}$ to $\varprojlim_{\mscr L}\mc P(\Omega_n)^{\msf S_n}$ defined by:
\begin{equation*}
    (\mathrm{sym}_N^{\mscr U} \mu_m)_{m \preceq_{\mscr U} N} \mapsto \left(\mathrm{Law}(L_{n,m} X)\right)_n
\end{equation*}
has a dense image, where $X \sim \mu_m$ is independent of $L_{n,m}$.  Further, we have nonasymptotic error bounds (all in the appropriate topology).


Section~\ref{sec:free_sym_bds} presents our main result on obtaining freely-symmetrized bounds for freely-described POPs.  These lower bounds are interpreted in terms of any-dimensional certificates of nonnegativity in Section~\ref{sec:any_dim_nonneg_cones}.  Finally, Section~\ref{sec:free_sym_bds_proofs} gives the proofs of some supporting results that are used in the proof of the main theorem in Section~\ref{sec:free_sym_bds}.  This section is organized so that it is possible to proceed directly to Section~\ref{sec:construct_bds} after Section~\ref{sec:free_sym_bds}.


\subsection{Freely-Symmetrized Bounds for Freely-Described Problems} \label{sec:free_sym_bds}

The goal of this section is to prove the following theorem describing a procedure for constructing freely-symmetrized lower bounds for freely-described problems.  We use de Finetti maps to identify freely-symmetrized POPs whose optimal values yield lower bounds on the given freely-described POPs, and then to derive associated rates of convergence for these bounds by applying Theorems~\ref{thm:DeFin_general_Zseq} and \ref{thm:dual_deFin_Pseq} from Section~\ref{sec:deFinetti}.

\begin{theorem}\label{thm:free_sym_bds}
    Let $(\mscr U,\mscr L)$ be a $\FS$-pair of degree $D$, let $(\Omega_n\subseteq \vct V_n)$ be $\FS$-compatible and compact, and let $(p_n)\in\varprojlim_{\mscr U}\RR[\vct V_n]^{\msf S_n}_{\leq d}$ be a freely-described polynomial of degree $d$.  If $\mscr U$ is a $\mscr Z$-sequence, endow each $\vct V_n$ with the 1-norm with respect to the basis of Lemma~\ref{lem:ell1_ext} and assume that $R=\sup_n\sup_{x\in\Omega_n}\|x\|_1<\infty$.
    Consider the freely-described POP $u_n=\inf_{x\in\Omega_n}p_n(x)$ with its limiting optimal value $u_{\infty}=\inf_n u_n$.  We have the following results.
    \begin{enumerate}[font=\emph, align=left]
        \item[(Dual cost)] For each $k\geq dD$, there exists $q_k\in\RR[\vct V_k]_{\leq d}$ such that
        \begin{equation}\label{eq:representation_eqn}
            p_n(x)=\mbb E[q_k(L_{k,n}x)],\quad \textrm{ for all } x\in\vct V_n \textrm{ and all } n\in\NN. 
        \end{equation}
        The corresponding freely-symmetrized element $(\mathrm{sym}_n^{\mscr L}q_k)$
        is unique, meaning that if $q_{k_1}$ and $q_{k_2}$ are two polynomials satisfying~\eqref{eq:representation_eqn}, then $\mathrm{sym}_n^{\mscr L}q_{k_1} = \mathrm{sym}_n^{\mscr L}q_{k_2}$ whenever $n \succeq_{\mscr L} k_1,k_2$.
        In terms of the above $q_k$, we have
        \begin{equation}\label{eq:FDM}
            u_\infty = \inf_{(\mu_n)\in\varprojlim_{\mscr L}\mc P(\Omega_n)^{\msf S_n}}\mbb E_{\mu_k}q_k.
        \end{equation}

        \item[(Freely-symmetrized bounds)] Let $q_k$ be as above and consider the freely-symmetrized sequence of problems $$\ell_n=\inf_{x\in\Omega_n}\mathrm{sym}_n^{\mscr L}q_k(x),$$ 
        for $k\preceq_{\mscr L} n$. Then $\ell_n\leq \ell_N\leq u_{\infty}$ for all $k\preceq_{\mscr L} n\preceq_{\mscr L} N$, and denoting by $\ell_{\infty}=\sup_n\ell_n$ their limiting optimal value, we have $\ell_\infty=u_\infty$.

        \item[(Rate of convergence I)] If $\mscr U$ is a $\mscr D$-sequence and $k\leq n$, then
        \begin{equation}\label{eq:rate_coFS}
            u_n-\ell_n \leq \frac{k(k-1)}{n} \|q_k\|_{\Omega_k},
        \end{equation}
        where $\|q_k\|_{\Omega_k}=\sup_{x\in\Omega_k}|q_k(x)|$. 

        \item[(Rate of convergence II)] If $\mscr U$ is a $\mscr Z$-sequence and the underlying vector spaces $(\vct V_n)$ are endowed with the 1-norms of Lemma~\ref{lem:ell1_ext}, then for any $k|n$
        \begin{equation}\label{eq:rate_FS}
            u_{n}-\ell_{n} \leq \frac{d(d+1)D(dD-1)}{n}\|q_k\|_{c}\max\{1,R\}^d,
        \end{equation}
        where $\|q_k\|_{c}$ is the largest magnitude of a coefficient in $q_k$.
    \end{enumerate}
\end{theorem}
Explicitly, if $\{e_{\alpha}\}_{\alpha\in \mc A_k}$ is the basis for $\vct V_k$ constructed in Section~\ref{sec:intro} and we write $q_k(x)=\sum_{I\in\NN^{\mc A_k}}c_Ix^I$ where $x^I=\prod_{\alpha\in\mc A_k}x_{\alpha}^{I_{\alpha}}$, then $\|q_k\|_c=\sup_I|c_I|$. Note also that $\mathrm{sym}_n^{\mscr L}q_k(x) = \mbb Eq_k(\psi_{n,k}^\star\Pi_n x)$ where $\Pi_n\in\msf S_n$ is uniformly random, so we can view the lower bounds $\ell_n$ as being obtained from the original sequence of POPs $u_n$ by modifying the sampling map arising in the objective functions.

In summary, given a $\FS$-pair $(\mscr U, \mscr L)$ and a freely-described POP defined over a consistent sequence $\mscr U$, we obtain a reformulation~\eqref{eq:FDM} involving optimization over a sequence of measures that is freely-described over $\mscr L$; this is the promised generalization of~\eqref{eq:limiting_lower_bd} from the example of Section~\ref{sec:case_study}.  Truncating the sequence of measures in this reformulation then yields the desired lower bounds.  For freely-described problems $(u_n)$ over $\mscr D$-sequences, we get freely-symmetrized lower bounds $(\ell_n)$ over the corresponding $\mscr Z$-sequence satisfying
\begin{equation*}
    \ldots\leq \ell_n\leq \ell_{n+1}\leq \ldots \leq \ell_\infty=u_\infty\leq\ldots\leq u_{nk}\leq u_n\leq\ldots,
\end{equation*}
for all $n,k\in\NN$, converging at a rate $u_n-\ell_n\lesssim 1/n$. 
For freely-described problems $(u_n)$ over $\mscr Z$-sequences, we get freely-symmetrized lower bounds $(\ell_n)$ over the corresponding $\mscr D$-sequence satisfying
\begin{equation*}
    \ldots\leq \ell_n\leq \ell_{nk}\leq \ldots \leq \ell_\infty=u_\infty\leq\ldots\leq u_{n+1}\leq u_n\leq\ldots,
\end{equation*}
for all $n,k\in\NN$, also converging at a rate $u_{nk}-\ell_{nk}\lesssim 1/n$ for fixed $k$ as $n\to\infty$.  We remark that if $(\Omega_n)$ are not compact, the optimal values $(\ell_n)$ are still monotonically increasing lower bounds for $u_\infty$ as above, but may not converge to $u_\infty$. For example, if $(\mscr U,\mscr L)=(\mscr D,\mscr Z)$ and $u_n=\inf_{x\in\RR^n}\left(\frac{1}{n}\sum_{i=1}^nx_i\right)^2$, we have $u_n=0$ for all $n$ but $\ell_n = \inf_{x\in\RR^n}\frac{1}{\binom{n}{2}}\sum_{i<j}x_ix_j=-\infty$ for all $n\geq2$.
We also note that while $u_n$ is not a lower bound on $u_\infty$ in general, one could imagine computing $u_n$ and subtracting the right-hand sides of~\eqref{eq:rate_coFS} or~\eqref{eq:rate_FS} and thereby obtaining lower bounds on $u_\infty$.  However, such an approach leads to more conservative lower bounds than $\ell_n$ by construction, as we illustrate in several examples in Section~\ref{sec:construct_bds}.

As Theorem~\ref{thm:free_sym_bds} shows, the main step in our framework is the construction of the fixed polynomial $q_k$ satisfying~\eqref{eq:representation_eqn} for a given freely-described polynomial $(p_n)$.  We explain how to do so for various $\FS$-pairs arising in applications in Section~\ref{sec:construct_bds}.  
We show next that the rates of convergence in Theorem~\ref{thm:free_sym_bds} are tight, followed by a proof of the theorem.
\begin{example}[Tightness of rates in Theorem~\ref{thm:free_sym_bds}]
    To see that the $1/n$ rate for freely-described problems over $\mscr D$-sequences is tight, consider the following freely-described POP over $\mscr D$:
    \begin{equation*}
        u_n = \inf_{x\in[-1,1]^n} \left(\frac{1}{n}\sum_{i=1}^nx_i\right)^2.
    \end{equation*}
    Here $D = 1, ~ d = 2$ and $u_n=0$ for all $n$.  For $k = 2$, we have $q_k(x) = x_1x_2$, so that the lower bounds are given by:
    \begin{equation*}
        \ell_n = \inf_{x\in[-1,1]^n}\frac{1}{\binom{n}{2}}\sum_{i<j}x_ix_j = \inf_{x\in[-1,1]^n} \frac{\left(\sum_ix_i\right)^2-\sum_ix_i^2}{n(n-1)}\geq -\frac{1}{n-1},
    \end{equation*}
    where the final inequality follows from the observation that $\left(\sum_i x_i\right)^2 \geq 0$ and $\sum_ix_i^2 \leq n$ for $x \in [-1,1]^n$.
    Setting $\lfloor n/2\rfloor$ of the  $x_i$'s to 1 and the remaining entries to $-1$ shows that $\ell_n=-\frac{1}{n-1}$ when $n$ is even and $-\frac{1}{n-1}\leq \ell_n\leq -\frac{1}{n-1} + \frac{1}{n(n-1)}$ when $n$ is odd. Thus, $u_n-\ell_n\sim 1/n$.

    To see that the $1/n$ rate for freely-described problems over $\mscr Z$-sequences is tight, consider
    \begin{equation*}
        u_n = \inf_{x\in \Delta^{n-1}}\sum_{i=1}^nx_i^2.
    \end{equation*}
    Here $D=1, ~ d=2,$ and $u_n=1/n$ for all $n$. For $k=2$, we have $q_k(x)=(x_1-x_2)^2$, so the lower bounds are given by
    \begin{equation*}
        \ell_{2n} = \inf_{x\in\Delta^{2n-1}}\frac{1}{\binom{2n}{n}}\sum_{\substack{I\subseteq[2n]\\ |I|=n}}\left(\sum_{i\in I}x_i - \sum_{i\notin I}x_i\right)^2.
    \end{equation*}
    Note that $\ell_{2n}=0$ for all $n$ since we clearly have $\ell_{2n}\geq0$ while the constant vector $x=\frac{1}{2n}\mathbbm{1}_{2n}$ attains zero cost above. Thus, we have $u_{2n}-\ell_{2n}\sim 1/n$.
\end{example}

We state two intermediate propositions needed to prove Theorem~\ref{thm:free_sym_bds} followed by a proof of the theorem.  The proofs of the propositions themselves are deferred to Section~\ref{sec:free_sym_bds_proofs}.  The first proposition pertains to an isomorphism between freely-symmetrized elements in $\mscr L$ and freely-described elements in $\mscr U$ induced by the de Finetti maps, which facilitates the identification of the dual cost in Theorem~\ref{thm:free_sym_bds}.

\begin{proposition}\label{prop:free_sym_and_free_descr_isom}
Let $(\mscr U,\mscr L)$ be a $\FS$-pair over a sequence of vector spaces $(\vct V_n)$.  The map from $\varinjlim_{\mscr L}\vct V_n^{\msf S_n}$ to $\varprojlim_{\mscr U}\vct V_n^{\msf S_n}$ which sends $(\mathrm{sym}_n^{\mscr L} v_k)_{k\preceq_{\mscr L} n} \mapsto (\mbb EL_{k,n}^\star v_k)_n$ is an isomorphism, where $L_{k,n}$ is the appropriate de Finetti map on $(\mscr U, \mscr L)$.
\end{proposition}
 

One consequence of this isomorphism is that projections and embeddings between invariants become isomorphisms starting from some dimension,
which also implies that the spaces of freely-described and freely-symmetrized elements over such $\mscr V$ are isomorphic to fixed-dimensional spaces of invariants $\vct V_k^{\msf S_k}$ for all large enough $k$. 
Formally, for each $k\in\NN$ let $\pi_k\colon \varprojlim_{\mscr V}\vct V_n^{\msf S_n}\to \vct V_k^{\msf S_k}$ be the map sending $(v_n)\mapsto v_k$ and $\varsigma_k\colon \vct V_k^{\msf S_k}\to \varinjlim_{\mscr V}\vct V_n^{\msf S_n}$ be the map sending $v_k\mapsto (\mathrm{sym}_n^{\mscr V}v_k)_{n\succeq k}$. We then have the following strengthening of Proposition~\ref{prop:finite_dim_free_invariants}.
 \begin{proposition}\label{prop:calc_for_stability}
     If $\mscr V$ is a $\mscr Z$- or a $\mscr D$-sequence of degree $d_{\mscr V}$, then $\dim\varprojlim\vct V_n^{\msf S_n}=\dim\varinjlim\vct V_n^{\msf S_n}=\dim\vct V_{d_\mscr V}^{\msf S_{d_{\mscr V}}}$.  
     Moreover, 
     the maps $\pi_n$ and $\varsigma_n$ are isomorphisms for all $n\geq d_{\mscr V}$, and the map $\mathrm{sym}_{\msf S_N}\varphi_{N,n}\colon\allowbreak\vct V_n^{\msf S_n}\to\vct V_N^{\msf S_N}$ and $\varphi_{N,n}^\star\colon\vct V_N^{\msf S_N}\to\vct V_n^{\msf S_n}$, are isomorphisms whenever $N\succeq n \geq d_{\mscr V}$.
\end{proposition}
We are now in a position to prove Theorem~\ref{thm:free_sym_bds}.

\begin{proof}[Proof (Theorem~\ref{thm:free_sym_bds}).]
    The existence and uniqueness of the freely-symmetrized element $(\mathrm{sym}_n^{\mscr L}q_k)\in\varinjlim_{\mscr L}\RR[\vct V_n]^{\msf S_n}_{\leq d}$ satisfying~\eqref{eq:representation_eqn} follows by applying Propositions~\ref{prop:free_sym_and_free_descr_isom} and~\ref{prop:calc_for_stability} to the $\FS$-pair $(\RR[\mscr U]_{\leq d},\allowbreak\RR[\mscr L]_{\leq d})$ of degree $dD$, and observing that if $q\in\RR[\vct V_k]$ then $(\mc P_{\leq d}(L_{k,n})^\star q)(x)=q(L_{k,n}x)$. 

    To obtain~\eqref{eq:FDM}, note that 
    \begin{equation*}
        u_\infty = \inf_n \inf_{x\in\Omega_n}\mbb Eq_k(L_{k,n}x) = \inf_{(\mu_n)\in\varprojlim_{\mscr L}\mc P(\Omega_n)^{\msf S_n}}\mbb E_{\mu_k}q_k \quad \textrm{s.t. } (\mu_n=\mathrm{Law}(L_{n,m}x))_n \textrm{ for some } x\in\Omega_m.
    \end{equation*}
    Since the objective is linear in $(\mu_n)$ and continuous in the weak topology on $\mc P(\Omega_k)^{\msf S_k}$, we can further optimize over limits of mixtures of such freely-described measures, which yields all of $\varprojlim_{\mscr L}\mc P(\Omega_n)^{\msf S_n}$ by Theorems~\ref{thm:DeFin_general_Zseq} and~\ref{thm:dual_deFin_Pseq} applied to the $\FS$-pair $(\mscr U,\mscr L)$.

    Next, the optimal values $\ell_n$ are increasing by~\eqref{eq:lower-bound-fs-pops}. They are lower bounds for $u_{\infty}$ because whenever $k\preceq_{\mscr L} n$ we have $\mathrm{sym}_n^{\mscr L}q_k=\mathrm{sym}_{\msf S_n}(q_k\circ\psi_{n,k}^\star )$ so
    \begin{equation}\begin{aligned}\label{eq:free_sym_lower_bd_proof}
        \ell_n &= \inf_{x\in\Omega_n}\mathrm{sym}_{\msf S_n}(q_k\circ\psi_{n,k}^\star ) = \inf_{\mu\in\psi_{n,k}^\star \mc P(\Omega_n)^{\msf S_n}}\mbb E_{\mu}q_k\leq \inf_{(\mu_n)\in\varprojlim_{\mscr L}\mc P(\Omega_n)^{\msf S_n}}\mbb E_{\mu_k}q_k=u_{\infty},
    \end{aligned}\end{equation}
    where the inequality follows since if $(\mu_n)\in\varprojlim_{\mscr L}\mc P(\Omega_n)^{\msf S_n}$ then $\mu_n\in\mc P(\Omega_n)^{\msf S_n}$ and $\mu_k=\psi_{n,k}^*\mu_n$.
    Finally, we prove the claimed convergence rates, which also imply that $\ell_\infty=u_\infty$. For each $x\in\Omega_n$ 
    let $X=\Pi_n x$ where $\Pi_n\in\msf S_n$ is a uniformly random permutation, and let $L_{k,n}$ be independent of $X$. For any $k\preceq_{\mscr L} n$ we have
    \begin{equation}\label{eq:free_poly_ptwise_bd}
        |p_n(x) - \mathrm{sym}_n^{\mscr L}q_k(x)| = \left|\mbb Eq_k(L_{k,n}X) - \mbb Eq_k(\psi_{n,k}^\star X)\right| \leq \begin{cases}
            \|\mathrm{Law}(L_{k,n}X) - \mathrm{Law}(\psi_{n,k}^\star X)\|_{\mathrm{TV}}\|q_k\|_{\Omega_k},\\ 
            \|\mbb E\psi_{n,k}^\star X^{\otimes \leq d} - \mbb EL_{k,n} X^{\otimes \leq d}\|_1\|q_k\|_{c},
        \end{cases}
    \end{equation}
    since $q_k(x)=\langle Q_k, x^{\otimes \leq d}\rangle$ where $x^{\otimes\leq d}=\bigoplus_{i=0}^dx^{\otimes i}$ and $Q_k$ is a tensor whose entries are either zero or equal to coefficients of $q_k$.
    If $\mscr U$ is a $\mscr D$-sequence then $\mscr L$ is a $\mscr Z$-sequence and Theorem~\ref{thm:DeFin_general_Zseq} gives $\|\mathrm{Law}(L_{k,n}X) - \mathrm{Law}(\psi_{n,k}^\star X)\|_{\mathrm{TV}}\leq \frac{k(k-1)}{n}$, while if $\mscr U$ is a $\mscr Z$-sequence then $\mscr L$ is a $\mscr D$-sequence and Theorem~\ref{thm:dual_deFin_Pseq} gives 
    $\|\mbb E\psi_{n,k}^\star X^{\otimes\leq d} - \mbb EL_{k,n} X^{\otimes\leq d}\|_1\leq \frac{d(d+1)D(dD-1)}{n}\max\{1,\|x\|_1^d\}$ since $\|X^{\otimes\leq d}\|_1=\sum_{i=0}^d\|x\|_1^i\leq (d+1)\max\{1,\|x\|_1^d\}$.
    Since
    \begin{equation*}
        u_n-\ell_n = \inf_{x\in\Omega_n}p_n(x) - \inf_{x\in\Omega_n}\mathrm{sym}_n^{\mscr L}q_k(x)\leq \sup_{x\in\Omega_n}|p_n(x)-\mathrm{sym}_n^{\mscr L}q_k(x)|,
    \end{equation*}
    this proves both claimed rates of convergence.
\end{proof}

\begin{remark}[Rates for non-polynomial costs]\label{rmk:non_poly}
    We briefly discuss extending Theorem~\ref{thm:free_sym_bds} to non-polynomial cost functions. If $(p_n)$ is freely-described but not polynomial, there might not exist a function $q_k\colon\Omega_k\to\RR$ satisfying~\eqref{eq:representation_eqn}. For example, the sequence of Shannon entropies $p_n(x)=-\sum_ix_i\log(x_i)$ is freely-described over $\mscr Z$ (with the convention $0\log 0=0$), but $p_n(\mathbbm{1}_n/n)=\log n\to\infty$ even though $L_{k,n}\mathbbm{1}_n/n\overset{d}{=}\frac{1}{n}\mathrm{Multinom}(n,k,\mathbbm{1}_k/k)$ converges weakly to $\mathbbm{1}_k/k$ for each $k$. 
    However, if a suitably regular $q_k$ does exist then the results of Theorem~\ref{thm:free_sym_bds} continue to hold with small modifications. Specifically, assume a bounded $q_k$ exists if $\mscr U$ is a $\mscr D$-sequence, or a Lipschitz-continuous $q_k$ exists if $\mscr U$ is a $\mscr Z$-sequence. Then the above proof shows that the reformulation~\eqref{eq:FDM} still holds and that $\ell_n=\inf_{x\in\Omega_n}\mathrm{sym}_n^{\mscr L}q_k$ give monotonically increasing and convergent lower bounds. Likewise, the proof of the $O(1/n)$ rate of convergence~\eqref{eq:rate_coFS} for freely-described problems over $\mscr D$-sequences continues to hold. However, the $O(1/n)$ rate~\eqref{eq:rate_FS} for freely-described problems over $\mscr Z$-sequences no longer holds. Instead, we get the following $O(1/\sqrt{n})$ rate of convergence
    \begin{equation}\label{eq:old_sqrt_n_rate}
        u_n-\ell_n\leq 2DL_{q_k}R\sqrt{\frac{k(k-1)}{n/k}},
    \end{equation}
    whenever $k|n$, where $L_{q_k}$ is the Lipschitz constant of $q_k$ with respect to the $\ell_1$-norm on $\Omega_k$.
    The bound~\eqref{eq:old_sqrt_n_rate} is a consequence of the similar $W_1$-bound in our dual de Finetti theorem (Theorem~\ref{thm:dual_deFin_Pseq}). This $O(1/\sqrt{n})$ rate is tight for general Lipschitz $q_k$. For example, if $q_2(x)=|x_1-x_2|$ and $\Omega_n=\Delta^{n-1}$ then $p_n(x)=\mbb Eq_2(L_{2,n}x) = \mbb E\left|\sum_{i=1}^n\varepsilon_ix_i\right|$ where $\varepsilon_i$ are iid uniformly random signs. In this case, we have $\ell_{2n} = 0$ while $u_{2n}\geq \frac{1}{2\sqrt{n}}$ since $p_{2n}(x)\geq \frac{\|x\|_2}{\sqrt{2}}$ by Khintchine's inequality.
\end{remark}

\subsection{Any-Dimensional Certificates of Nonnegativity}\label{sec:any_dim_nonneg_cones}
We can rewrite $u_\infty$ and our lower bounds $\ell_n$ for it in terms of any-dimensional certificates of nonnegativity.
Suppose $\mscr U$ is a $\mscr Z$- or $\mscr D$-consistent sequence, the sets $(\Omega_n\subseteq\vct V_n)$ are $\FS$-compatible, and $u_n=\inf_{x\in\Omega_n}p_n(x)$ is a freely-described problem over $\mscr U$ of degree at most $d$. 
Define the nonnegativity cones
\begin{equation}\label{eq:nonneg_cones}
    \mc K_n = \{p\in\RR[\vct V_n]_{\leq d}^{\msf S_n}: p|_{\Omega_n}\geq0\},
\end{equation}
where we suppress the dependence of $\mc K_n$ on the sets $(\Omega_n)$ and degree $d$, held fixed throughout this section.
We can rewrite each optimal value $u_n$ in the sequence as the greatest lower bound on $p_n$ over $\Omega_n$, namely,
\begin{equation}\label{eq:un_dual_form}
    u_n = \sup_{\gamma\in\RR}\gamma \quad \textrm{ s.t. } p_n-\gamma\in\mc K_n.
\end{equation}
The limiting optimal value can similarly be written as
\begin{align}
    u_{\infty} &= \sup_{\gamma\in\RR}\gamma \quad \textrm{ s.t. } p_n-\gamma\in\mc K_n \textrm{ for all } n\in\NN, \nonumber\\
    &= \sup_{\gamma\in\RR}\gamma \quad \textrm{ s.t. } (p_n-\gamma)_n\in\bigcap_{n\in\NN}\pi_n^{-1}(\mc K_n),\label{eq:u_inf_reform}
\end{align}
where $\pi_k\colon \varprojlim_{\mscr U}\RR[\vct V_n]^{\msf S_n}_{\leq d}\to \RR[\vct V_k]^{\msf S_k}_{\leq d}$ sends $(p_n)\mapsto p_k$.
The latter intersection is precisely the collection of freely-described polynomials that are nonnegative over $(\Omega_n)$ in all dimensions, denoted
\begin{equation}\label{eq:fd_nn_cone}
    \varprojlim_{\mscr U}\mc K_n = \left\{(p_n)\in\varprojlim_{\mscr U}\RR[\vct V_n]_{\leq d}^{\msf S_n}:p_n|_{\Omega_n}\geq0 \textrm{ for all } n\right\} = \bigcap_n\pi_n^{-1}(\mc K_n),
\end{equation}
which is a convex cone in the finite-dimensional vector space $\varprojlim_{\mscr U}\RR[\vct V_n]_{\leq d}^{\msf S_n}$.

Producing lower bounds on $u_{\infty}$ amounts to optimizing over subcones of $\varprojlim_{\mscr U}\mc K_n$, and one family of such subcones is furnished by Theorem~\ref{thm:free_sym_bds}. Indeed, if $h_m\in\mc K_m$ is a fixed-dimensional polynomial nonnegative over $\Omega_m$, then its image under the de Finetti map $(\mbb Eh_m\circ L_{m,n})_n\in\varprojlim_{\mscr U}\mc K_n$ is a freely-described polynomial nonnegative in all dimensions, meaning that $\mbb Eh_m\circ L_{m,n}|_{\Omega_n}\geq0$ for all $n\in\NN$. Therefore, for each fixed $m\in\NN$ we get the lower bound
\begin{equation}\label{eq:nn_cone_lower_bd}
    u_{\infty}\geq \sup_{\gamma\in\RR}\gamma \quad \textrm{ s.t. } (p_n-\gamma) = (\mbb Eh_m\circ L_{m,n}) \textrm{ for some } h_m\in\mc K_m.
\end{equation}
We now argue that the lower bound in~\eqref{eq:nn_cone_lower_bd} is precisely $\ell_m$ from Theorem~\ref{thm:free_sym_bds}. Indeed, recall from Theorem~\ref{thm:free_sym_bds} that $(p_n)=(\mbb Eq_k\circ L_{k,n})$ for a unique freely-symmetrized element $(\mathrm{sym}_n^{\mscr L}q_k)$. Therefore, if $m\succeq_{\mscr L}k$ then 
\begin{equation*}
    \begin{aligned}
        (p_n-\gamma)=(\mbb Eh_m\circ L_{m,n}) &\iff (\mbb E(q_k-\gamma)\circ L_{k,n}) = (\mbb Eh_m\circ L_{m,n})\\
        &\iff (\mathrm{sym}_n^{\mscr L}(q_k-\gamma))=(\mathrm{sym}_n^{\mscr L}h_m)\\
        &\iff \mathrm{sym}_m^{\mscr L}q_k - \gamma = h_m,
    \end{aligned}
\end{equation*}
where the first equivalence follows by definition of $q_k$, the second from the fact that the de Finetti map defines an isomorphism as in Proposition~\ref{prop:free_sym_and_free_descr_isom}, and the third from Definition~\ref{def:free_elements} and the assumption that $m\succeq_{\mscr L}k$. 
We also note that $\mathrm{sym}_m^{\mscr L}q_k-\gamma =h_m$ for some $h_m\in\mc K_m$ precisely when $\mathrm{sym}_m^{\mscr L}q_k-\gamma \in \mc K_m$. 
Thus, the lower bound in~\eqref{eq:nn_cone_lower_bd} becomes
\begin{equation}\label{eq:ellm_dual_form}
    \ell_m=\sup_{\gamma\in\RR}\gamma \quad \textrm{ s.t. } \mathrm{sym}_m^{\mscr L}q_k-\gamma\in \mc K_m.
\end{equation}
Note that the only difference between the expression~\eqref{eq:ellm_dual_form} for $\ell_m$ and the expression~\eqref{eq:un_dual_form} for $u_n$ is that $p_n$ is replaced by $\mathrm{sym}_m^{\mscr L}q_k$, again highlighting the fact that we only modify the cost function but keep the same constraints in our construction of lower bounds.

Similarly, any convex relaxation of $\ell_m$ obtained by replacing $\mc K_m$ with a tractable subcone of it, including SOS~\cite{SOS_chapter} and relative entropy-based relaxations~\cite{murray2021newton}, can be interpreted in this way as searching over lower bounds $\gamma$ that can be certified by a subcone of freely-described polynomials nonnegative in all dimensions.

Finally, taking the greatest lower bound that can be certified by letting the dimension $m$ vary above, we obtain the lower bound
\begin{equation*}
    u_{\infty}\geq \ell_{\infty} = \sup_{\gamma\in\RR}\gamma \quad \textrm{ s.t. } (p_n-\gamma) = (\mbb Eh_m\circ L_{m,n}) \textrm{ for some } (\mathrm{sym}_n^{\mscr L}h_m)\in\bigcup_n\varsigma_n(\mc K_n),
\end{equation*}
where $\varsigma_k\colon \RR[\vct V_k]^{\msf S_k}_{\leq d}\to \varinjlim_{\mscr L}\RR[\vct V_n]^{\msf S_n}_{\leq d}$ sends $h_k\mapsto (\mathrm{sym}_n^{\mscr L}h_k)_{n\succeq_{\mscr L} k}$. 
Just like our reformulation of $u_{\infty}$ in~\eqref{eq:u_inf_reform}, the bound $\ell_{\infty}$ involves a limiting cone in a finite-dimensional space of polynomials, this time the freely-symmetrized nonnegativity cone
\begin{equation}\label{eq:fs_nn_cone}
    \varinjlim_{\mscr L}\mc K_n = \left\{(\mathrm{sym}_n^{\mscr L}h_m)_{m\preceq_{\mscr L} n}\in\varinjlim_{\mscr L}\RR[\vct V_n]_{\leq d}^{\msf S_n}: h_m|_{\Omega_m}\geq0\right\} = \bigcup_n\varsigma_n(\mc K_n),
\end{equation}
contained in $\varinjlim_{\mscr L}\RR[\vct V_n]_{\leq d}^{\msf S_n}$.  

The de Finetti map $(\mathrm{sym}_n^{\mscr L}h_m)\mapsto (\mbb Eh_m\circ L_{m,n})$ sends the cone $\varinjlim_{\mscr L}\mc K_n$ into the cone $\varprojlim_{\mscr U}\mc K_n$.
The fact that $u_{\infty}=\ell_{\infty}$ can then be derived from our present geometric perspective by proving that the image of this map is dense. We proceed to prove this, along with a few basic geometric properties of the two limit cones.
%

\begin{theorem}\label{thm:any_dim_nonneg_cones}
    Suppose $(\mscr U,\mscr L)$ is a $\FS$-pair, let $(\Omega_n)$ be $\FS$-compatible, fix degree $d\in\NN$, and denote by $\mc K_n$ the nonnegativity cones~\eqref{eq:nonneg_cones}.
    \begin{enumerate}[font=\emph, align=left]
        \item[(Density)] The isomorphism $(\mathrm{sym}_n^{\mscr L}q_k)\mapsto (\mbb Eq_k\circ L_{k,n})$ from Proposition~\ref{prop:free_sym_and_free_descr_isom} maps $\varinjlim_{\mscr L}\mc K_n$ into a dense subset of $\varprojlim_{\mscr U}\mc K_n$. 
    
        \item[(Monotonicity)] We have
        \begin{equation}\label{eq:invar_cone_inclusions}
            \begin{aligned}
                &\pi_n^{-1}(\mc K_n)\supset \pi_N^{-1}(\mc K_N)\ \textrm{ if } n\preceq_{\mscr U} N, & \textrm{and} && \varsigma_n(\mc K_n)\subseteq \varsigma_N(\mc K_N)\ \textrm{ if } n\preceq_{\mscr L} N.
            \end{aligned}
        \end{equation}

        \item[(Geometry)] The cones $\varprojlim_{\mscr U}\mc K_n$ and $\varinjlim_{\mscr L}\mc K_n$ are convex, and the former is closed. Moreover, these cones are pointed or have nonempty interior if the cones $\mc K_n$ are pointed or have nonempty interior, respectively.
    \end{enumerate}
\end{theorem}
For the density result, closures are taken with respect to the usual unique Hausdorff TVS topology on a finite-dimensional space (e.g., induced by a norm on the coefficients in a basis). 
The cone $\mc K_n$ is pointed and has nonempty interior if $\Omega_n$ is compact and has nonempty interior.
\begin{proof}
    Note that if $q_k|_{\Omega_k}\geq0$ then $\mbb Eq_k(L_{k,n}x)\geq0$ for all $x\in\Omega_n$ since $L_{k,n}x\in\Omega_k$ almost surely. This shows that $\varinjlim_{\mscr L}\mc K_n$ is indeed mapped into $\varprojlim_{\mscr U}\mc K_n$. Moreover, we argue next that this image is dense.  If $(p_n)\in\varprojlim_{\mscr U}\mc K_n$ then for each $i\in\NN$ we have
    $(\mathrm{sym}_n^{\mscr L}p_i)_{n\succeq_{\mscr L} i}\in\varinjlim_{\mscr L}\mc K_n$, which is mapped to $(\mbb Ep_i
    \circ L_{i,n})\in\varprojlim_{\mscr U}\mc K_n$. By the proof of Proposition~\ref{prop:free_sym_and_free_descr_isom}, we have $\lim_{i\to\infty}\mbb Ep_i\circ L_{i,n}=p_n$ for each $n\in\NN$, which implies $(\mbb Ep_i\circ L_{i,n})\to (p_n)$ in the usual topology on the finite-dimensional space $\varprojlim_{\mscr U}\RR[\vct V_n]^{\msf S_n}_{\leq d}$. This proves the claimed density.
    
    For monotonicity, note that if $(p_n)\in\varprojlim_{\mscr U}\RR[\vct V_n]_{\leq d}^{\msf S_n}$ satisfies $p_N|_{\Omega_N}\geq0$ and $n\preceq_{\mscr U} N$, then $p_n|_{\Omega_n}=p_N\circ\varphi_{N,n}|_{\Omega_n}\geq0$ since $\varphi_{N,n}\Omega_n\subseteq \Omega_N$. Similarly, if $q_k|_{\Omega_k}\geq0$ and $k\preceq_{\mscr L} n$ then $\mathrm{sym}_n^{\mscr L}q_k|_{\Omega_n}(x) = \mbb E[q_k(\psi_{n,k}^\star \Pi_n x)]\geq0$ since $\psi_{n,k}^*\Pi_nx\in\Omega_k$ almost surely by compatibility of $(\Omega_n)$. 

    Since $\pi_n$ and $\varsigma_n$ are isomorphisms for all large $n$ by Proposition~\ref{prop:calc_for_stability} and as $\mc K_n$ is closed and convex, we conclude that $\pi_n^{-1}(\mc K_n)$ and $\varsigma_n(\mc K_n)$ are both closed and convex.  Therefore, the expressions~\eqref{eq:fd_nn_cone} and~\eqref{eq:fs_nn_cone} imply that $\varprojlim_{\mscr U}\mc K_n$ and $\varinjlim_{\mscr L}\mc K_n$ are both convex, and that the former is closed. If all the $\mc K_n$ are pointed then $\pi_n^{-1}(\mc K_n)$ is pointed for all large $n$, hence~\eqref{eq:fd_nn_cone} shows that $\varprojlim_{\mscr U}\mc K_n$ is pointed. This implies $\varinjlim_{\mscr L}\mc K_n$ is pointed as it embeds linearly into $\varprojlim_{\mscr U}\mc K_n$. Similarly, if all the $\mc K_n$ have nonempty interior then $\varsigma_n(\mc K_n)$ has nonempty interior for all large $n$, hence~\eqref{eq:fs_nn_cone} shows that $\varinjlim_{\mscr L}\mc K_n$ has nonempty interior, which in turn implies that $\varprojlim_{\mscr U}\mc K_n$ has nonempty interior.
\end{proof}
The last two results in Theorem~\ref{thm:any_dim_nonneg_cones} were proved for $\mscr U=\mscr Z$ in~\cite[Thm.~3.2]{acevedo2024symmetric} and for $\mscr U=\mscr D$ in~\cite[\S1.1]{acevedo2024power}. Our proof for the former case is significantly simpler thanks to the first result in Theorem~\ref{thm:any_dim_nonneg_cones}, and our proof applies more generally to any $\FS$-pair.

\subsection{Proofs for Section~\ref{sec:free_sym_bds}}\label{sec:free_sym_bds_proofs}
In this section, we prove Propositions~\ref{prop:free_sym_and_free_descr_isom} and~\ref{prop:calc_for_stability} from Section~\ref{sec:free_sym_bds}.
 \begin{proof}[Proof (Proposition~\ref{prop:free_sym_and_free_descr_isom}).]
    First, observe that the given map is well-defined and yields freely-described elements. Indeed, stated in terms of de Finetti maps, Lemma~\ref{lem:unif_map_invars} states that $L_{n,m}^\star \psi_{n,k}\overset{d}{=}L_{k,m}^\star$ whenever $k\preceq_{\mscr L} n$ and $\varphi_{m,k}^\star L_{n,m}^\star\overset{d}{=}L_{n,k}^\star$ whenever $k\preceq_{\mscr U} m$.    
    The former property shows that two vectors which correspond to the same freely-symmetrized elements have the same images, and the latter shows that the images of our map are indeed freely-described over $\mscr U$.
    
    Second, it suffices to prove the claim when $\mscr U=\mc F(\mscr Z)$ is a $\mscr Z$-sequence and $\mscr L=\mc F(\mscr D)$ is the corresponding $\mscr D$-sequence, as the opposite case follows from this one via duality by Proposition~\ref{prop:free-duality}.
%
    Third, we show that the given map is surjective, which suffices for the proof as $\dim\varinjlim_{\mscr L}\vct V_n^{\msf S_n}\leq \dim\varprojlim_{\mscr U}\vct V_n^{\msf S_n}$ by Proposition~\ref{prop:finite_dim_free_invariants}.
    To this end, fix a freely-described element $(v_n)\in\varprojlim_{\mscr U}\vct V_n^{\msf S_n}$ and let $\Omega_n = \bigcup_{m\in\NN}\mathrm{supp}(L_{m,n}^\star v_m) = \bigcup_{m\in\NN}\{\rho(f_{m,n})v_m: f_{m,n}\colon[n]\to[m]\}$, which is $\FS$-compatible with respect to the reversed $\FS$-pair $(\mscr L,\mscr U)$ by construction. Note that each $\Omega_n$ is finite. Indeed, any map $f_{m,n}\colon[n]\to[m]$ decomposes as $f_{m,n}=\sigma_m\circ\iota_{m,r} \circ \tilde f_{r,n}$ where $\tilde f_{r,n}\colon[n]\to[r]$ is surjective, $\iota_{m,r}\colon[r]\to[m]$ is the usual inclusion, and $\sigma_m\in\msf S_m$ is a permutation. Since $(v_m)$ is freely-described, we have $\rho(f_{m,n})v_m=\rho(\tilde f_{r,n})\rho(\iota_{m,r})\rho(\sigma_m)v_m = \rho(\tilde f_{r,n})v_r$ where we used the fact that $\rho(\iota_{m,r})=\varphi_{m,r}^\star$ (see Section~\ref{sec:deFinetti_coFS}). Thus, we have $\Omega_n = \bigcup_{r\leq n}\{\rho(f_{r,n})v_r: f_{r,n}\colon[n]\to[r]\textrm{ surjective}\}$.
    
    Then $(v_n)$ can be identified with the freely-described measure $(\delta_{v_n})\in\varprojlim_{\mscr U}\mc P(\Omega_n)^{\msf S_n}$ over $\mscr U$. By Theorem~\ref{thm:DeFin_general_Zseq}, the random variables $L_{k,n}^\star v_{k}$ converge in total variation to $\delta_{v_n}$ for any fixed $n$ as $k\to\infty$. Taking expectations and using the finiteness of $(\Omega_n)$, we conclude that $\lim_{k\to\infty}\mbb EL_{k,n}^\star v_{k} = v_n$ for each $n$. Since $(\mbb EL_{k,n}^\star v_{k})$ is the image under the given map of $(\mathrm{sym}_{n}^{\mscr L}v_{k})_{k|n}$, this shows that the given map has a dense image. Since its image is a finite-dimensional space it must be further closed. Thus, the given map is surjective. 
 \end{proof}
Combining the isomorphism from Proposition~\ref{prop:free_sym_and_free_descr_isom} with Proposition~\ref{prop:finite_dim_free_invariants} yields Proposition~\ref{prop:calc_for_stability}.
  \begin{proof}[Proof (Proposition~\ref{prop:calc_for_stability}).]
    This result was proved for $\mscr Z$-sequences in~\cite[Prop.~2.9, Thm.~2.11]{levin2023free}, so we turn to proving it for $\mscr D$-sequences $\mscr V$. It suffices to prove that $\varphi_{N,n}^\star\colon\vct V_N^{\msf S_N}\to\vct V_n^{\msf S_n}$ is an isomorphism whenever $n|N$ and $n\geq d_{\mscr V}$, since then its adjoint $\mathrm{sym}_N^{\mscr V}\colon\vct V_n^{\msf S_n}\to\vct V_N^{\msf S_N}$ is an isomorphism as well, and hence both $\pi_n$ and $\varsigma_n$ are isomorphisms for $n\geq d_{\mscr V}$.
    
    Suppose first that $\mscr V = \mscr D^{\otimes d}$ so $\vct V_n=(\RR^n)^{\otimes d}$ and $\varphi_{N,n}=\delta_{N,n}^{\otimes d}$. We argue that $\varphi_{N,n}^\star\colon \vct V_N^{\msf S_N}\to\vct V_n^{\msf S_n}$ is an isomorphism whenever $n|N$ and $n\geq d$. For each $I\in[n]^d$ we consider a partition $P_I$ of $[d]$ where $i,j\in[d]$ are in the same part if and only if $I_i=I_j$. In other words, the partition $P_I$ encodes the equality pattern of the multi-index $I$.
    Then for $n\geq d$ a basis for $\vct V_n^{\msf S_n}$ is given by $\{e_P^{(n)}=\sum_{I: P_I=P}e_I\}_P$ where $P$ ranges over partitions of $[d]$, since $I,J\in[n]^d$ are in the same $\msf S_n$ orbit if and only if $P_I=P_J$ in this case. For two partitions $P,P'$ of $[d]$, we denote $P \preceq P'$ if $P'$ is a refinement of $P$, meaning each block in $P$ is a union of blocks in $P'$. Note that $\varphi_{N,n}^\star e_I = e_J$ where if $I=(i_1,\ldots,i_d)$ then $J=(\lceil i_1/(N/n)\rceil,\ldots,\lceil i_d/(N/n)\rceil)$, so $P_J\preceq P_I$. Furthermore, for any partition $P$ and multi-index $J\in[n]^d$ with $P_J = P$, note that $I=(N/n)J\in[N]^d$ satisfies $P_I = P$ and $\varphi_{N,n}^\star e_I=e_J$. 
    Thus, we have $\varphi_{N,n}^\star e_P^{(N)} = \sum_{P'\preceq P}c_{P'}e_{P'}^{(n)}$ for some $c_{P'}\in \NN$ and $c_P>0$. We conclude that the matrix of $\varphi_{N,n}^\star\colon\vct V_N^{\msf S_N}\to\vct V_n^{\msf S_n}$ with respect to the above partition-indexed bases is lower-triangular with respect to the refinement partial order and with strictly positive entries on the diagonal, so it is invertible as claimed. This proves the claim for $\mscr V=\mscr D^{\otimes d}$.

    For a general $\mscr D$-sequence $\mscr V=\mc F(\mscr D)$ of degree $d_{\mscr V}$, we can embed $\vct V_n\subseteq\bigoplus_{m=1}^M(\RR^n)^{\otimes d_m}$ for some $M\in\NN$ and $\max_md_m=d_{\mscr V}$, where $\varphi_{N,n}$ is the restriction of $\bigoplus_{m=1}^M\delta_{N,n}^{\otimes d_m}$ to $\vct V_n$ so $\mathrm{sym}_N^{\mscr V}\colon \vct V_n^{\msf S_n}\to\vct V_N^{\msf S_N}$ is injective whenever $N\succeq n\geq d_{\mscr V}$. By the result of~\cite{levin2023free} applied to the corresponding $\mscr Z$-sequence $\mc F(\mscr Z)$, we have $\dim \vct V_n^{\msf S_n}=\dim\vct V_{d_{\mscr V}}^{\msf S_{d_{\mscr V}}}$ for all $n\geq d_{\mscr V}$. We conclude that $\mathrm{sym}_N^{\mscr V}$ and $\varphi_{N,n}^\star$ are isomorphisms of invariants for $n\geq d_{\mscr V}$, as claimed.
 \end{proof}


\section{Illustrations of Our Framework}\label{sec:construct_bds}

In this section we present a systematic framework for obtaining bounds on freely-described POPs by leveraging Theorem~\ref{thm:free_sym_bds} in Section~\ref{sec:free_sym_bds}.  We describe the steps of our method here and we illustrate our approach in four consistent sequences arising in different applications.

Concretely, suppose that $(\mscr U,\mscr L)$ is a $\FS$-pair of degree $D$ and consider freely-described problems $u_n=\inf_{x\in\Omega_n}p_n(x)$ over $\mscr U$, where $\deg p_n=d$ and the constraints $(\Omega_n)$ are $\FS$-compatible (i.e., embed into each other in $\mscr U$ and project into each other in $\mscr L$). 
We construct freely-symmetrized lower bounds for such problems by carrying out the following steps.
\begin{enumerate}[font=\emph, align=left]
    \item[(Basis for $\mscr L$)] Fix a $k \geq Dd$.  Let $q_{k}^{(1)},\ldots,q_k^{(m)}$ be a basis for $\RR[\vct V_k]_{\leq d}^{\msf S_k}$. 
    

    \item[(Basis for $\mscr U$)] Derive $p_n^{(i)} = \mbb E[q^{(i)}_k \circ L_{k,n}]$ to which the above basis elements map under the de Finetti map. Propositions~\ref{prop:free_sym_and_free_descr_isom} and~\ref{prop:calc_for_stability} imply that $(p_n^{(1)}), \ldots, (p_n^{(m)})$ form a basis for freely-described polynomials over $\mscr U$ of degree at most $d$.
    
    
    \item[(Dual cost)] Express the freely-described cost in this basis as $p_n = \sum_{i=1}^m c^{(i)} p^{(i)}_n$.  The dual costs for $n\succeq_{\mscr L} k$ are given by $q_n = \sum_{i=1}^m c^{(i)} \mathrm{sym}^{\mscr L}_n q_k^{(i)}$.
\end{enumerate}


The above steps yield the freely-symmetrized bounds $\ell_n = \inf_{x \in \Omega_n} q_n(x)$ for $n \succeq_{\mscr L} k$.  When $\mscr L$ is a $\mscr Z$-sequence, the partial order $\preceq_{\mscr L}$ is just the standard order on $\NN$.  Consequently, it suffices to just perform the above derivation for $k = D d$ to form the lower bounds for all dimensions $n \geq k$.  On the other hand, when $\mscr L$ is a $\mscr D$-sequence, the partial order $\preceq_{\mscr L}$ is the divisibility order, and performing the above three steps for $k = D d$ only yields the bounds $\ell_n$ for $k | n$.  If we seek a lower bound $\ell_n$ when $k \nmid n$, we need to perform the preceding computations by setting $k = n$ in the first step (or for some other $k' \geq D d$ such that $k' | n$).

On a related but distinct point, explicitly forming the polynomials $\mathrm{sym}_n^{\mscr L}q_k^{(i)}$ can also be inconvenient when $\mscr L$ is a $\mscr D$-sequence, in which case embeddings do not map monomials to monomials.  It is then more convenient to directly perform the preceding derivations with $k = n$ in the first step.  As we illustrate in this section, in several applications it is possible to choose convenient bases $(q_k^{(i)})$ in the first step which are explicitly given as simple functions of $k$, and to perform the above three steps for them to obtain $p_n^{(i)}$ as a function of $k$ for all $k$; see Sections~\ref{sec:symmetric_funcs} and~\ref{sec:graph_numbers}.


All the code producing the numerical results in this section is available at: \url{https://github.com/eitangl/anyDimPolyOpt}.


\subsection{Polynomial Mean-Field Games}\label{sec:power_means}
\paragraph{Model problem.} We consider freely-described problems over $\mscr D^{d_g}$ arising as the values of potential mean-field games. 
Mean-field games are limits of games with a large number of identical players; see~\cite{cardaliaguet2010notes} for an introduction.
Formally, in a potential mean-field game with a compact strategy set $\Theta\subseteq\RR^{d_g}$, we consider symmetric games in a growing number $n$ of players with costs $C_1^{(n)},\ldots,C_n^{(n)}$ satisfying $C_{i}^{(n)}(x_{\sigma(1)},\ldots,x_{\sigma(n)})=C_{\sigma(i)}^{(n)}(x_1,\ldots,x_n)$ for all $\sigma\in\msf S_{n}$ and all $x_1,\ldots,x_n\in\Theta$.  The game is collaborative in that the players jointly incur a cost $\frac{1}{n}\sum_{i=1}^nC_i^{(n)}(x)$ that they seek to minimize, and we are interested in the minimum value that they can achieve in the limit as $n\to\infty$. 
Under suitable assumptions~\cite[Thm.~2.1]{cardaliaguet2010notes}, there exists a single function $C\colon \Theta\times \mc P(\Theta)\to\RR$ such that $\max_{i\in[n]}\sup_{x\in\Theta^n}|C_i^{(n)}(x) - C(x_i, \frac{1}{n}\sum_j\delta_{x_j})|\to0$ along a subsequence of $n\to\infty$. 
The limiting value of the game becomes an optimization problem over measures $\inf_{\mu\in\mc P(\Theta)}\mbb E_\mu[C(X,\mu)]$ where $X \sim \mu$.  We restrict ourselves to polynomial costs of the form $f(x, \mathbb{E}_{\mu}[X^{\alpha_1}], \dots, \mathbb{E}_{\mu}[X^{\alpha_m}])$ for a polynomial $f : \R^{d_g} \times \R^m \rightarrow \R$ and some fixed multi-indices $\alpha_1,\dots,\alpha_m \in \mbb{N}^{d_g}\setminus\{0\}$.  We then have the following sequence of freely-described POPs, similar to the one considered in Section~\ref{sec:case_study}
\begin{equation}\label{eq:un_mfg}
    u_n = \inf_{(x_1,\ldots,x_n)\in\Theta^n} \frac{1}{n} \sum_{i=1}^nf\left(x_i,\frac{1}{n}\sum_{j=1}^nx_j^{\alpha_1},\ldots,\frac{1}{n}\sum_{j=1}^nx_j^{\alpha_m}\right),
\end{equation}
with the associated limiting optimal value of
\begin{equation*}
    u_\infty = \inf_{\mu \in \mathcal{P}(\Theta)} \mathbb{E}_\mu\Big[f(X, \mathbb{E}_\mu [X^{\alpha_1}], \ldots, \mathbb{E}_\mu [X^{\alpha_m}])\Big].
\end{equation*}
Our objective is to obtain lower bounds on $u_\infty$ by identifying the corresponding freely-symmetrized POP over $\mscr Z^{d_g}$.  Note that the vector spaces underlying $\mscr D^{d_g}$ and $\mscr Z^{d_g}$ are $\vct V_n = \R^{d_g \times n}$.

\paragraph{Basis for $\mscr L=\mscr Z^{d_g}$.} Consider a list of multi-indices $\Lambda=(\alpha_1\geq\ldots\geq\alpha_{\mathrm{len}(\Lambda)})$ sorted lexicographically with $\alpha_i\in\mbb{N}^{d_g}\setminus\{0\}$ for all $i$.  For $k \geq \mathrm{len}(\Lambda)$, we define the monomial means
\begin{equation}\label{eq:monomial_means}
    \bar m_k^{(\Lambda)}(x_1,\dots,x_k) = \mathrm{sym}_k^{\mscr Z^{d_g}} x_1^{\alpha_1}\cdots x_{\mathrm{len}(\Lambda)}^{\alpha_{\mathrm{len}(\Lambda)}} = \frac{1}{(k)_{\mathrm{len}(\Lambda)}}\sum_{\substack{i_1,\ldots i_{\mathrm{len}(\Lambda)}\\ \textrm{distinct}}}x_{i_1}^{\alpha_1}\cdots x_{i_{\mathrm{len}(\Lambda)}}^{\alpha_{\mathrm{len}(\Lambda)}},
\end{equation}
where $(k)_r = \frac{k!}{(k-r)!}$, and whose degree is $|\Lambda|=\sum_{i=1}^{\mathrm{len}(\Lambda)}|\alpha_i|$.
The collection $\{\bar m_k^{(\Lambda)}\}_{|\Lambda|\leq d}$ forms a basis for $\RR[\vct V_k]_{\leq d}^{\msf S_k}$ for any $k\geq d$, as can be seen by writing any polynomial in the monomial basis and symmetrizing.
Moreover, we have $\bar m_K^{(\Lambda)}=\mathrm{sym}_K^{\mscr L}\bar m_k^{(\Lambda)}$ whenever $k\leq K$ by construction. Thus, $(\bar m_n^{(\Lambda)})_{n\geq d}$ form a basis for freely-symmetrized polynomials of degree at most $d$ over $\mscr L$, as proved more generally in~\cite{levin2023free}.

\paragraph{Basis for $\mscr U=\mscr D^{d_g}$.} To compute $\mbb E[\bar m_k^{(\Lambda)}\circ L_{k,n}]$ for the de Finetti map $L_{k,n}$, observe that if $x_1,\ldots,x_n \in \R^{d_g}$ then $\mathrm{Law}(L_{k,n} \cdot (x_1,\dots,x_n))=\mu_x^{\otimes k}$ where $\mu_x=\frac{1}{n}\sum_{i=1}^n\delta_{x_i}\in\mc P(\RR^{d_g})$. Therefore, from \eqref{eq:monomial_means}
\begin{equation*}
    \mbb E[\bar m_{k}^{(\Lambda)}(L_{k,n} \cdot (x_1,\dots,x_n))] = {\mbb E}_{\mu_x^{\otimes k}}[x_1^{\alpha_1}\cdots x_{\mathrm{len}(\Lambda)}^{\alpha_{\mathrm{len}(\Lambda)}}] = \prod_{i=1}^{\mathrm{len}(\Lambda)}\mbb E_{\mu_x}[x^{\alpha_i}] = \prod_{i=1}^{\mathrm{len}(\Lambda)}\left(\frac{1}{n}\sum_{j=1}^nx_j^{\alpha_i}\right).
\end{equation*}
By defining the power means as
\begin{equation*}\begin{aligned}
    \bar s_{n}^{(\Lambda)}(x_1,\dots,x_n)=\prod_{i=1}^{\mathrm{len}(\Lambda)}\bar s_{n}^{(\alpha_i)}(x_1,\dots,x_n), && \textrm{where} &&    \bar s_{n}^{(\alpha_i)}(x_1,\dots,x_n)=\frac{1}{n}\sum_{j=1}^nx_j^{\alpha_i},
\end{aligned}\end{equation*}
we have that $\mbb E[\bar m_k^{(\Lambda)}\circ L_{k,n}] = \bar s_{n}^{(\Lambda)}$.  We conclude by Proposition~\ref{prop:free_sym_and_free_descr_isom} that $\{(\bar s_n^{(\Lambda)})\}_{|\Lambda|\leq d}$ is a basis for freely-described polynomials of degree at most $d$ over $\mscr D^{d_g}$.

\paragraph{Dual cost.} Write the freely-described cost $(p_n)$ as $p_n=\sum_{\Lambda} c^{(\Lambda)} \bar s_n^{(\Lambda)}$. Then the dual cost is given by $q_n=\sum_{\Lambda} c^{(\Lambda)} \bar m_n^{(\Lambda)}$ for all $n\geq d$.  The lower bounds on the freely-described problems~\eqref{eq:un_mfg} are then given by $\ell_n = \inf_{X\in\Theta^n}q_n(X)$, which satisfy $\ell_n\leq \ell_{n+1}$ and $u_n-\ell_n\lesssim 1/n$ by Theorem~\ref{thm:free_sym_bds}.

\paragraph{$\FS$-compatible sets.} The sequences of constraint sets $(\Omega_n\subseteq \RR^{d_g \times n})$ for which our framework is applicable consist of $\msf S_n$-invariant compact sets satisfying $(\underbrace{x_1,\dots,x_1}_{k \text{ times}}, \ldots, \underbrace{x_n,\dots,x_n}_{k \text{ times}}) \in \Omega_{nk}$ and $(x_1,\ldots,x_{n-1}) \in \Omega_{n-1}$ if $(x_1,\ldots,x_n) \in\Omega_n$.  An example is a product set $\Omega_n=\Theta^n$ for any compact $\Theta\subseteq\RR^{d_g}$.  For non-examples, neither the sequence of Frobenius norm balls in $\RR^{d_g \times n}$ nor the sequence of row stochastic matrices in $\RR^{d_g \times n}$ are $\FS$-compatible.


\begin{example}\label{ex:mfg}
    For $d_g=1$ and $\Theta=[-1,1]$, consider the polynomial cost $f(x,\mu_1,\mu_2)=5x\mu_1-2x^2
    \mu_1-2x\mu_2-\mu_1$, where $\mu_k=\mbb E_{\mu}[t^k]$. This is a mean-field extension of the non-convex-concave two-player game studied in~\cite[Ex.~3.2]{parrilo2006games}, where we replace one of the players by a growing number of symmetric players and take their payoff function to be our cost.
    Then the value of the game~\eqref{eq:un_mfg} becomes
    \begin{equation*}
        u_n=\inf_{x\in[-1,1]^n} 5\bar s_n^{(1,1)}(x) - 4\bar s_n^{(2,1)}(x) - \bar s_n^{(1)}(x), 
    \end{equation*}
    which is once again neither convex nor concave. 
    We then get the lower bounds for $n\geq 3$
    \begin{equation*}
        \ell_n = \inf_{x\in[-1,1]^n} 5\bar m_n^{(1,1)}(x) - 4\bar m_n^{(2,1)}(x) - \bar m_n^{(1)}(x).
    \end{equation*}
    These are precisely the optimization problems from Example~\ref{ex:mfg_intro} in Section~\ref{sec:intro}.
    We numerically compute and plot $u_n$ and $\ell_n$ for $n\in[3,21]$ in Figure~\ref{fig:mfg_numerics}. Specifically, in Figure~\ref{fig:mfg_bounds} we obtain upper bounds on $u_n$ and intervals containing $\ell_n$; the upper bounds on $u_n,\ell_n$ are computed using a nonconvex solver (Matlab's \texttt{fmincon}), while the lower bounds on $\ell_n$ are computed using a degree-$4$ sums-of-squares relaxation.  In Figure~\ref{fig:mfg_bound_diffs} we plot the error bars containing the differences $u_n-\ell_n$ against the theoretical bound on these differences from Theorem~\ref{thm:free_sym_bds}, using $k=3$ and $\|q_3\|_{[-1,1]^3}\leq 10$ to obtain~\eqref{eq:theor_bd_mfg}. 
    Note that the upper bounds on $u_n$ are not monotonically decreasing. This does not contradict~\eqref{eq:upper-bound-fd-pops}, which only guarantees $u_{nk}\leq u_n$. 
\end{example}

\subsection{Inequalities in Symmetric Functions}\label{sec:symmetric_funcs}


\paragraph{Model problem.} Dually to the previous section, we consider freely-described problems over $\mscr U=\mscr Z$. Freely-described polynomials over $\mscr Z$ are precisely the extensively-studied symmetric functions~\cite[Chap.~7]{Stanley_Fomin_1999}, and certifying inequalities between them which hold in all dimensions is therefore of interest.  Any homogeneous symmetric polynomial is a linear combination of products of power sums $s_n^{(m)}(x)=\sum_{i=1}^nx_i^m$ by~\cite[\S7.7]{Stanley_Fomin_1999}.  Moreover, certifying nonnegativity of a homogeneous polynomial is equivalent to certifying that its minimum over a norm ball is nonnegative. Thus, we consider freely-described problems of the form
\begin{equation}\label{eq:pwr_sums_prob}
    u_n = \inf_{\substack{x\in\RR^n\\\|x\|_1\leq 1}}\sum_{\lambda} c^{(\lambda)} s_n^{(\lambda)}(x),
\end{equation}
where for each partition $\lambda = (\lambda_1 \geq \dots \geq \lambda_k > 0)$ we set $s_n^{(\lambda)}(x)=\prod_{i=1}^k s_n^{(\lambda_{i})}(x)$ as the product of power sums.
Here we systematically construct lower bounds for symmetric functions which hold in all dimensions in terms of freely-symmetrized problems over $\mscr D$. 


\paragraph{Basis for $\mscr L=\mscr D$.} 
Consider any partition $\lambda=(\lambda_1\geq\ldots\geq\lambda_{\mathrm{len}(\lambda)}>0)$ where $\mathrm{len}(\lambda)$ denotes the length of $\lambda$ and $|\lambda| = \sum_i \lambda_i$.  We view $\lambda\in\mbb N^k$ for any $k\geq \mathrm{len}(\lambda)$ by zero-padding $\lambda$, and let $\msf S_k$ act on $\lambda$ by coordinate permutations. 
The monomial sums are defined by sums over orbits of such partitions for $x \in \RR^k$:
\begin{equation}\label{eq:monomial-sums}
    m_k^{(\lambda)}(x) = \sum_{\lambda'\in\msf S_k\lambda}x^{\lambda'}.
\end{equation}
For example, for each $m \in \mbb N$ we have $m_k^{(m)}=s_k^{(m)}$ and the $m$'th elementary symmetric polynomial $m_k^{(1^m)}(x)=\sum_{i_1<\ldots<i_m}x_{i_1}\cdots x_{i_m}$, where $(1^m)=(1,\ldots,1)$ is the partition consisting of $m$ ones.
Note that $\deg m_k^{(\lambda)}=|\lambda|$.  For each $k\geq \mathrm{len}(\lambda)$ the monomial sums $m_k^{(\lambda)}$ are scalar multiples of the monomial means $\bar m_k^{(\lambda)}$ in~\eqref{eq:monomial_means}. 
In particular, for each $k\geq d$ the monomial sums $\{m_k^{(\lambda)}\}_{|\lambda|\leq d}$ form a basis for $\RR[x_1,\ldots,x_k]_{\leq d}^{\msf S_k}$.

\paragraph{Basis for $\mscr U=\mscr Z$.} 
We proceed to derive $\mbb E[m_k^{(\lambda)}\circ L_{k,n}]$ for each $k$.
To state the resulting expressions, we recall the refinement partial order on partitions. 
Specifically, for two partitions $\mu,\lambda$, we set $\lambda\leq_R\mu$ if $\lambda$ is a \emph{refinement} of $\mu$, meaning that there exists a surjection $f\colon[\mathrm{len}(\lambda)]\to[\mathrm{len}(\mu)]$, denoted $f(\lambda)=\mu$, such that $\sum_{j\in f^{-1}(i)}\lambda_j=\mu_i$ for each $i\in[\mathrm{len}(\mu)]$. 
For example, $(1,1,1)\leq_R(2,1)\leq_R(3)$ are all the refinements of $(3)$.
Define the integer $R_{\lambda,\mu}$ to be the number of such surjections $f$. 
The refinement partial order and integers $R_{\lambda,\mu}$ are classical objects of study, e.g., they relate power sums and monomial sums via~\cite[Prop.~7.7.1]{Stanley_Fomin_1999}
\begin{equation}\label{eq:R_trans_matrix}
    s_n^{(\lambda)}(x) = \sum_{\lambda\leq_R\mu}R_{\lambda,\mu}m_n^{(\mu)}(x),
\end{equation}
for any $n\geq \mathrm{len}(\lambda)$. 
Remarkably, these integers $R_{\lambda,\mu}$ also appear when applying our de Finetti map to monomial sums.
\begin{proposition}\label{prop:rep_thm_symmetric_funcs}
    For any $k\geq\mathrm{len}(\lambda)$, we have that $p_{n}^{(k,\lambda)}(x) = \mbb E[m_k^{(\lambda)}(L_{k,n}x)]$ for $x \in \RR^n$ is given by
        \begin{equation}\label{eq:fancy_symmetric_funcs}
            p_{n}^{(k,\lambda)}(x) = \sum_{\mu\leq_R\lambda}\frac{\lambda!}{\mu!}\cdot\frac{(k)_{\mathrm{len}(\lambda)}}{k^{\mathrm{len}(\mu)}}\cdot\frac{R_{\mu,\lambda}}{R_{\lambda,\lambda}}\cdot m_n^{(\mu)}(x),
        \end{equation}
        where $(k)_{\ell}=\frac{k!}{(k-\ell)!}$ and $\lambda!=\prod_{i=1}^{\mathrm{len}(\lambda)}\lambda_i!$. 
\end{proposition}
\begin{proof}
We have $L_{k,n}x=\sum_{i=1}^k\left(\sum_{j\in F_{k,n}^{-1}(i)}x_j\right)e_i$ where $F_{k,n}\colon[n]\to[k]$ is uniformly random, by definition of the de Finetti map in Section~\ref{sec:free_poly_opts}. 
    Then
    \begin{equation*}
        p_n^{(k,\lambda)}(x) = \frac{(k)_{\mathrm{len}(\lambda)}}{R_{\lambda,\lambda}}\mbb E\left[\prod_{i=1}^{\mathrm{len}(\lambda)}\left(\sum_{j\in F_{k,n}^{-1}(i)}x_j\right)^{\lambda_i}\right]=\frac{(k)_{\mathrm{len}(\lambda)}}{R_{\lambda,\lambda}}\cdot\frac{1}{k^n}\sum_{f\colon[n]\to[k]}\prod_{i=1}^{\mathrm{len}(\lambda)}\left(\sum_{j\in f^{-1}(i)}x_j\right)^{\lambda_i},
    \end{equation*}
    where the first equality follows since $m_k^{(\lambda)}$ is a sum of $(k)_{\mathrm{len}(\lambda)}/R_{\lambda,\lambda}$ monomials by~\cite[Cor.~7.7.2, Prop.~7.8.3]{Stanley_Fomin_1999}, which have the same expectations by symmetry.
    To simplify the above expression, note that
    \begin{equation*}
       \prod_{i=1}^{\mathrm{len}(\lambda)}\left(\sum_{j\in f^{-1}(i)}x_j\right)^{\lambda_i} = \prod_{i=1}^{\mathrm{len}(\lambda)}\sum_{\substack{\mu^{(i)}\in\mbb N^{f^{-1}(i)}\\ |\mu^{(i)}|=\lambda_i}}\frac{\lambda_i!}{\mu^{(i)}!}x^{\mu^{(i)}} = \sum_{\substack{\mu'\in\mbb N^n\\ f(\mu')=\lambda}}\frac{\lambda!}{\mu'!}x^{\mu'},
    \end{equation*}
    where the first equality is just the multinomial theorem. For the second equality, split a multi-index $\mu'\in\mbb N^n$ indexed by $[n]$ into multi-indices indexed by the fibers of $f$, i.e., $\mu^{(i)}=(\mu'_j)_{j\in f^{-1}(i)}$ for each $i\in[k]$. Now note that $x^{\mu'}=\prod_{i=1}^kx^{\mu^{(i)}}$ appears on the right-hand side if and only if $|\mu^{(i)}|=\lambda_i$, or more concisely, iff $f(\mu')=\lambda$, and each such $\mu'$ appears $\prod_{i=1}^{\mathrm{len}(\lambda)}\frac{\lambda_i!}{\mu^{(i)}!}=\frac{\lambda!}{\mu'!}$ times.
    Finally, 
    \begin{equation*}
        \sum_{f\colon[n]\to[k]}\prod_{i=1}^{\mathrm{len}(\lambda)}\left(\sum_{j\in f^{-1}(i)}x_j\right)^{\lambda_i} = \sum_{\mu\leq_R\lambda}k^{n-\mathrm{len}(\mu)}R_{\mu,\lambda}\frac{\lambda!}{\mu!}\sum_{\mu'\in\msf S_n\mu}x^{\mu'} = \sum_{\mu\leq_R\lambda}k^{n-\mathrm{len}(\mu)}\frac{\lambda!}{\mu!}R_{\mu,\lambda}m_n^{(\mu)}(x),
    \end{equation*}
    where the first equality follows since if $\mu$ is the unique partition obtained as a permutation of a multi-index $\mu'$, then $\mu'!=\mu!$ and there are $R_{\mu,\lambda}k^{n-\mathrm{len}(\mu)}$ maps $f\colon[n]\to[k]$ such that $f(\mu')=\lambda$. 
    Combining the two displayed equations above yields~\eqref{eq:fancy_symmetric_funcs}. 
\end{proof}
We conclude from Propositions~\ref{prop:free_sym_and_free_descr_isom} and~\ref{prop:calc_for_stability} that $\{(p_n^{(k,\lambda)})_n\}_{|\lambda| = d}$ is a basis for freely-described homogeneous polynomials of degree $d$ for each $k\geq d$. These bases do not appear to have been studied before.  

\paragraph{Dual cost.} For each $k\geq d$, expand $p_n=\sum_\lambda c^{(k,\lambda)}p_n^{(k,\lambda)}$ in the $k$th basis from Proposition~\ref{prop:rep_thm_symmetric_funcs}.
We can obtain $c^{(k,\lambda)}$ by first changing basis from power sums to monomial sums using~\eqref{eq:R_trans_matrix}, and then changing to the bases $\{p_n^{(k,\lambda)}\}$ using~\eqref{eq:fancy_symmetric_funcs}.  The $k$'th dual cost is then given by $q_k = \sum_\lambda c^{(k,\lambda)} m_k^{(\lambda)}$,
and the lower bounds on~\eqref{eq:pwr_sums_prob} are given by $\ell_n=\inf_{\substack{x\in\RR^n\\\|x\|_1\leq 1}}q_n(x)$.
We then have $\ell_n\leq \ell_{nk}$ for all $n,k$ and $u_{nk}-\ell_{nk}\lesssim 1/n$ for fixed $k$ and growing $n$.  

\paragraph{$\FS$-compatible sets.} The sequences of constraint sets $(\Omega_n\subseteq \RR^{n})$ for which our framework is applicable consist of $\msf S_n$-invariant compact sets satisfying $(x_1+\ldots+x_k,\dots,x_{nk-k+1}+\ldots+x_{nk})\in\Omega_n$ if $x\in\Omega_{nk}$ and $(x,0) \in\Omega_{n+1}$ if $x\in\Omega_n$.  These assumptions are satisfied by the sequence of simplices $\Omega_n=\Delta^{n-1}$ and $\ell_1$-balls, but not by $\ell_p$-balls for any $p>1$ (this is the reason we use the $\ell_1$-balls in~\eqref{eq:pwr_sums_prob} instead of the more commonly-used $\ell_2$-balls).

\begin{example}[Quadratic symmetric functions]
    Any homogeneous degree-2 symmetric polynomial can be written in the power sums basis as
    \begin{equation*}
        p_n(x) = \alpha_1s_n^{(2)}(x) + \alpha_2s_n^{(1,1)}(x).
    \end{equation*}
    For each $k\geq 2$, our bases~\eqref{eq:fancy_symmetric_funcs} are given by
    \begin{equation*}\begin{aligned}
        &p_n^{(k,(1,1))}(x) = \frac{k(k-1)}{k^2}m_n^{(1,1)}(x) = \frac{1}{2}\left(1-\frac{1}{k}\right)(s_n^{(1,1)}(x)-s_n^{(2)}(x)),\\
        &p_n^{(k,(2))}(x) = m_n^{(2)}(x) + \frac{2}{k}m_n^{(1,1)}(x) = \left(1-\frac{1}{k}\right)s_n^{(2)}(x) + \frac{1}{k}s_n^{(1,1)}(x).
    \end{aligned}\end{equation*}
    This gives a linear system we can invert to represent $p_n$ in terms of our basis functions, yielding
    \begin{equation*}
        p_n = (\alpha_1+\alpha_2)p_n^{(k,(2))} + 2\left(\alpha_2-\frac{1}{k-1}\alpha_1\right)p_n^{(k,(1,1))}.
    \end{equation*}
    The dual cost is then given by
    \begin{equation*}\begin{aligned}
        q_n &= (\alpha_1+\alpha_2)m_n^{(2)} + 2\left(\alpha_2-\frac{1}{n-1}\alpha_1\right)m_n^{(1,1)} = \left(\frac{n \alpha_1}{n-1}\right)s_n^{(2)} + \left(\frac{\alpha_2(n-1) - \alpha_1}{n-1}\right)s_n^{(1,1)}.
    \end{aligned}\end{equation*}
    As a simple example of our lower bounds, consider
    \begin{equation*}
        u_n = \inf_{x\in\Delta^{n-1}}s_n^{(2)}(x) = \frac{1}{n},
    \end{equation*}
    for which $u_{\infty}=0$. For all $n \geq 2$, we get the lower bounds
    \begin{equation*}
        \ell_n = \inf_{x\in\Delta^{n-1}}\frac{n}{n-1}s_n^{(2)}(x) - \frac{1}{n-1}s_n^{(1,1)}(x)=0.
    \end{equation*}
\end{example}
For higher degree problems, we automate the derivation of $q_n$ from $p_n$, as the relevant change-of-basis matrices can be formed and inverted computationally. The code to do so is available on \href{https://github.com/eitangl/anyDimPolyOpt/blob/main/trans_matrix_p2q.m}{our GitHub}, using the code of~\cite{transition_mats} to form the matrices $R_{\mu,\lambda}$. We utilize this code below to prove lower bounds that hold in all dimensions for one of the examples of~\cite{acevedo2024symmetric}.
\begin{example}
    In~\cite[Thm.~3.6]{acevedo2024symmetric}, the authors show that \begin{equation}\label{eq:bad_quartic}
        p_n=4s_n^{(4)}-\frac{139}{20}s_n^{(3,1)} + 4s_n^{(2,2)} - 5s_n^{(2,1,1)} + 4s_n^{(1^4)},
    \end{equation}
    is a symmetric function which is nonnegative for all $n$ but is not a sum of squares for any $n\geq 4$.  
    Since $p_n$ is homogeneous of degree 4, certifying that $p_n\geq0$ for all $n$ amounts to certifying $u_{\infty}\geq0$ where $u_n=\inf_{\|x\|_1\leq 1}p_n(x)$.
    Our framework yields finite-dimensional polynomial bounds $\ell_n=\inf_{\|x\|_1\leq 1}q_n(x)$ where
    \begin{equation*}\begin{aligned}
        q_n &= \tfrac{n^3}{(n-1)(n-2)(n-3)}\Bigg[4(1+3\tfrac{1}{n})s_n^{(4)} - \tfrac{1}{20}(139-97\tfrac{1}{n}+960\tfrac{1}{n^2})s_n^{(3,1)} + 4(1-9\tfrac{1}{n}+9\tfrac{1}{n^2})s_n^{(2,2)}\\ &-\tfrac{1}{20}(100-757\tfrac{1}{n}-69\tfrac{1}{n^2})s_n^{(2,1,1)} + \frac{1}{10}(2-\tfrac{1}{n})(20-85\tfrac{1}{n}+3\tfrac{1}{n^2})s_n^{(1^4)}\Bigg],
    \end{aligned}\end{equation*}
    for $n\geq 4$.
    Running both a nonconvex solver (Matlab's \texttt{fmincon}) and a degree-$6$ sums-of-squares relaxation, we find the lower bounds $\ell_n$ in Table~\ref{tab:lower_bds_bad_quartic} up to the displayed digits. 
    Note that we are only guaranteed that $\ell_{nk}\geq \ell_n$ but the values we obtain increase monotonically with $n$.
    \begin{table}[h]
    \centering
    \begin{tabular}{c|ccccc}
        $n$ & 4 & 5 & 6 & 7 & 8\\ \hline
        $\ell_n$ & -0.8403 & -0.4023 & -0.2541 & -0.1817 & -0.1396
    \end{tabular}
    \caption{Lower bounds over $\ell_1$-balls for~\eqref{eq:bad_quartic}. Here $u_{\infty}=0$.}
    \label{tab:lower_bds_bad_quartic}
    \end{table}
\end{example}
Finally, we return to Example~\ref{ex:SAGE_nonSOS} from Section~\ref{sec:intro}, prove that the symmetric function there is not a sum of squares (SOS), and apply our framework to certify its nonnegativity.
\begin{example}\label{ex:SAGE_nonSOS_explicit}
    Note that the symmetric function $(p_n=f(s_1,s_2,s_4,s_6))$ in~\eqref{eq:sage_nonSOS_poly} is not SOS in any dimension $n\geq 2$, since for $n=2$ it reduces to $p_2(x)=x_1^4x_2^2+x_1^2x_2^4-\frac{3}{2}x_1^2x_2^2+1$. If this were SOS then the coefficient of $x_1^2x_2^2$ would be nonnegative~\cite[Exer.~3.97]{SOS_chapter}, a contradiction. Similarly, if $p_n$ admitted an SOS decomposition for $n>2$ then restricting this decomposition to vectors of the form $(x_1,x_2,0,\ldots,0)$ would yield an SOS decomposition for $p_2$, a contradiction.

    Applying our framework to the freely-described problems $u_n=\inf_{x\in\RR^n}p_n(x)$, we get the lower bounds $\ell_n=\inf_{x\in\RR^n}q_n(x)$ with
    \begin{equation*}\begin{aligned}
        q_n = \tfrac{n}{2(n-1)}\Bigg[&m_n^{(4,2)} + m_n^{(4,1^2)} + 2m_n^{(3,2,1)} + \tfrac{6(n-2)}{n-3}m_n^{(3,1^3)}+9m_n^{(2^3)}+\tfrac{3(3n-10)}{n-3}m_n^{(2^2,1^2)}\\&  + \tfrac{3(7n-18)}{(n-3)}m_n^{(2,1^4)} + \tfrac{45(n-2)(n-6)}{(n-3)(n-5)}m_n^{(1^6)}-3m_n^{(2^2)} -3m_n^{(2,1^2)} -\tfrac{9(n-2)}{n-3}m_n^{(1^4)}\Bigg]+1,
    \end{aligned}\end{equation*}
    for $n=2,4$ and $n\geq 6$. In particular, for $n=2$ any monomial sum $m_n^{(\lambda)}$ with $\mathrm{len}(\lambda)\geq3$ vanishes, so we get $q_2(x)=m_2^{(4,2)}(x)-3m_2^{(2^2)}(x)+1$, which is precisely the Motzkin polynomial.
\end{example}
We remark that Proposition~\ref{prop:rep_thm_symmetric_funcs} only guarantees the existence of $q_k\in\RR[\vct V_k]$ satisfying $(p_n)=(\mbb Eq_k\circ L_{k,n})$ when $k\geq \deg p_n$. The denominators of the coefficients of $q_k$ in the above examples vanish precisely for those $k<\deg p_n$ for which there is no such $q_k$.


\subsection{Inequalities in Graph Homomorphism Densities}\label{sec:graph_densities}
\paragraph{Model problem.} We consider the problem of minimizing linear expressions in weighted graph homomorphism densities as in Example~\ref{ex:graph_densities}, which are freely-described problems over $\mathrm{Sym}^2\mscr D$.  A graph $X$ on $n$ vertices with weights on both its edges and vertices is represented by its adjacency matrix, which we also denote $X\in\mbb S^n$ by abuse of notation, where $X_{i,i}$ is the weight on vertex $i$ and $X_{i,j}$ is the weight on edge $\{i,j\}$.  An unweighted graph $H$, possibly containing multiple edges and self-loops, is denoted $H\in \mbb S^k(\mathbb{N})$, and the homomorphism number of $H$ in $X$ is defined by~\cite[Eq.~(5.3)]{lovasz2012large}
\begin{equation}\label{eq:hom_nums}
\begin{aligned}
    \mathrm{hom}(H;X) = \sum_{f\colon [k]\to [n]}\prod_{\substack{1\leq i\leq j\leq k\\ H_{i,j}\neq0}}X_{f(i),f(j)}^{H_{i,j}},
\end{aligned}
\end{equation}
which is a polynomial over $\mbb S^n$.  If $H$ has no isolated vertices, note that $\hom(H;X)$ is unchanged if we relabel the vertices or append extra isolated vertices to $X$; the former operation corresponds to conjugating the adjacency matrices by permutation matrices, and the latter to zero-padding the adjacency matrices. The homomorphism density of $H$ in $X$ is defined as:
\begin{equation}
    \begin{aligned}
        t(H;X) = \frac{\mathrm{hom}(H;X)}{n^k},
    \end{aligned}
\end{equation}
where $n^k$ is the number of all maps $[k]\to [n]$.  This is once again a polynomial over $\mbb S^n$.  The homomorphism density only depends on the isomorphism class $[H]$ of $H$ with all isolated vertices removed, and we denote the number of non-isolated vertices in $H$ by $|V(H)|$ and the number of edges by $|H|=\sum_{i\leq j}H_{i,j}$.  Certifying inequalities in homomorphism densities over unweighted graphs amounts to considering the sequence of problems
\begin{equation}\label{eq:graph_densities_prob}
    u_n=\inf_{X\in\mbb S^n(\{0,1\})}\sum\nolimits_j c_jt(H_j;X),
\end{equation}
for some unweighted graphs $H_j$, and certifying that $u_\infty\geq0$. We remark that problems of the form~\eqref{eq:graph_densities_prob} encompass not only minimization of linear combinations but also of \emph{polynomials} in homomorphism densities, since $t(H_1;X)\cdot t(H_2;X)=t(H_1\sqcup H_2;X)$ where $H_1\sqcup H_2 = \mathrm{blkdiag}(H_1,H_2)$ denotes the disjoint union of the two multigraphs. We proceed to show that these problems are freely-described over $\mathrm{Sym}^2\mscr D$. Conversely, we show that any such freely-described problem can be written as~\eqref{eq:graph_densities_prob} for more general constraint sets $\Omega_n$ and for some $H_j\in\mbb S^k(\NN)$.

\paragraph{Basis for $\mscr L=\mathrm{Sym}^2\mscr Z$.} 
Define the injective homomorphism number $\mathrm{inj}(H;X)$ similarly to~\eqref{eq:hom_nums} except that we only sum over injective maps $[k]\hookrightarrow [n]$, and define $t_{\mathrm{inj}}(H;X) = \frac{\mathrm{inj}(H;X)}{(n)_{k}}$ to be the fraction of injective maps $[k]\hookrightarrow [n]$ that are homomorphisms, where $(n)_k=\frac{n!}{(n-k)!}$ is the total number of these injective maps.  
For $k\geq|V(H)|$, consider the monomial $X^H = \prod_{1\leq i\leq j\leq k}X_{i,j}^{H_{i,j}}\in\RR[\mbb S^k]_{|H|}$ for $X \in \mbb S^k$. 
Recalling the action $\rho$ of maps between finite sets from Section~\ref{sec:deFinetti_intro}, we have the following polynomial by symmetrizing $X^H$
\begin{equation}\label{eq:tinj_is_sym}
    t_{\mathrm{inj}}(H;X) = \frac{(k-|V(H)|)!}{k!}\sum_{f\colon[|V(H)|]\hookrightarrow[k]}(\rho(f)X)^H = \frac{1}{k!}\sum_{g\in\msf S_k}((\zeta_{k,|V(H)|}^{\otimes 2})^\star g\cdot X)^H = \mathrm{sym}_k^{\mathrm{Sym}^2\mscr Z}X^H,
\end{equation}
since a map $f\colon[|V(H)|]\to[k]$ is injective precisely when $f=g\iota_{k,|V(H)|}$ for some $g\in\msf S_k$, and $\rho(\iota_{k,|V(H)|})=(\zeta_{k,|V(H)|}^{\otimes 2})^\star$.  Thus, the polynomials $t_{\mathrm{inj}}(H;\cdot)$ are symmetrizations of monomials, in analogy to the monomial means from Section~\ref{sec:power_means}.  For $k\geq |V(H)|$ denote by $q_k^{[H]}(X)=t_{\mathrm{inj}}(H;X)$ the restriction of $t_{\mathrm{inj}}(H;X)$ to $X \in \mbb S^k$, which only depends on the isomorphism class $[H]$ of $H$. 
Observe that the collection $\{q_k^{[H]}\}_{|H|\leq d}$ as $[H]$ ranges over isomorphism classes of graphs with at most $d$ edges forms a basis for $\RR[\mbb S^k]_{\leq d}^{\msf S_k}$ for all $k\geq 2d$.
Indeed, any monomial over $\RR[\mbb S^k]_{\leq d}$ is of the form $X^H$ for some graph $H$ with $|H|\leq d$, which must have at most $2d$ non-isolated vertices, so symmetrizations of these monomials form a basis for $\msf S_k$-invariants. 
Furthermore,~\eqref{eq:tinj_is_sym} shows that $q_K^{[H]}=\mathrm{sym}_K^{\mathrm{Sym}^2\mscr Z}q_k^{[H]}$ for all $k\leq K$. Thus, the collection $\{(q_k^{[H]})_{k\geq 2d}\}_{|H|\leq d}$ is a basis for freely-symmetrized elements over $\mathrm{Sym}^2\mscr Z$ of degree $\leq d$, again recovering a special case of~\cite{levin2023free}.



\paragraph{Basis for $\mscr U=\mathrm{Sym}^2\mscr D$.}
To derive $\mbb E[q_k^{[H]}\circ L_{k,n}]$, observe that $(L_{k,n}X)_{i,j}=X_{F_{n,k}(i),F_{n,k}(j)}$ for $X \in \mathbb{S}^n$, where $F_{n,k}(1),\ldots,F_{n,k}(k)$ are sampled iid uniformly from $[n]$. 
Therefore,
\begin{equation*}
    \mbb E[t_{\mathrm{inj}}(H;L_{k,n}X)] = \mbb E[\mathrm{sym}_{\msf S_k}(L_{k,n}X)^H] = \mbb E[(L_{k,n}X)^H] = \frac{1}{n^k}\sum_{f\colon[k]\to[n]}(\rho(f) X)^H = t(H;X),
\end{equation*}
where the first equality was shown above, and the second follows since $gL_{k,n}\overset{d}{=}L_{k,n}$ for all $g\in\msf S_k$.  Denoting by $p_n^{[H]}(X)=t(H;X)$ the restriction of $t(H;X)$ to $X \in \mbb S^n$, we conclude that $(p_n^{[H]})$ is freely-described over $\mathrm{Sym}^2\mscr D$ for each $H$, and that the collection $\{(p_n^{[H]})_n\}_{|H|\leq d}$ forms a basis for such freely-described polynomials of degree at most $d$ by Propositions~\ref{prop:free_sym_and_free_descr_isom} and~\ref{prop:calc_for_stability}.

\begin{remark}[Connection to graphons]
    The above conclusions could have been derived using the connection between our generalized de Finetti theorem and the theory of graphons explained in Example~\ref{ex:graphon_deFin}. Briefly, a graphon is a symmetric, bounded, measurable function $W\colon[0,1]^2\to\RR$, and we can define the homomorphism densities $t(H;W)$ by an integral analog of~\eqref{eq:hom_nums}, see~\cite[\S7.2]{lovasz2012large}. For any $X\in\mbb S^n$ we can define a step graphon $W_X$ as in Example~\ref{ex:graphon_deFin}, and note that $t(H;X)=t(H;W_X)$, showing that $t(H;X)$ is freely-described over $\mathrm{Sym}^2\mscr D$ as $X$ and $X\otimes\mathbbm{1}\mathbbm{1}^\top$ define the same step graphon. Furthermore, note that $L_{k,n}X\overset{d}{=}(W_X(x_i,x_j))_{i,j\in[k]}$ where $x_1,\ldots,x_k$ are sampled iid uniformly from $[0,1]$. 
    The definition of $t(H;W)$ shows that $\mbb E[t_{\mathrm{inj}}(H;L_{k,n}X)]=\mbb E[(L_{k,n}X)^H]=t(H;X)$.
\end{remark}

\paragraph{Dual cost.} For a freely-described problem expressed in terms of homomorphism densities as in~\eqref{eq:graph_densities_prob}, form the dual cost $q_n(X)=\sum_jc_jt_{\mathrm{inj}}(H_j;X)$ for $n\geq V=\max_j|V(H_j)|$. 
The lower bounds on~\eqref{eq:graph_densities_prob} are then given by $\ell_n=\inf_{X\in\mbb S^n(\{0,1\})}q_n(X)$, which satisfy $\ell_n\leq \ell_{n+1}$ and $u_n-\ell_n\leq\frac{V(V-1)\sum_j|c_j|}{n}$ by Theorem~\ref{thm:free_sym_bds}, since $|t_{\mathrm{inj}}(H;X)|\leq 1$ for $X\in\mbb S^n(\{0,1\})$. This is the convergence bound from Example~\ref{ex:graph_densities_intro} in Section~\ref{sec:intro}.

\paragraph{$\FS$-compatible sets.} The sequences of constraint sets $(\Omega_n\subseteq \mbb S^n)$ to which our framework is applicable consist of $\msf S_n$-invariant sets satisfying $\zeta_{n,n-1}^\star X\zeta_{n,n-1}\in\Omega_{n-1}$ and $\delta_{nk,n}X\delta_{nk,n}^\star \in\Omega_{nk}$ whenever $X\in\Omega_n$, where $\zeta$ and $\delta$ are the zero-padding and duplication embeddings from Section~\ref{sec:consist_seqs}.  Examples of constraints satisfying the above hypotheses include entrywise constraints, such as the set of simple graphs $\{X\in\mbb S^n(\{0,1\}):\diag(X)=0\}$, and $H$-homomorphism-free simple graphs for any graph $H$ (i.e., sets of simple graphs $G$ with $t(H;G)=0$). Non-examples include sets specified by spectral constraints such as $\lambda_{\max}(X)\leq 1$ or $\mathrm{Tr}(X)\leq 1$, as well as their normalized counterparts $\lambda_{\max}(X)/n\leq 1$ or $\mathrm{Tr}(X)/n\leq 1$. 


\begin{example}[Goodman's Theorem]\label{ex:goodman}
    Goodman's theorem states that $t(K_3;X)-2t(K_2\sqcup K_2;X)+t(K_2;X)\geq0$ for all simple graphs $X$ of all sizes, or equivalently, that $u_n=\inf_{X\in\mbb S^n(\{0,1\})} t(K_3;X)-2t(K_2\sqcup K_2;X)+t(K_2;X)$ is equal to zero for all $n$ (since $X=K_n$ gives objective value zero). Our framework yields the lower bounds $\ell_n=\inf_{X\in\mbb S^n(\{0,1\})} t_{\mathrm{inj}}(K_3;X)-2t_{\mathrm{inj}}(K_2\sqcup K_2;X) + t_{\mathrm{inj}}(K_2;X)$ for $n\geq 4$.  That is, the costs in the lower bounds are simply given by replacing $t$ by $t_{\mathrm{inj}}$.
    By Tur\'an's theorem, the minimizer is attained at the complete bipartite graph $K_{\lfloor n/2\rfloor, \lceil n/2\rceil}$, so
    \begin{equation*}
        \ell_n = \begin{cases} -\frac{n}{2(n^2-4n+3)} & n \textrm{ even},\\ -\frac{n+1}{2(n^2-2n)} & n\textrm{ odd}.\end{cases}
    \end{equation*}
    We thus have $u_n-\ell_n =  \frac{1}{2n} + o(\frac{1}{n})$, while the bound from Theorem~\ref{thm:free_sym_bds} gives $u_n-\ell_n\leq \frac{48}{n}$.
\end{example}

\begin{example}\label{ex:ramsey_mult}
    The \emph{Ramsey multiplicity} of a graph $H$ is defined as $\mathrm{mult}(H) = u_{\infty}$ where $u_n=\inf_{X\in\mbb S^n(\{0,1\})}t(H;X) + t(H;1-X)$. Burr and Rosta~\cite{burr_rosta} conjectured that $\mathrm{mult}(H)=2^{1-|H|}$ for all graphs $H$, generalizing a conjecture of Erd\H{o}s~\cite{Erdös1969} for complete graphs. 
    Both conjectures turned out to be false. Burr and Rosta's conjecture was disproved by Sidorenko~\cite{sidorenko1989cycles}, who showed that the triangle with a pendant edge $H$, the graph on $V(H)=[4]$ with edges $E(H)=\{\{1,2\},\{2,3\},\{2,4\},\{3,4\}\}$, has $\mathrm{mult}(H)\leq 3/4-4\sqrt{2}/9<1/8$. 
    Since then, there has been much follow-up work yielding better upper bounds on Ramsey multiplicities, especially for complete graphs~\cite{parczyk2024new}. We use our framework to produce \emph{lower} bounds on the Ramsey multiplicity of Sidorenko's counterexample above. 
    Our bounds read $\ell_n = \inf_{X\in\mbb S^n(\{0,1\})}t_{\mathrm{inj}}(H;X)+t_{\mathrm{inj}}(H;1-X)$, and we compute $\ell_n$ by exhaustively searching over all simple graphs with at most 10 vertices, obtaining the bounds in Table~\ref{tab:ramsey_mult}. 
    
    \begin{table}[h]
    \centering
    \begin{tabular}{c|ccccc}
        $n$ & $\leq 6$ & 7, 8 & 9, 10 \\ \hline
        $\ell_n$ & 0 & 0.0286 &0.0476
    \end{tabular}
    \caption{Lower bounds for Ramsey multiplicity from Example~\ref{ex:ramsey_mult}.}
    \label{tab:ramsey_mult}
    \end{table}
\end{example}

\subsection{Inequalities in Graph Homomorphism Numbers}\label{sec:graph_numbers}
\paragraph{Model problem.} 
We study the problem of minimizing linear expressions in homomorphism \emph{numbers}, which are freely-described problems over $\mathrm{Sym}^2\mscr Z$.  For example, consider certifying nonnegativity of a linear expression $\sum_jc_j\mathrm{hom}(H_j;X)$ in homomorphism numbers over all weighted graphs $X$, where $\mathrm{hom}$ is defined in \eqref{eq:hom_nums} and we assume throughout this section that the $H_j$ have no isolated vertices, with the convention $\mathrm{hom}(K_1;X)= 1$.  If $|H_j|=d$ for all $j$, i.e., all graphs $H_j$ have the same number of edges $d$, then proving nonnegativity of the above expression is equivalent to considering
\begin{equation}\label{eq:inj_number_ineq}
    u_n=\inf_{\substack{X\in\mbb S^n\\ \|X\|_1\leq 1}}\sum\nolimits_jc_j\mathrm{hom}(H_j;X),
\end{equation}
and certifying $u_\infty\geq0$.
We apply our framework to derive lower bounds on $u_\infty$ in terms of freely-symmetrized problems over $\mathrm{Sym}^2\mscr D$.
The development in this section parallels that in Section~\ref{sec:symmetric_funcs}, as we observe below.

\paragraph{Basis for $\mscr L=\mathrm{Sym}^2\mscr D$.}
If $H\in\mbb S^k(\mbb N)$ denote by $X^H = \prod_{1\leq i\leq j\leq k}X_{i,j}^{H_{i,j}}\in\RR[\mbb S^k]_{|H|}$ as before. We view $H$ as an element of $\mbb S^n(\mbb N)$ for any $n\geq k$ by zero-padding it, and define the polynomial over $\mbb S^n$
\begin{equation*}
    m_n^{[H]}(X)=\sum_{G\in \msf S_n H}X^G,
\end{equation*}
which is the sum over all distinct monomials $X^G$ obtained by relabeling the vertices of $H$ by elements from $[n]$.
These are the analogs of monomial sums from Section~\ref{sec:symmetric_funcs}, with isomorphism classes of graphs $[H]$ replacing partitions $\lambda$; indeed, the monomial sums of Section~\ref{sec:symmetric_funcs} can be obtained by considering graphs with only self-loops whose adjacency matrix is $\mathrm{diag}(\lambda)$.  When $H$ and $X$ are simple graphs, the polynomial $m_n^{[H]}(X)$ counts the number of times $H$ appears as a subgraph of $X$ up to isomorphism (i.e., $m_k^{[H]}(H)=1$, for instance). More generally, we have
\begin{equation}\label{eq:m_to_inj_trans}
    m_n^{[H]}(X)=|\mathrm{Aut}(H)|^{-1}\mathrm{inj}(H;X),
\end{equation}
where $|\mathrm{Aut}(H)|$ is the number of automorphisms of $H$ without isolated vertices, which only depends on $[H]$.
Since $m_k^{[H]}(X)$ is a multiple of $t_{\mathrm{inj}}(H;X)$ for any $k\geq |V(H)|$, we conclude that the collection $\{m_k^{[H]}\}_{|H|\leq d}$ as $[H]$ ranges over isomorphism classes of graphs with at most $d$ edges forms a basis for $\RR[\mbb S^k]_{\leq d}^{\msf S_k}$ for all $k\geq 2d$, as the same is true for $t_{\mathrm{inj}}$ as shown in Section~\ref{sec:graph_densities}.


\paragraph{Basis for $\mscr U=\mathrm{Sym}^2\mscr Z$.}
In order to derive $\mbb E[m_k^{[H]}\circ L_{k,n}]$, we generalize the refinement partial order on partitions from Section~\ref{sec:symmetric_funcs} to graphs and prove an analog of Proposition~\ref{prop:rep_thm_symmetric_funcs}. This refinement partial order on graphs does not seem to have been studied before.  If $X\in\mbb S^n$ and $f\colon[n]\to[k]$, let $P_f=(f^{-1}(1),\ldots,f^{-1}(k))$ be the ordered partition of $[n]$ induced by $f$ and denote $X/P_f=\rho(f)^\star X = \sum_{i\leq j}X_{i,j}E_{f(i),f(j)}$ where $E_{i,j}=e_ie_i^\top$ if $i=j$ and $E_{i,j}=e_ie_j^\top + e_je_i^\top$ otherwise. 
The reason for this notation is that $\rho(f)^\star X$ is the adjacency matrix of the quotient graph $X/P_f$ obtained by identifying all vertices of $X$ in each partition of $P_f$, with the weight of an edge $\{P_i,P_j\}$ being the sum of the weights in $X$ of the edges $\{k,m\}$ with $k\in P_i$ and $m\in P_j$. See Figure~\ref{fig:graph_quotient} for an example, and see~\cite[Eq.~(6.2)]{lovasz2012large} for an example usage of this quotient. 
\begin{figure}[h]
    \centering
    \includegraphics[width=0.4\linewidth]{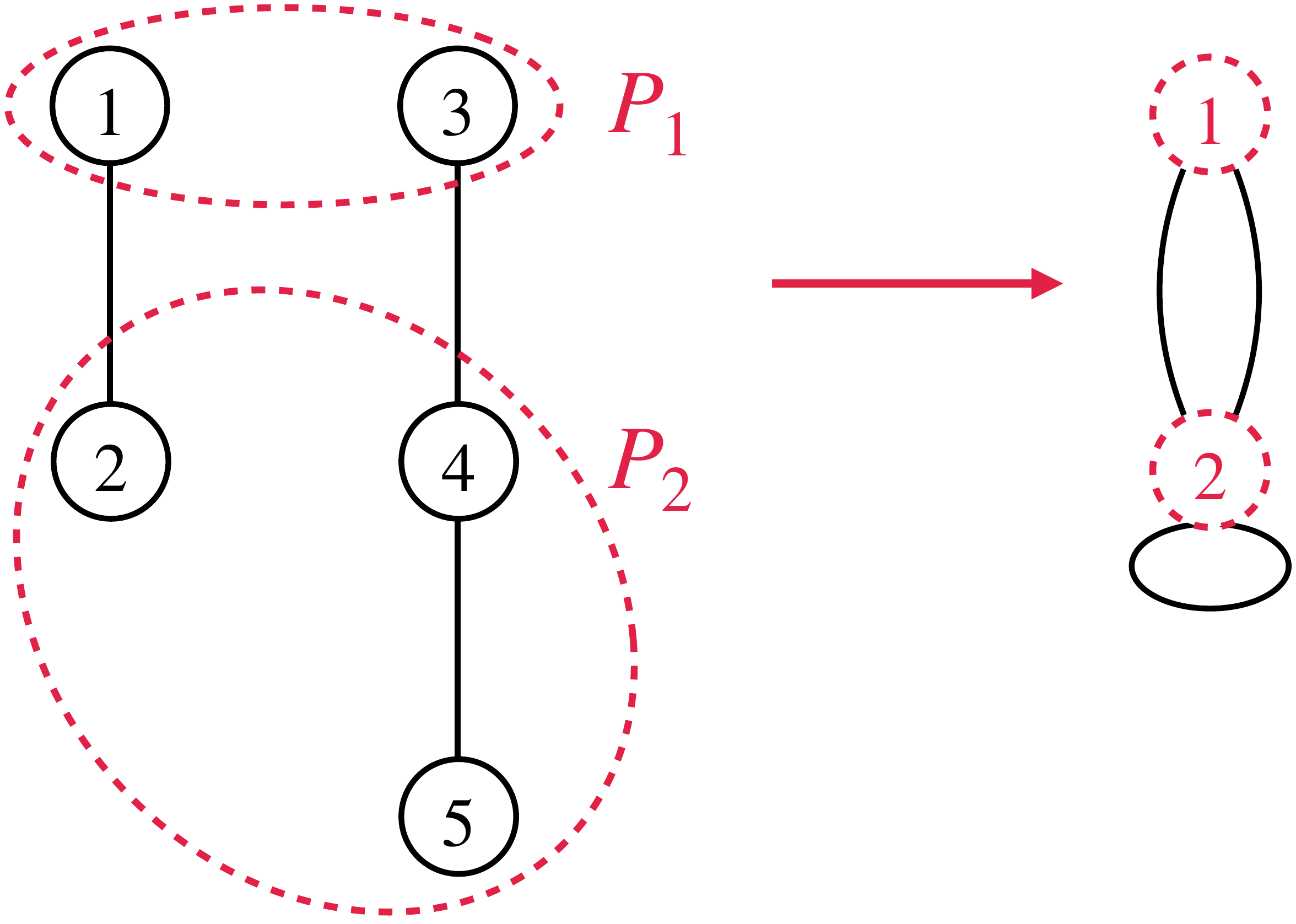}
    \caption{Example of graph quotient. Note that quotients of a simple graph need not be simple.}
    \label{fig:graph_quotient}
\end{figure}
We denote $G\leq_R H$ if $G$ is a \emph{refinement} of $H$, meaning that there is a surjective $f\colon V(G)\to V(H)$ such that $H = G/P_f$. For example, we show all refinements of $K_3$ in Figure~\ref{fig:triangle_refinements}. Note that if $G\leq_R H$ then $\sigma G\leq_R \tau H$ for any permutations $\sigma,\tau$, so we also write $[G]\leq_R[H]$. 
If $G\leq_R H$ and neither $G$ nor $H$ have isolated vertices, we let $R_{[G],[H]}$ be the number of such surjective maps $f$. Note that this extends our notation for partitions in Section~\ref{sec:symmetric_funcs} upon identifying partitions with diagonal matrices.
In particular, note the graph analog of the relation~\eqref{eq:R_trans_matrix} between power sums and monomial sums
\begin{equation}\label{eq:R_graph_trans_matrix}
    \mathrm{hom}(H;X) = \sum_{[H]\leq_R [G]}|\mathrm{Aut}(G)|^{-1}R_{[H],[G]}\mathrm{inj}(G;X) = \sum_{[H]\leq_R [G]}R_{[H],[G]}m_{[G]}(X),
\end{equation}
which follows since each homomorphism $f\colon H\to X$ can be factored into two homomorphisms $H\to H/P_f\hookrightarrow X$, and conversely every composition of a quotient map $H\to G$ with an injective homomorphism $G\hookrightarrow X$ gives a homomorphism $H\to X$ unique up to automorphisms of $G$. This shows that the homomorphism numbers are the graph analogs of power sums. Indeed, if $H=(k)$ is the graph on a single vertex with $k$ self loops, then $\mathrm{hom}(H;X)=\sum_{i=1}^nX_{i,i}^k$.
\begin{figure}[h]
    \centering
    \includegraphics[width=0.5\linewidth]{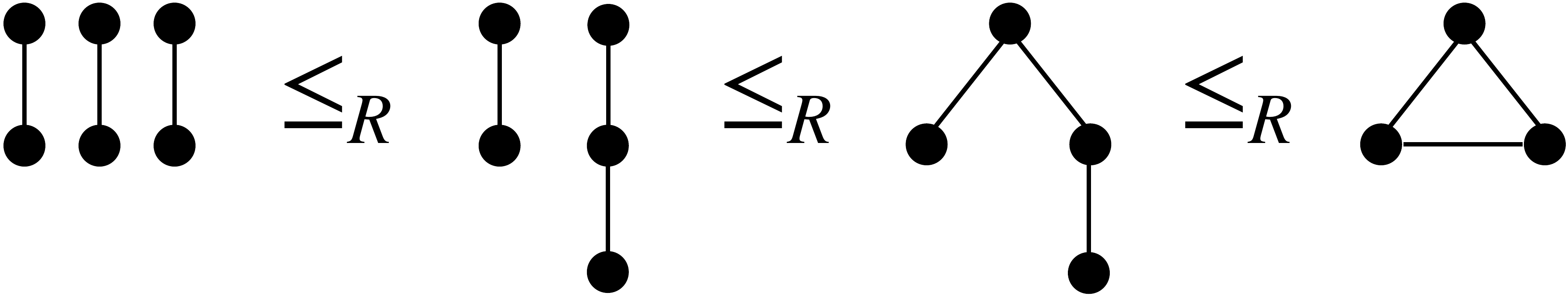}
    \caption{All refinements of $K_3$. Note that refinements of a simple graph are simple.}
    \label{fig:triangle_refinements}
\end{figure}
Using the above refinement order and the associated integers $R_{[G],[H]}$, we are ready to derive $\mbb E[m_k^{[H]}\circ L_{k,n}]$ in analogy with Proposition~\ref{prop:rep_thm_symmetric_funcs}.
\begin{proposition}\label{prop:graph_numbers}
    For any $k\geq |V(H)|$, we have that $p_{n}^{(k,[H])}(X) = \mbb E[m_k^{[H]}(L_{k,n}X)]$ for $X \in \mbb S^n$ is given by
    \begin{equation}\label{eq:fancy_symmetric_graph_funcs}
        p_n^{(k,[H])}(X) =  \sum_{[G]\leq_R [H]}\frac{H!}{G!}\cdot\frac{(k)_{|V(H)|}}{k^{|V(G)|}}\cdot\frac{R_{[G],[H]}}{R_{[H],[H]}}m_n^{[G]}(X),
    \end{equation}
    where $H!=\prod_{i\leq j}H_{i,j}!$.  As $[H]$ ranges over classes of graphs with at most $d$ edges, each of the collections $\{(p_n^{(k,[H])})_n\}_{|H|\leq d}$ where $k\geq 2d$, $\{(m_n^{[H]})_n\}_{|H|\leq d}$, $\{(\mathrm{inj}(H;\cdot)|_{\mbb S^n})_n\}_{|H|\leq d}$, and $\{(\mathrm{hom}(H;\cdot)|_{\mbb S^n})_n\}_{|H|\leq d}$ forms a basis for freely-described polynomials of degree at most $d$ over $\mathrm{Sym}^2\mscr Z$.
\end{proposition}
Observe that~\eqref{eq:fancy_symmetric_graph_funcs} reduces to~\eqref{eq:fancy_symmetric_funcs} when $H=\diag(\lambda)$ and $X$ both consist only of self-loops. 
In fact, this entire proposition and its proof reduce to Proposition~\ref{prop:rep_thm_symmetric_funcs} in this way.
\begin{proof}
    Observe that if $X\in\mbb S^n$, then $L_{k,n}X\overset{d}{=}X/P_{n,k}$ where $P_{n,k}$ is a uniformly random ordered partition of $[n]$ into $k$ parts. 
    Indeed, we have $P_{n,k}\overset{d}{=}P_{F_{k,n}}$ since ordered partitions of $[n]$ into $k$ parts are in bijection with maps $[n]\to[k]$, hence a uniformly random element of the former is induced by a uniformly random element of the latter. We then have $L_{k,n}X = \rho(F_{k,n})^\star X = X/P_{F_{k,n}}\overset{d}{=}X/P_{n,k}$ by definition. 
    Therefore, for any $k\geq |V(H)|$ we have
    \begin{equation*}\begin{aligned}
        p_n^{(k,[H])}(X)&= |\msf S_kH|\cdot \mbb E[(X/P_{n,k})^H] = \frac{(k)_{|V(H)|}}{R_{[H],[H]}}\cdot\frac{1}{k^n}\sum_{f\colon[n]\to[k]}\prod_{i_1\leq i_2}\left(\sum_{\substack{(j_1,j_2)\in f^{-1}(i_1,i_2)\\ j_1\leq j_2}}X_{j_1,j_2}\right)^{H_{i_1,i_2}},
    \end{aligned}\end{equation*}
    since $|\msf S_kH|=(k)_{|V(H)|}/|\mathrm{Aut}(H)|$ and $|\mathrm{Aut}(H)|=R_{[H],[H]}$. Here $f^{-1}(i_1,i_2)=\{(j_1,j_2)\in[n]^2:(f(j_1),f(j_2))=(i_1,i_2) \textrm{ or } (f(j_2),f(j_1))=(i_1,i_2)\}$.
    Now observe that
    \begin{equation*}
        \prod_{i_1\leq i_2}\left(\sum_{\substack{(j_1,j_2)\in f^{-1}(i_1,i_2)\\ j_1\leq j_2}}X_{j_1,j_2}\right)^{H_{i_1,i_2}} = \prod_{i_1\leq i_2}\sum_{\substack{G^{(i_1,i_2)}\in \NN^{f^{-1}(i_1,i_2)}\\|G^{(i_1,i_2)}|=H_{i_1,i_2}}}\frac{H_{i_1,i_2}!}{G^{(i_1,i_2)}!}X^{G^{(i_1,i_2)}} = \sum_{\substack{G\in\mbb S^n(\NN)\\ G/P_f=H}}\frac{H!}{G!}X^G,
    \end{equation*}
    where $\mbb N^{f^{-1}(i_1,i_2)}$ consists of symmetric matrices indexed by $f^{-1}(i_1,i_2)$ with entries in $\mbb N$. 
    The first equality above is just the multinomial theorem. To see the second equality, note the bijection between $G\in\mbb N^{n\times n}\cap\mbb S^n$ such that $G/P_f=H$ and sequences $(G^{(i_1,i_2)}\in\mbb N^{f^{-1}(i_1,i_2)})_{i_1\leq i_2}$ with $|G^{(i_1,i_2)}|=H_{i_1,i_2}$ given by 
    $G\mapsto G^{(i_1,i_2)}_{j_1,j_2}=G_{j_1,j_2}$ and $(G^{(i_1,i_2)})\mapsto G_{j_1,j_2}=G^{(i_1,i_2)}_{j_1,j_2}$ if $(f(j_1),f(j_2))=(i_1,i_2)$ or $(i_2,i_1)$.
    Under this bijection, $X^G=\prod_{j_1\leq j_2}X_{j_1,j_2}^{G_{j_1,j_2}}=\prod_{i_1\leq i_2}X^{G^{(i_1,i_2)}}$, which proves the second equality above.
    Next, 
    \begin{equation*}\begin{aligned}
        \sum_{f\colon[n]\to[k]}\sum_{\substack{G\in\mbb S^n(\NN)\\ G/P_f=H}}\frac{H!}{G!}X^G &= \sum_{\substack{G\in\mbb S^n(\NN) \\ [G]\leq_R [H]}}k^{n-|V(G)|}R_{[G],[H]}\frac{H!}{G!}X^G=\sum_{[G]\leq_R[H]}k^{n-|V(G)|}R_{[G],[H]}\frac{H!}{G!}m^{[G]}_{n}(X),
    \end{aligned}\end{equation*}
    where the first equality follows since for each $G\leq_RH$ there are $R_{[G],[H]}k^{n-|V(G)|}$ maps $f\colon[n]\to[k]$ satisfying $G/P_f=H$ (note that there are $n-|V(G)|$ isolated vertices above), and the second follows since $|V(G)|$, $R_{[G],[H]}$ and $G!$ only depend on $G$ via $[G]$.
    Combining the above three displayed equations we get~\eqref{eq:fancy_symmetric_graph_funcs}.

    For the last claim, it is clear that the monomial sums $\{(m_n^{[H]})_n\}_{|H|\leq d}$ form a basis for freely-described polynomials, since $\{m_n^{[H]}\}_{|H|\leq d}$ is a basis for invariant polynomials of degree at most $d$ over $\mbb S^n$ for each $n\geq 2d$. 
    Further, we have $\mathrm{inj}(H;G)=|\mathrm{Aut}(H)|m_n^{[H]}(G)$ for all $H,G$, hence $\{(\mathrm{inj}(H;\cdot)|_{\mbb S^n})_n\}_{|H|\leq d}$ also gives a basis. 
    The relation~\eqref{eq:R_graph_trans_matrix} shows that the matrix expressing homomorphism numbers in terms of injective numbers is upper triangular in refinement order and has nonzero entries on the diagonal $R_{[H],[H]}/|\mathrm{Aut}(H)| = 1$. It is therefore invertible so we conclude that $\{(\mathrm{hom}(H;\cdot)|_{\mbb S^n})_n\}_{|H|\leq d}$ is also a basis. 
    Finally, the collection $\{(p_n^{(k,[H])})_n\}_{|H|\leq d}$ forms a basis by Propositions~\ref{prop:free_sym_and_free_descr_isom} and~\ref{prop:calc_for_stability}.
\end{proof}
We are not aware of prior studies on the refinement partial order on graphs, or of graph functions of the form~\eqref{eq:fancy_symmetric_graph_funcs}. In particular, it would be of interest to efficiently compute the integers $R_{[G],[H]}$ given two graphs, in analogy to the code~\cite{transition_mats} for the diagonal case.

\paragraph{Dual cost.} For each $k\geq 2d$, expand the freely-described cost as $p_n=\sum_{[H]} c^{(k,[H])}p_n^{(k,[H])}$ in the $k$'th basis from Proposition~\ref{prop:graph_numbers}. If $p_n$ is given in terms of (injective) homomorphism numbers as in~\eqref{eq:inj_number_ineq}, we can obtain $c^{(k,[H])}$ by first changing basis to $\{(m_n^{[H]})\}$ using~\eqref{eq:R_graph_trans_matrix} and~\eqref{eq:m_to_inj_trans}, and then changing to the bases $\{(p_n^{(k,[H])})\}$ using~\eqref{eq:fancy_symmetric_graph_funcs}.  The $k$'th dual cost is then given by $q_k=\sum_{[H]}c^{(k,[H])}m_k^{[H]}$, and the lower bounds on~\eqref{eq:inj_number_ineq} are given by $\ell_n=\inf_{\substack{X\in\mbb S^n\\ \|X\|_1\leq 1}}q_n(X)$. These satisfy $\ell_n\leq \ell_{nk}$ for all $n,k$ and $u_{nk}-\ell_{nk}\lesssim 1/n$ for fixed $k$ and growing $n$.  Similar to Section~\ref{sec:symmetric_funcs}, performing the preceding steps for some $k$ in principle yields the dual costs via symmetrization for all $q_{nk}$; however, the sequence of monomial sums $(m_k^{[H]})$ is not freely-symmetrized over $\mathrm{Sym}^2\mscr D$, so it is not in general true that $m_{nk}^{[H]}=\mathrm{sym}_{nk}^{\mathrm{Sym}^2\mscr D}m_k^{[H]}$.  As discussed at the beginning of Section~\ref{sec:construct_bds}, we have an explicit characterization of $(m_k^{[H]})$ as a function of $k$ in \eqref{eq:m_to_inj_trans}, so it is instead more convenient to directly perform the preceding three steps starting from the monomial sums $m_{nk}^{[H]}$ to obtain $q_{nk}$.


\paragraph{$\FS$-compatible sets.} The sequence of constraint sets $(\Omega_n\subseteq\mbb S^n)$ to which our framework applies consists of $\msf S_n$-invariant sets satisfying $\zeta_{n+1,n}X\zeta_{n+1,n}^\star \in\Omega_{n+1}$ if $X\in\Omega_n$ and $(\delta_{nk,n}^{\otimes 2})^\star X = \sum_{1\leq i\leq j\leq nk}X_{i,j}E_{\lceil i/k\rceil, \lceil j/k\rceil}\in\Omega_n$ if $X\in\Omega_{nk}$, where $\zeta$ and $\delta$ are the zero-padding and duplication embeddings from Section~\ref{sec:consist_seqs} and the adjoint is taken with respect to the inner product $\langle X,Y\rangle=\sum_{i\leq j}X_{i,j}Y_{i,j}$.
As in Section~\ref{sec:symmetric_funcs}, these assumptions are satisfied by the simplices $\{X\in\mbb S^n:X\geq0, \sum_{i\leq j}X_{i,j}=1\}$, but not (for example) by their spectral analogues $\{X\in\mbb S^n: X\succeq0, \mathrm{Tr}(X)=1\}$.

\begin{example}
    Let $G_1=K_2\sqcup K_2\sqcup K_2\leq_R G_2=K_2\sqcup P_3\leq_R G_3=P_4\leq_R G_4=K_3$ be the chain in Figure~\ref{fig:triangle_refinements}. Then
    \begin{equation*}
        R_{[G_1],[G_{4}]}=48,\quad R_{[G_2],[G_{4}]}=12,\quad R_{[G_3],[G_4]}=6,\quad R_{[G_4],[G_4]}=6.
    \end{equation*}
    Since these are all simple graphs, we also have $G_i!=1$ for all $i$. 
    Therefore, we have
    \begin{equation*}
        p_{n}^{(k,[K_3])} = \frac{k(k-1)(k-2)}{k^3}\left(m^{[K_3]}_{n} + \frac{1}{k} m^{[P_4]}_{n} + \frac{2}{k^2}m^{[K_2\sqcup P_3]}_{n} + \frac{8}{k^3}m^{[K_2\sqcup K_2\sqcup K_2]}_{n}\right).
    \end{equation*}
\end{example}

\begin{example}\label{ex:graph_nums}
    In Example~\ref{ex:graph_nums_intro} from Section~\ref{sec:intro}, we consider the freely-described problem~\eqref{eq:un_graph_nums_intro}, where we seek the largest $\gamma\in\RR$ satisfying $\mathrm{inj}(P_3;X) - \mathrm{inj}(K_2\sqcup K_2; X)\geq \gamma \mathrm{hom}(K_2;X)^2$ for all $X\geq0$ of all sizes. Applying Proposition~\ref{prop:graph_numbers}, we obtain the lower bounds~\eqref{eq:ln_graph_nums_intro} for $n\geq 4$.
    We numerically compute upper bounds on $u_n$ and intervals containing $\ell_n$ for $n\in[4,9]$, with the results shown in Table~\ref{tab:graph_inj_numerics}.  The upper bounds on $u_n,\ell_n$ are computed using a nonconvex solver (Matlab's \texttt{fmincon}), while the lower bounds on $\ell_n$ are computed using a degree-$4$ sums-of-squares relaxation.

\end{example}





\section{Conclusions and Future Work}\label{sec:concs}
We have studied sequences of POPs of growing dimension and presented a framework to produce lower bounds on their limiting optimal values in terms of finite-dimensional POPs. Our framework is based on reformulating the limiting optimal value as an optimization problem over measures projecting onto each other. To produce convergent sequences of finite-dimensional relaxations for such problems, we prove new generalizations of de Finetti's theorem approximating such sequences of measures by  particularly simple ones.
Our proofs are based on a new perspective on de Finetti theorems via representations of categories, which also underpins results from representation stability that we use throughout the paper. 
Finally, we demonstrated our bounds on a range of numerical experiments.
There are a number of questions suggested by our work.
\begin{enumerate}[font=\textbf, align=left]

    
    \item[(Beyond permutations)] Our generalizations of de Finetti's theorem relied on $\FS$-representations, and hence only apply to representations of the permutation groups $\msf S_n$. How can we further generalize these theorems to sequences of representations of other groups, in particular, the other classical Weyl groups and Lie groups? 
    
    \item[(Finite step convergence)] For which freely-described problems do we have $\ell_m=u_\infty$ for some finite $m\in\NN$, i.e., when does our hierarchy of lower bounds converge in finitely many steps? When does it converge in one step? From our numerical experiments, it seems that this is the case for generic symmetric functions, for instance.
    Finite-step convergence would imply a finite-dimensional reformulation for $u_\infty$. 
    
    \item[(Semialgebraicity of limiting cones)] When is the cone of nonnegative freely-described polynomials from Section~\ref{sec:any_dim_nonneg_cones} semialgebraic? This would imply the decidability of checking $u_\infty\geq0$.

    \item[(Any-dimensional Positivstellens\"atze)] Are there any-dimensional analogs of fixed-dimensional Positivstellens\"atze, such as those due to Krivine, Schm\"udgen, and Putinar? Can we obtain such results for all freely-described costs $(p_n)$ or just a subset of them? How should the sequence of constraints $(\Omega_n)$ be described for such results to hold? 

    \item[(Bases for symmetric functions)] Classical bases for symmetric functions and the relations between them have rich connections to combinatorics and representation theory~\cite{Stanley_Fomin_1999,macdonald1998symmetric}. It would be interesting to develop such connections for the de Finetti maps acting on these symmetric functions, as well as the bases $(p_n^{(k,\lambda)})$ and $(p_n^{(k,[H])})$ over vectors and graphs they produce (see Propositions~\ref{prop:rep_thm_symmetric_funcs} and~\ref{prop:graph_numbers}).
\end{enumerate}

\bibliographystyle{unsrt}
\bibliography{free_cvx_refs}

@article{convergent_seqs1,
title = {Convergent sequences of dense graphs {I}: Subgraph frequencies, metric properties and testing},
journal = {Advances in Mathematics},
volume = {219},
number = {6},
pages = {1801-1851},
year = {2008},
issn = {0001-8708},
doi = {https://doi.org/10.1016/j.aim.2008.07.008},
url = {https://www.sciencedirect.com/science/article/pii/S0001870808002053},
author = {C. Borgs and J.T. Chayes and L. Lovász and V.T. Sós and K. Vesztergombi}
}

@book{lovasz2012large,
  author    = {Lov{\'a}sz, L{\'a}szl{\'o}},
  title     = {Large networks and graph limits},
  series    = {American Mathematical Society Colloquium Publications},
  volume    = {60},
  publisher = {American Mathematical Society},
  address   = {Providence, RI},
  year      = {2012},
  doi       = {10.1090/coll/060}
}

@article{CHURCH2013250,
title = {Representation theory and homological stability},
journal = {Advances in Mathematics},
volume = {245},
pages = {250-314},
year = {2013},
issn = {0001-8708},
doi = {https://doi.org/10.1016/j.aim.2013.06.016},
url = {https://www.sciencedirect.com/science/article/pii/S0001870813002259},
author = {Thomas Church and Benson Farb},
}

@article{FImods,
author = {Thomas Church and Jordan S. Ellenberg and Benson Farb},
title = {{FI-modules and stability for representations of symmetric groups}},
volume = {164},
journal = {Duke Mathematical Journal},
number = {9},
publisher = {Duke University Press},
pages = {1833--1910},
year = {2015},
doi = {10.1215/00127094-3120274},
URL = {https://doi.org/10.1215/00127094-3120274}
}

@article{WILSON2014269,
title = {{FI\textsubscript{$\mathcal{W}$}}-modules and stability criteria for representations of classical {Weyl} groups},
journal = {Journal of Algebra},
volume = {420},
pages = {269-332},
year = {2014},
issn = {0021-8693},
doi = {https://doi.org/10.1016/j.jalgebra.2014.08.010},
url = {https://www.sciencedirect.com/science/article/pii/S0021869314004505},
author = {Jennifer C.H. Wilson},
}

@article{GADISH2017450,
title = {Categories of {FI} type: A unified approach to generalizing representation stability and character polynomials},
journal = {Journal of Algebra},
volume = {480},
pages = {450-486},
year = {2017},
issn = {0021-8693},
doi = {https://doi.org/10.1016/j.jalgebra.2017.03.010},
url = {https://www.sciencedirect.com/science/article/pii/S0021869317301849},
author = {Nir Gadish},
}

@article{FI_hom,
author = {Thomas Church and Jordan Ellenberg},
title = {{Homology of FI-modules}},
volume = {21},
journal = {Geometry \& Topology},
number = {4},
publisher = {MSP},
pages = {2373--2418},
year = {2017},
doi = {10.2140/gt.2017.21.2373},
URL = {https://doi.org/10.2140/gt.2017.21.2373}
}

@article{raymond2018symmetric,
  title={Symmetric sums of squares over $k$-subset hypercubes},
  author={Raymond, Annie and Saunderson, James and Singh, Mohit and Thomas, Rekha R.},
  journal={Mathematical Programming},
  volume={167},
  number={2},
  pages={315--354},
  year={2018},
  publisher={Springer}
}

@article{raymond2018symmetry,
  title={Symmetry in {Tur{\'a}n} sums of squares polynomials from flag algebras},
  author={Raymond, Annie and Singh, Mohit and Thomas, Rekha R.},
  journal={Algebraic Combinatorics},
  volume={1},
  number={2},
  pages={249--274},
  year={2018}
}

@article{riener2013exploiting,
  title={Exploiting symmetries in {SDP}-relaxations for polynomial optimization},
  author={Riener, Cordian and Theobald, Thorsten and Andr{\'e}n, Lina Jansson and Lasserre, Jean B.},
  journal={Mathematics of Operations Research},
  volume={38},
  number={1},
  pages={122--141},
  year={2013},
  publisher={INFORMS}
}

@article{debus2020reflection,
title = {Reflection groups and cones of sums of squares},
journal = {Journal of Symbolic Computation},
volume = {119},
pages = {112-144},
year = {2023},
issn = {0747-7171},
doi = {https://doi.org/10.1016/j.jsc.2023.03.001},
url = {https://www.sciencedirect.com/science/article/pii/S0747717123000263},
author = {Sebastian Debus and Cordian Riener},
}

@article{moustrou2021symmetry,
author = {Moustrou, Philippe and Naumann, Helen and Riener, Cordian and Theobald, Thorsten and Verdure, Hugues},
title = {Symmetry Reduction in {AM/GM}-Based Optimization},
journal = {SIAM Journal on Optimization},
volume = {32},
number = {2},
pages = {765-785},
year = {2022},
doi = {10.1137/21M1405691},
URL = {https://doi.org/10.1137/21M1405691},
eprint = {https://doi.org/10.1137/21M1405691},
}

@article{sam2017grobner,
  title={Gr{\"o}bner methods for representations of combinatorial categories},
  author={Sam, Steven V. and Snowden, Andrew},
  journal={Journal of the American Mathematical Society},
  volume={30},
  number={1},
  pages={159--203},
  year={2017}
}

@article{sam_snowden_2015, 
title={STABILITY PATTERNS IN REPRESENTATION THEORY}, 
volume={3}, 
DOI={10.1017/fms.2015.10}, 
journal={Forum of Mathematics, Sigma}, 
publisher={Cambridge University Press}, 
author={Sam, Steven V. and Snowden, Andrew}, 
year={2015}, 
pages={e11}
}

@book{serre1977linear,
  title={Linear representations of finite groups},
  author={Serre, Jean-Pierre},
  volume={42},
  year={1977},
  publisher={Springer}
}

@book{fulton2013representation,
  title={Representation theory: a first course},
  author={Fulton, William and Harris, Joe},
  volume={129},
  year={2013},
  publisher={Springer}
}

@article{NPA,
author = {Pironio, S. and Navascu\'{e}s, M. and Ac\'{\i}n, A.},
title = {Convergent Relaxations of Polynomial Optimization Problems with Noncommuting Variables},
journal = {SIAM Journal on Optimization},
volume = {20},
number = {5},
pages = {2157-2180},
year = {2010},
doi = {10.1137/090760155},
URL = {https://doi.org/10.1137/090760155},
eprint = {https://doi.org/10.1137/090760155},
}

@incollection{SOS_chapter,
  author    = {Parrilo, Pablo A.},
  title     = {Polynomial optimization, sums of squares, and applications},
  editor    = {Blekherman, Grigoriy and Parrilo, Pablo A. and Thomas, Rekha R.},
  booktitle = {Semidefinite Optimization and Convex Algebraic Geometry},
  series    = {MOS--SIAM Series on Optimization},
  volume    = {13},
  chapter   = {3},
  pages     = {47--157},
  publisher = {Society for Industrial and Applied Mathematics},
  year      = {2012},
  doi       = {10.1137/1.9781611972290.ch3}
}

@article{acevedo2023wonderful,
  author  = {Acevedo, Jose and Blekherman, Grigoriy and Debus, Sebastian
             and Riener, Cordian},
  title   = {The wonderful geometry of the {Vandermonde} map},
  journal = {Foundations of Computational Mathematics},
  volume  = {26},
  number  = {4},
  pages   = {2005--2051},
  year    = {2026},
  doi     = {10.1007/s10208-025-09718-6}
}

@article{razborov_2007, 
title={Flag algebras}, 
volume={72}, 
DOI={10.2178/jsl/1203350785}, 
number={4}, 
journal={The Journal of Symbolic Logic}, 
publisher={Cambridge University Press}, 
author={Razborov, Alexander A.}, 
year={2007}, 
pages={1239–1282}
}

@article{lovasz2012random,
author = {Lovász, László and Szegedy, Balázs},
title = {Random graphons and a weak {P}ositivstellensatz for graphs},
journal = {Journal of Graph Theory},
volume = {70},
number = {2},
pages = {214-225},
doi = {https://doi.org/10.1002/jgt.20611},
url = {https://onlinelibrary.wiley.com/doi/abs/10.1002/jgt.20611},
eprint = {https://onlinelibrary.wiley.com/doi/pdf/10.1002/jgt.20611},
year = {2012}
}

@article{levin2023free,
  title={Dimension-free descriptions of convex sets},
  author={Levin, Eitan and Chandrasekaran, Venkat},
  journal={arXiv preprint arXiv:2307.04230},
  year={2025}
}

@article{polak2022symmetry,
author = {Polak, Sven C.},
title = {Symmetry Reduction to Optimize a Graph-based Polynomial From Queueing Theory},
journal = {SIAM Journal on Applied Algebra and Geometry},
volume = {6},
number = {2},
pages = {243-266},
year = {2022},
doi = {10.1137/21M1413298},
URL = {https://doi.org/10.1137/21M1413298},
eprint = {https://doi.org/10.1137/21M1413298},
}

@INPROCEEDINGS{parrilo2006games,
  author={Parrilo, Pablo A.},
  booktitle={Proceedings of the 45th IEEE Conference on Decision and Control}, 
  title={Polynomial games and sum of squares optimization}, 
  year={2006},
  volume={},
  number={},
  pages={2855-2860},
  doi={10.1109/CDC.2006.377261}}

@phdthesis{brosch_thesis,
title = "Symmetry reduction in convex optimization with applications in combinatorics",
author = "Daniel Brosch",
year = "2022",
doi = "10.26116/12v4-k024",
language = "English",
series = "CentER Dissertation Series",
publisher = "CentER, Center for Economic Research",
school = "Tilburg University",

}

@article{parrilo2003semidefinite,
  title={Semidefinite programming relaxations for semialgebraic problems},
  author={Parrilo, Pablo A. },
  journal={Mathematical Programming},
  volume={96},
  pages={293--320},
  year={2003},
  publisher={Springer}
}

@article{lasserre2001global,
author = {Lasserre, Jean B.},
title = {Global Optimization with Polynomials and the Problem of Moments},
journal = {SIAM Journal on Optimization},
volume = {11},
number = {3},
pages = {796-817},
year = {2001},
doi = {10.1137/S1052623400366802},
URL = {https://doi.org/10.1137/S1052623400366802},
eprint = {https://doi.org/10.1137/S1052623400366802},
}

@book{blekherman2012semidefinite,
  editor    = {Blekherman, Grigoriy and Parrilo, Pablo A. and Thomas, Rekha R.},
  title     = {Semidefinite Optimization and Convex Algebraic Geometry},
  series    = {MOS--SIAM Series on Optimization},
  volume    = {13},
  publisher = {Society for Industrial and Applied Mathematics},
  year      = {2012},
  doi       = {10.1137/1.9781611972290}
}

@article{sampling_tv_dist,
author = {Stam, A. J.},
title = {Distance between sampling with and without replacement},
journal = {Statistica Neerlandica},
volume = {32},
number = {2},
pages = {81-91},
doi = {https://doi.org/10.1111/j.1467-9574.1978.tb01387.x},
url = {https://onlinelibrary.wiley.com/doi/abs/10.1111/j.1467-9574.1978.tb01387.x},
eprint = {https://onlinelibrary.wiley.com/doi/pdf/10.1111/j.1467-9574.1978.tb01387.x},
year = {1978}
}

@article{acevedo2024power,
  title={Power mean inequalities and sums of squares},
  author={Acevedo, Jose and Blekherman, Grigoriy},
  journal={Discrete \& Computational Geometry},
  volume={73},
  number={4},
  pages={1097--1143},
  year={2025},
  publisher={Springer}
}

@misc{cardaliaguet2010notes,
  author       = {Cardaliaguet, Pierre},
  title        = {Notes on mean field games},
  howpublished = {Lecture notes},
  year         = {2013},
  note         = {From P.-L. Lions' lectures at Coll{\`e}ge de France;
                  revised September 27, 2013},
  url          = {https://www.ceremade.dauphine.fr/~cardaliaguet/MFG20130420.pdf}
}

@article{Blekherman_Raymond_Wei_2024, 
title={Undecidability of polynomial inequalities in weighted graph homomorphism densities}, 
volume={12}, 
DOI={10.1017/fms.2024.19}, 
journal={Forum of Mathematics, Sigma}, 
author={Blekherman, Grigoriy and Raymond, Annie and Wei, Fan}, 
year={2024}, 
pages={e40}
}

@article{hatami2011undecidability,
  title={Undecidability of linear inequalities in graph homomorphism densities},
  author={Hatami, Hamed and Norine, Serguei},
  journal={Journal of the American Mathematical Society},
  volume={24},
  number={2},
  pages={547--565},
  year={2011}
}

@article{acevedo2024symmetric,
  author  = {Acevedo, Jose and Blekherman, Grigoriy and Debus, Sebastian
             and Riener, Cordian},
  title   = {Symmetric nonnegative functions, the tropical {Vandermonde} cell
             and superdominance of power sums},
  journal = {Collectanea Mathematica},
  volume  = {77},
  number  = {2},
  pages   = {555--600},
  year    = {2026},
  doi     = {10.1007/s13348-025-00481-z}
}

@book{Stanley_Fomin_1999,
  author    = {Stanley, Richard P.},
  title     = {Enumerative Combinatorics, Volume 2},
  series    = {Cambridge Studies in Advanced Mathematics},
  number    = {62},
  publisher = {Cambridge University Press},
  year      = {1999},
  note      = {With an appendix by Sergey Fomin},
  doi       = {10.1017/CBO9780511609589}
}

@article{diaconis_freedman,
author = {P. Diaconis and D. Freedman},
title = {{Finite Exchangeable Sequences}},
volume = {8},
journal = {The Annals of Probability},
number = {4},
publisher = {Institute of Mathematical Statistics},
pages = {745--764},
year = {1980},
doi = {10.1214/aop/1176994663},
URL = {https://doi.org/10.1214/aop/1176994663}
}

@misc{transition_mats,
   author={Raymond Kan}, 
   title={Transition Matrices between Symmetric Polynomials}, 
   year={2023},
   howpublished = {(\url{https://www.mathworks.com/matlabcentral/fileexchange/136299-transition-matrices-between-symmetric-polynomials}). MATLAB Central File Exchange},
   note = "[Retrieved May 18, 2025]"
}

@article{burr_rosta,
author = {Burr, Stefan A. and Rosta, Vera},
title = {On the {R}amsey multiplicities of graphs—problems and recent results},
journal = {Journal of Graph Theory},
volume = {4},
number = {4},
pages = {347-361},
doi = {https://doi.org/10.1002/jgt.3190040403},
url = {https://onlinelibrary.wiley.com/doi/abs/10.1002/jgt.3190040403},
eprint = {https://onlinelibrary.wiley.com/doi/pdf/10.1002/jgt.3190040403},
year = {1980}
}

@article{Erdös1969,
  author  = {Erd{\H{o}}s, P{\'a}l},
  title   = {On the number of complete subgraphs and circuits contained in graphs},
  journal = {{\v{C}}asopis pro p{\v{e}}stov{\'a}n{\'i} matematiky},
  volume  = {94},
  number  = {3},
  pages   = {290--296},
  year    = {1969},
  doi     = {10.21136/CPM.1969.108598}
}

@article{sidorenko1989cycles,
  title={Cycles in graphs and functional inequalities},
  author={Sidorenko, Alexander F.},
  journal={Matematicheskie Zametki},
  volume={46},
  number={5},
  pages={72--79},
  year={1989},
  publisher={Russian Academy of Sciences, Steklov Mathematical Institute of Russian}
}

@article{parczyk2024new,
  title={New {R}amsey multiplicity bounds and search heuristics},
  author={Parczyk, Olaf and Pokutta, Sebastian and Spiegel, Christoph and Szab{\'o}, Tibor},
  journal={Foundations of Computational Mathematics},
  volume={25},
  number={5},
  pages={1777--1814},
  year={2025},
  publisher={Springer}
}

@article{PROCESI1978219,
title = {Positive symmetric functions},
journal = {Advances in Mathematics},
volume = {29},
number = {2},
pages = {219-225},
year = {1978},
issn = {0001-8708},
doi = {https://doi.org/10.1016/0001-8708(78)90011-7},
url = {https://www.sciencedirect.com/science/article/pii/0001870878900117},
author = {Claudio Procesi}
}

@article{Hunter_1977, 
title={The positive-definiteness of the complete symmetric functions of even order},
volume={82}, 
DOI={10.1017/S030500410005386X}, 
number={2},
journal={Mathematical Proceedings of the Cambridge Philosophical Society}, 
author={Hunter, D. B.}, 
year={1977}, 
pages={255–258}
}

@article{khare2021sign,
  title={On the sign patterns of entrywise positivity preservers in fixed dimension},
  author={Khare, Apoorva and Tao, Terence},
  journal={American Journal of Mathematics},
  volume={143},
  number={6},
  pages={1863--1929},
  year={2021},
  publisher={Johns Hopkins University Press}
}

@article{Chavez_Garcia_Hurley_2023, 
title={Norms on complex matrices induced by random vectors}, 
volume={66}, 
DOI={10.4153/S0008439522000741}, 
number={3}, 
journal={Canadian Mathematical Bulletin}, 
author={Ch{\'a}vez, {\'A}ngel and Garcia, Stephan Ramon and Hurley, Jackson}, 
year={2023}, 
pages={808–826}
}

@book{macdonald1998symmetric,
  title={Symmetric functions and {H}all polynomials},
  author={Macdonald, Ian G.},
  year={1998},
  publisher={Oxford University Press}
}

@article{austin_survey,
author = {Tim Austin},
title = {{On exchangeable random variables and the statistics of large graphs and hypergraphs}},
volume = {5},
journal = {Probability Surveys},
publisher = {Institute of Mathematical Statistics and Bernoulli Society},
pages = {80--145},
year = {2008},
doi = {10.1214/08-PS124},
URL = {https://doi.org/10.1214/08-PS124}
}

@inproceedings{de1929funzione,
  author    = {de Finetti, Bruno},
  title     = {Funzione caratteristica di un fenomeno aleatorio},
  booktitle = {Atti del Congresso Internazionale dei Matematici
               (Bologna, 3--10 settembre 1928)},
  volume    = {6},
  pages     = {179--190},
  publisher = {Nicola Zanichelli},
  address   = {Bologna},
  year      = {1932},
  url       = {https://www.mathunion.org/fileadmin/ICM/Proceedings/ICM1928.6/ICM1928.6.ocr.pdf}
}

@article{hewitt_savage,
 ISSN = {00029947, 10886850},
 URL = {http://www.jstor.org/stable/1992999},
 author = {Edwin Hewitt and Leonard J. Savage},
 journal = {Transactions of the American Mathematical Society},
 number = {2},
 pages = {470--501},
 publisher = {American Mathematical Society},
 title = {Symmetric Measures on {C}artesian Products},
 urldate = {2025-05-29},
 volume = {80},
 year = {1955}
}

@article{de1937prevision,
  author  = {de Finetti, Bruno},
  title   = {La pr{\'e}vision: ses lois logiques, ses sources subjectives},
  journal = {Annales de l'institut Henri Poincar{\'e}},
  volume  = {7},
  number  = {1},
  pages   = {1--68},
  year    = {1937},
  url     = {https://www.numdam.org/item/AIHP_1937__7_1_1_0/}
}

@article{dynkin1953classes,
  title={Classes of equivalent random quantities},
  author={Dynkin, Evgenii Borisovich},
  journal={Uspekhi Matematicheskikh Nauk},
  volume={8},
  number={2},
  pages={125--130},
  year={1953},
  publisher={Russian Academy of Sciences, Steklov Mathematical Institute of Russian~…}
}

@article{aldous1981representations,
  title={Representations for partially exchangeable arrays of random variables},
  author={Aldous, David J.},
  journal={Journal of Multivariate Analysis},
  volume={11},
  number={4},
  pages={581--598},
  year={1981},
  publisher={Elsevier}
}

@book{kallenberg2005probabilistic,
  author    = {Kallenberg, Olav},
  title     = {Probabilistic Symmetries and Invariance Principles},
  series    = {Probability and Its Applications},
  publisher = {Springer},
  address   = {New York},
  year      = {2005},
  doi       = {10.1007/0-387-28861-9}
}

@article{hoover1979relations,
  title={Relations on Probability Spaces and Arrays of Random Variables},
  author={Hoover, Douglas N.},
  journal={Institute for Advanced Study},
  year={1979}
}

@article{diaconis2007graph,
  author  = {Diaconis, Persi and Janson, Svante},
  title   = {Graph limits and exchangeable random graphs},
  journal = {Rendiconti di Matematica e delle sue Applicazioni. Serie VII},
  volume  = {28},
  number  = {1},
  pages   = {33--61},
  year    = {2008},
  url     = {https://www2.math.uu.se/~svantejs/papers/sj209.pdf}
}

@article{Blyth01081980,
author = {Colin R. Blyth},
title = {Expected Absolute Error of the Usual Estimator of the Binomial Parameter},
journal = {The American Statistician},
volume = {34},
number = {3},
pages = {155--157},
year = {1980},
publisher = {ASA Website},
doi = {10.1080/00031305.1980.10483022},
URL = {https://www.tandfonline.com/doi/abs/10.1080/00031305.1980.10483022},
eprint = {https://www.tandfonline.com/doi/pdf/10.1080/00031305.1980.10483022}

}

@article{klep2025sums,
  title={Sums of squares certificates for polynomial moment inequalities},
  author={Klep, Igor and Magron, Victor and Vol{\v{c}}i{\v{c}}, Jurij},
  journal={Foundations of Computational Mathematics},
  volume={26},
  number={3},
  pages={1381--1423},
  year={2026},
  publisher={Springer}
}

@article{murray2021newton,
  title={Newton polytopes and relative entropy optimization},
  author={Murray, Riley and Chandrasekaran, Venkat and Wierman, Adam},
  journal={Foundations of Computational Mathematics},
  volume={21},
  number={6},
  pages={1703--1737},
  year={2021},
  publisher={Springer}
}

@article{LOVASZ2006933,
title = {Limits of dense graph sequences},
journal = {Journal of Combinatorial Theory, Series B},
volume = {96},
number = {6},
pages = {933-957},
year = {2006},
issn = {0095-8956},
doi = {https://doi.org/10.1016/j.jctb.2006.05.002},
url = {https://www.sciencedirect.com/science/article/pii/S0095895606000517},
author = {László Lovász and Balázs Szegedy},
}

@article{ligthart2023inflation,
    author = {Ligthart, Laurens T. and Gross, David},
    title = {The inflation hierarchy and the polarization hierarchy are complete for the quantum bilocal scenario},
    journal = {Journal of Mathematical Physics},
    volume = {64},
    number = {7},
    pages = {072201},
    year = {2023},
    issn = {0022-2488},
    doi = {10.1063/5.0143792},
    url = {https://doi.org/10.1063/5.0143792},
    eprint = {https://pubs.aip.org/aip/jmp/article-pdf/doi/10.1063/5.0143792/18037636/072201_1_5.0143792.pdf},
}

@article{ligthart2023convergent,
  title={A convergent inflation hierarchy for quantum causal structures},
  author={Ligthart, Laurens T. and Gachechiladze, Mariami and Gross, David},
  journal={Communications in Mathematical Physics},
  volume={401},
  number={3},
  pages={2673--2714},
  year={2023},
  publisher={Springer}
}

@article{renou_xu,
    author = {Renou, Marc-Olivier and Xu, Xiangling and Ligthart, Laurens T.},
    title = {Two convergent {NPA}-like hierarchies for the quantum bilocal scenario},
    journal = {Journal of Mathematical Physics},
    volume = {67},
    number = {1},
    pages = {012203},
    year = {2026},
    issn = {0022-2488},
    doi = {10.1063/5.0211008},
    url = {https://doi.org/10.1063/5.0211008},
    eprint = {https://pubs.aip.org/aip/jmp/article-pdf/doi/10.1063/5.0211008/20865530/012203_1_5.0211008.pdf},
}

@article{navascues2008convergent,
  title={A convergent hierarchy of semidefinite programs characterizing the set of quantum correlations},
  author={Navascu{\'e}s, Miguel and Pironio, Stefano and Ac{\'\i}n, Antonio},
  journal={New Journal of Physics},
  volume={10},
  number={7},
  pages={073013},
  year={2008}
}

@article{klep2024state,
  title={State polynomials: positivity, optimization and nonlinear {B}ell inequalities},
  author={Klep, Igor and Magron, Victor and Vol{\v{c}}i{\v{c}}, Jurij and Wang, Jie},
  journal={Mathematical Programming},
  volume={207},
  number={1},
  pages={645--691},
  year={2024},
  publisher={Springer}
}

@article{burgdorf2013tracial,
  author  = {Burgdorf, Sabine and Cafuta, Kristijan and Klep, Igor and Povh, Janez},
  title   = {The tracial moment problem and trace-optimization of polynomials},
  journal = {Mathematical Programming},
  volume  = {137},
  number  = {1--2},
  pages   = {557--578},
  year    = {2013},
  doi     = {10.1007/s10107-011-0505-8}
}

@incollection{huber2024positivity,
  author    = {Huber, Felix and Magron, Victor and Vol{\v{c}}i{\v{c}}, Jurij},
  title     = {Positivity of state, trace, and moment polynomials and applications in quantum information},
  editor    = {Alpay, Daniel and Colombo, Fabrizio and Sabadini, Irene},
  booktitle = {Operator Theory},
  edition   = {Second},
  pages     = {2773--2802},
  publisher = {Springer},
  address   = {Cham},
  year      = {2026},
  doi       = {10.1007/978-3-032-16356-1_88}
}

@article{levin2025limits,
  title={Limits of weighted graphs via random quotients},
  author={Levin, Eitan and Chandrasekaran, Venkat},
  journal={arXiv preprint arXiv:2512.23149},
  year={2025}
}

@article{kallenberg1990exchangeable,
  title={Exchangeable random measures in the plane},
  author={Kallenberg, Olav},
  journal={Journal of Theoretical Probability},
  volume={3},
  number={1},
  pages={81--136},
  year={1990},
  publisher={Springer}
}

@article{orbanz2011projective,
  author  = {Orbanz, Peter},
  title   = {Projective limit random probabilities on {Polish} spaces},
  journal = {Electronic Journal of Statistics},
  volume  = {5},
  pages   = {1354--1373},
  year    = {2011},
  doi     = {10.1214/11-EJS641}
}
\end{document}